\documentclass[a4paper,11pt]{article}

\usepackage[bookmarks=false,hyperfootnotes=false,colorlinks,
    linkcolor={red!60!black},
    citecolor={blue!50!black},
    urlcolor={blue!80!black}]{hyperref}

\usepackage{amsmath,amsthm,amssymb,tikz,enumitem,xcolor}

\usepackage[hang]{footmisc}
\usepackage[nosort,nocompress,noadjust]{cite}

\renewcommand{\eqref}[1]{\hyperref[#1]{(\ref{#1})}}

\newlist{enumlist}{enumerate}{2}
\setlist[enumlist,1]{labelindent=0cm,label=\arabic*.,ref=\arabic*,labelwidth=2.5ex,labelsep=0.5ex,leftmargin=3ex,align=left,topsep=0.5ex,itemsep=1ex,parsep=1ex}
\setlist[enumlist,2]{labelindent=0cm,label=\theenumlisti.\arabic*.,ref=\arabic*,labelwidth=5ex,labelsep=0.5ex,leftmargin=5.5ex,align=left,topsep=0.5ex,itemsep=1ex,parsep=1ex}

\newlist{itemlist}{itemize}{1}
\setlist[itemlist]{labelindent=0cm,label=$\bullet$,labelwidth=2.5ex,labelsep=0.5ex,leftmargin=3ex,align=left,topsep=0.5ex,itemsep=1ex,parsep=1ex}

\numberwithin{equation}{section}

{\theoremstyle{definition}\newtheorem{definition}{Definition}[section]
\newtheorem{remark}[definition]{Remark}
\newtheorem{example}[definition]{Example}
\newtheorem{notation}[definition]{Notation}}

\newtheorem{proposition}[definition]{Proposition}
\newtheorem{lemma}[definition]{Lemma}
\newtheorem{theorem}[definition]{Theorem}
\newtheorem{corollary}[definition]{Corollary}
\newtheorem{letterthm}{Theorem}

{\theoremstyle{definition}\newtheorem{letterdef}[letterthm]{Definition}}

\newcommand{\om}{\omega}
\newcommand{\si}{\sigma}
\newcommand{\actson}{\curvearrowright}
\newcommand{\R}{\mathbb{R}}
\newcommand{\N}{\mathbb{N}}
\newcommand{\Stab}{\operatorname{Stab}}

\newcommand{\cF}{\mathcal{F}}
\newcommand{\cP}{\mathcal{P}}
\newcommand{\eps}{\varepsilon}

\newcommand{\cU}{\mathcal{U}}
\newcommand{\ot}{\otimes}
\newcommand{\vphi}{\varphi}

\newcommand{\cR}{\mathcal{R}}
\newcommand{\cG}{\mathcal{G}}
\newcommand{\Tr}{\operatorname{Tr}}
\newcommand{\Z}{\mathbb{Z}}
\newcommand{\dpr}{^{\prime\prime}}

\newcommand{\al}{\alpha}
\newcommand{\be}{\beta}

\newcommand{\F}{\mathbb{F}}
\newcommand{\cS}{\mathcal{S}}
\newcommand{\C}{\mathbb{C}}
\newcommand{\cN}{\mathcal{N}}
\newcommand{\ovt}{\mathbin{\overline{\otimes}}}
\newcommand{\otalg}{\mathbin{\otimes_{\text{\rm alg}}}}
\newcommand{\Aut}{\operatorname{Aut}}

\newcommand{\id}{\mathord{\text{\rm id}}}

\newcommand{\Ker}{\operatorname{Ker}}

\newcommand{\cA}{\mathcal{A}}

\newcommand{\cH}{\mathcal{H}}

\newcommand{\cO}{\mathcal{O}}

\newcommand{\cB}{\mathcal{B}}

\newcommand{\cK}{\mathcal{K}}
\newcommand{\cC}{\mathcal{C}}
\newcommand{\cM}{\mathcal{M}}
\newcommand{\Mor}{\operatorname{Mor}}
\newcommand{\onb}{\operatorname{onb}}
\newcommand{\cE}{\mathcal{E}}
\newcommand{\cD}{\mathcal{D}}
\newcommand{\cL}{\mathcal{L}}
\newcommand{\Ptil}{\widetilde{P}}
\newcommand{\cKtil}{\widetilde{\mathcal{K}}}
\newcommand{\Ttil}{\widetilde{T}}
\newcommand{\cUB}{\mathcal{UB}}
\newcommand{\cUP}{\mathcal{UP}}
\newcommand{\cUA}{\mathcal{UA}}
\newcommand{\ibar}{\overline{\imath}}
\newcommand{\jbar}{\overline{\jmath}}
\newcommand{\cI}{\mathcal{I}}
\newcommand{\lspan}{\operatorname{span}}
\newcommand{\ibf}{\operatorname{ibf}}
\newcommand{\Irr}{\operatorname{Irr}}
\newcommand{\cJ}{\mathcal{J}}

\newcommand{\bG}{\mathbb{G}}
\newcommand{\op}{^\text{\rm op}}

\newcommand{\bGhat}{\widehat{\mathbb{G}}}
\newcommand{\inv}{\operatorname{inv}}
\newcommand{\Rep}{\operatorname{Rep}}
\newcommand{\SU}{\operatorname{SU}}
\newcommand{\homom}{\operatorname{hom}}
\newcommand{\cBtil}{\widetilde{\mathcal{B}}}
\newcommand{\cAtil}{\widetilde{\mathcal{A}}}
\newcommand{\vphitil}{\widetilde{\varphi}}
\newcommand{\rel}{\operatorname{rel}}
\newcommand{\SO}{\operatorname{SO}}
\newcommand{\QAut}{\operatorname{QAut}}
\newcommand{\cW}{\mathcal{W}}
\newcommand{\bH}{\mathbb{H}}

\begin{document}

\begin{center}
{\boldmath\LARGE\bf Quantum automorphism groups of connected locally\vspace{0.5ex}\\
finite graphs and quantizations of discrete groups}

\bigskip

{\sc by Lukas Rollier\footnote{\noindent KU~Leuven, Department of Mathematics, Leuven (Belgium).\\ E-mails: lukas.rollier@kuleuven.be and stefaan.vaes@kuleuven.be.}\textsuperscript{,}\footnote{L.R.\ is supported by FWO research project G090420N of the Research Foundation Flanders and by FWO-PAS research project VS02619N} and Stefaan Vaes\textsuperscript{1,}\footnote{S.V.\ is supported by FWO research project G090420N of the Research Foundation Flanders and by long term structural funding~-- Methusalem grant of the Flemish Government.}}
\end{center}

\begin{abstract}
\noindent We construct for every connected locally finite graph $\Pi$ the quantum automorphism group $\QAut \Pi$ as a locally compact quantum group. When $\Pi$ is vertex transitive, we associate to $\Pi$ a new unitary tensor category $\cC(\Pi)$ and this is our main tool to construct the Haar functionals on $\QAut \Pi$. When $\Pi$ is the Cayley graph of a finitely generated group, this unitary tensor category is the representation category of a compact quantum group whose discrete dual can be viewed as a canonical quantization of the underlying discrete group. We introduce several equivalent definitions of quantum isomorphism of connected locally finite graphs $\Pi$, $\Pi'$ and prove that this implies monoidal equivalence of $\QAut \Pi$ and $\QAut \Pi'$.
\end{abstract}

\section{Introduction and main results}

Quantum automorphism groups of finite graphs were introduced in \cite{Bic99,Ban03}. Contrary to what one might expect, they need not be finite quantum groups, but are compact quantum groups in the sense of Woronowicz. This leads to numerous intriguing families of compact quantum groups, see e.g.\ \cite{Sch19a,Sch19b}. To illustrate this with an isolated example, we mention that combining \cite{Kup96,Ara14,Edg19}, it follows that the dual $\bGhat$ of the quantum automorphism group $\bG$ of the Higman-Sims graph is an infinite discrete quantum group with property~(T), while the fusion rules of $\bG$ are abelian, since they coincide with the fusion rules of $\SO(5)$.

Quantum automorphism groups of finite graphs are in a deep way related to quantum information theory and in particular, quantum strategies for nonlocal games, through the notion of quantum isomorphism of graphs, see \cite{LMR17}. In particular, it was shown in \cite{MR19,MR20} that two finite graphs $\Pi$ and $\Pi'$ are quantum isomorphic if and only if for every finite \emph{planar} graph $K$, the number of graph homomorphisms from $K$ to $\Pi$ equals the number of graph homomorphisms from $K$ to $\Pi'$. Moreover, this implies that the quantum automorphism groups $\QAut \Pi$ and $\QAut \Pi'$ are monoidally equivalent, i.e.\ have equivalent representation categories.

The automorphism group $G$ of an \emph{infinite} graph $\Pi$ is a Polish group for the topology of pointwise convergence. If the graph $\Pi$ is \emph{connected} and \emph{locally finite}, then $G$ is a locally compact, totally disconnected group. These automorphism groups of connected locally finite graphs play a key role in the theory of totally disconnected groups. It is thus a natural problem to construct the \emph{quantum} automorphism group of a connected locally finite graph and fit this construction in the theory of locally compact quantum groups (in the sense of \cite{KV99,KV00}). Because of their totally disconnected nature, one should expect to define these quantum automorphism groups as ``algebraic'' quantum groups in the sense of \cite{VD96,KVD96}. If the graph is finite, one should find back the construction of \cite{Bic99,Ban03}.

Given a connected locally finite graph $\Pi$, it is straightforward to define a multiplier Hopf $*$-algebra $(\cA,\Delta)$ by mimicking the generators and relations approach for finite graphs. It was however an open problem to prove that $(\cA,\Delta)$ admits Haar functionals, i.e.\ positive functionals that are left, resp.\ right invariant. Since this existence problem for the Haar measure could not be solved, other approaches to quantum symmetries for infinite graphs were proposed. In \cite{Voi22}, quantum automorphism groups are considered without their topology, i.e.\ as discrete and therefore often uncountable quantum groups (see Remark \ref{rem.compare-voigt}). In \cite{GS12}, only a quantum semigroup structure is considered.

The main result of this paper is to construct the left and right Haar measure on the quantum automorphism group of any connected locally finite graph $\Pi$. Denoting by $I = V(\Pi)$ the vertex set of $\Pi$, we thus prove that the natural multiplier Hopf $*$-algebra $(\cA,\Delta)$ generated by self-adjoint idempotents $(u_{ij})_{i,j \in I}$ forming a ``magic unitary'' that commutes with the adjacency matrix of $\Pi$, admits positive faithful invariant functionals. As a result, we obtain $\QAut \Pi$ as an algebraic quantum group (in the sense of \cite{VD96,KVD96}) and as a locally compact quantum group (in the sense of \cite{KV99,KV00}). Our results thus provide a completely new and wide class of locally compact quantum groups.

The main new tool to construct the Haar measures on $(\cA,\Delta)$ is a unitary tensor category $\cC(\Pi)$ that we associate to a connected, locally finite, vertex transitive graph. This new unitary tensor category $\cC(\Pi)$ is actually of independent interest. When $\Pi$ is the Cayley graph of a finitely generated group $\Gamma$, with symmetric finite generating set $S = S^{-1} \subset \Gamma$, then $\cC(\Pi)$ has a natural fiber functor and thus is the representation category of a compact quantum group $\bG$. As we explain below, $\bG$ is isomorphic with the quantum isometry group defined in \cite{BhS10} and the dual $\bGhat$ can be viewed as a canonical quantization of the pair $(\Gamma,S)$. It turns out that many interesting families of discrete quantum groups arise as such quantizations. We prove for instance that the property~(T) discrete quantum groups of \cite{VV18} can be identified with the canonical quantization of the groups constructed in \cite{CMSZ91} as the groups acting simply transitively on $\widetilde{A_2}$-buildings.

We also relate quantum automorphism groups of connected locally finite graphs to quantum information theory and nonlocal games. We generalize the results of \cite{BCEHPSW18,MR19,MR20} to the infinite setting. We prove that two graphs $\Pi$ and $\Pi'$ are quantum isomorphic if and only if the appropriate equalities of pointed homomorphism counts from finite planar graphs to $\Pi$ and $\Pi'$ hold. We also prove that these properties are equivalent with the existence of an algebraic quantum isomorphism (in the sense of \cite{BCEHPSW18}) and that they imply the monoidal equivalence of the quantum automorphism groups $\QAut \Pi$ and $\QAut \Pi'$.

To formulate our results in more detail, we first need some basic terminology. For us, a \emph{graph} $\Pi$ is a pair $(I,E)$ of a set of \emph{vertices} $I$ and a set of \emph{edges} $E$ that is a subset $E \subset I \times I$ satisfying $(i,j) \in E$ iff $(j,i) \in E$. We thus do not consider multiple edges or orientations, but we do allow loops. We write $i \sim j$ when $(i,j) \in E$. A \emph{path} of length $n \geq 0$ in $\Pi$ is a sequence $(i_0,\ldots,i_n)$ of vertices satisfying $i_{k-1} \sim i_k$ for all $k \in \{1,\ldots,n\}$. The graph $\Pi$ is said to be \emph{connected} if for all $i,j \in I$, there exists an $n \geq 0$ and a path $(i_0,\ldots,i_n)$ in $\Pi$ with $i_0 = i$ and $i_n = j$. The \emph{degree} $\deg(i)$ of a vertex $i \in I$ is the cardinality of the set $\{j \in I \mid i \sim j\}$. The graph $\Pi$ is said to be \emph{locally finite} if $\deg(i) < \infty$ for all $i \in I$.

The automorphism group $\Aut \Pi$ of a graph $\Pi$ is the group of all permutations $\si : I \to I$ satisfying $(\si \times \si)(E) = E$. We equip $\Aut \Pi$ with the topology of pointwise convergence. If $I$ is countable, $\Aut \Pi$ is a Polish group. If $\Pi$ is connected and locally finite, then $\Aut \Pi$ is a locally compact, second countable, totally disconnected group.

We say that a sum $\sum_{i \in L} a_i$ of a family of elements $(a_i)_{i \in L}$ in an algebra $\cA$ \emph{converges strictly to $1$} if for all $b \in \cA$, we have that $a_i b = 0 = b a_i$ for all but finitely many $i \in L$, and
$$\sum_{i \in L} a_i b = b = \sum_{i \in L} b a_i \quad\text{for all $b \in \cA$.}$$
Recall that an algebra $\cA$ is said to be nondegenerate if for every nonzero $a \in \cA$, there exist $b,c \in \cA$ with $a b \neq 0$ and $c a \neq 0$. One then embeds $\cA$ as a subalgebra of its multiplier algebra $M(\cA)$. When $\cA$ and $\cB$ are nondegenerate $*$-algebras, a $*$-homomorphism $\pi : \cA \to M(\cB)$ is called nondegenerate if the linear span of all $\pi(a)b$, $a \in \cA$, $b \in \cB$, equals $\cB$. Such a nondegenerate $*$-homomorphism has a unique extension to a $*$-homomorphism $\pi : M(\cA) \to M(\cB)$. A net $(a_j)_{j \in J}$ in $M(\cA)$ is said to \emph{converge strictly} to $a \in M(\cA)$ if for every $b \in \cA$, the net $(b a_j)_{j \in J}$ is eventually equal to $ba$ and the net $(a_j b)_{j \in J}$ is eventually equal to $ab$. Note that this is equivalent with convergence in the strict topology on $M(\cA)$.

\begin{letterthm}\label{thm.lc-qaut}
Let $\Pi$ be a connected locally finite graph with vertex set $I$.
\begin{enumlist}
\item There exists a (necessarily unique) universal $*$-algebra $\cA$ generated by elements $(u_{ij})_{i,j \in I}$ satisfying the following properties.
\begin{itemlist}
\item $u_{ij}^* = u_{ij}$ and $u_{ij}^2 = u_{ij}$ for all $i,j \in I$.
\item For all $i,j,k \in I$ with $j \neq k$, we have that $u_{ij} u_{ik} = 0 = u_{ji} u_{ki}$. For all $i,j \in I$, we have that $\sum_{k \in I} u_{ik} = 1 = \sum_{k \in I} u_{kj}$ strictly.
\item For all $i,j \in I$, we have that $\sum_{k \in I : k \sim j} u_{ik} = \sum_{k \in I : k \sim i} u_{kj}$. This means that the matrix $u$ commutes with the adjacency matrix of $\Pi$.
\end{itemlist}
\item The $*$-algebra $\cA$ is nondegenerate and there is a unique nondegenerate $*$-homomorphism $\Delta : \cA \to M(\cA \ot \cA)$ satisfying $\Delta(u_{ij}) = \sum_{k \in I} (u_{ik} \ot u_{kj})$ strictly for all $i,j \in I$. The pair $(\cA,\Delta)$ is a multiplier Hopf $*$-algebra in the sense of \cite[Definition 2.4]{VD92}.
\item The multiplier Hopf $*$-algebra $(\cA,\Delta)$ admits a positive faithful left invariant, resp.\ right invariant, functional. It is thus an algebraic quantum group in the sense of \cite{VD96} and \cite[Definition 1.2]{KVD96}. Taking the norm closure, resp.\ weak closure, in the GNS-construc\-tion of these invariant functionals, we get the C$^*$, resp.\ von Neumann, versions of a locally compact quantum group in the sense of \cite[Definition 4.1]{KV99}, resp.\ \cite[Definition 1.1]{KV00}.
\end{enumlist}
\end{letterthm}

The precise meaning of statement~1 is that there exists a $*$-algebra $\cA$ generated by elements $(u_{ij})_{i,j \in I}$ satisfying the three listed properties and that for any other $*$-algebra $\cB$ generated by elements $(v_{ij})_{i,j \in I}$ satisfying the same properties, there exists a surjective $*$-homomorphism $\pi : \cA \to \cB$ satisfying $\pi(u_{ij}) = v_{ij}$ for all $i,j \in I$.

Recall that a functional $\vphi : \cA \to \C$ is said to be left invariant if $(\id \ot \vphi)\Delta(a) = \vphi(a)1$ for all $a \in \cA$, which should be interpreted as
$$(\id \ot \vphi)(\Delta(a) (b \ot 1)) = \vphi(a) b = (\id \ot \vphi)((b \ot 1)\Delta(a)) \quad\text{for all $a,b \in \cA$,}$$
which makes sense because $\Delta(A) (A \ot 1) = A \ot A$ in a multiplier Hopf $*$-algebra.

For a connected locally finite graph $\Pi$, the classical automorphism group $G = \Aut \Pi$ is a totally disconnected, locally compact group. The $*$-algebra $\cO(G)$ of locally constant functions, with its natural comultiplication, is a multiplier Hopf $*$-algebra. With $(\cA,\Delta)$ defined by Theorem~\ref{thm.lc-qaut}, we have the surjective multiplier Hopf $*$-algebra homomorphism
\begin{equation}\label{eq.canonical-closed-quantum-subgroup}
\pi : \cA \to \cO(G) : \pi(u_{ij}) = 1_{\{\si \in G \mid \si(j) = i\}} \; .
\end{equation}
In this way, $\Aut \Pi$ is a closed quantum subgroup of the locally compact quantum group defined by $(\cA,\Delta)$ (see Remark \ref{rem.closed-quantum-subgroup}). We thus introduce the following terminology.

\begin{letterdef}\label{def.qAut}
Given a connected locally finite graph $\Pi$, we denote by $\QAut \Pi$ the locally compact quantum group characterized by Theorem \ref{thm.lc-qaut} and call it the \emph{quantum automorphism group} of $\Pi$.
\end{letterdef}

As mentioned above, the main difficulty to prove Theorem \ref{thm.lc-qaut} and the main novelty of this paper is to show that the natural multiplier Hopf $*$-algebra $(\cA,\Delta)$ admits positive left and right invariant functionals. We prove this by first associating to $\Pi$ a unitary tensor category $\cC(\Pi)$ in which the morphisms are defined by matrices given by homomorphism counts from finite planar graphs to $\Pi$. This is closely related to \cite{MR19,MR20}, but also a bit different, since we should not try to define $\Rep(\QAut \Pi)$ as a unitary tensor category, simply because $\QAut \Pi$ need not be compact.

Once Theorem \ref{thm.lc-qaut} has been proven, all this will become more transparent and we will prove in Proposition \ref{prop.link-with-equivariant-bimodules} that $\cC(\Pi)$ can be viewed as the unitary tensor category of all $\QAut \Pi$-equivariant $\ell^\infty(I)$-bimodules $\cH$ of finite type, where the latter means that $p_i \cdot \cH$ and $\cH \cdot pi_i$ are finite dimensional for every minimal projection $p_i \in \ell^\infty(I)$. In the classical case, given a transitive and proper action of a locally compact group $G$ on a countable set $I$, this unitary tensor category of $G$-equivariant $\ell^\infty(I)$-bimodules of finite type was introduced in \cite{AV16}.

Note that in the preceding two paragraphs, we are implicitly assuming that the graph $\Pi$ is sufficiently regular, for instance vertex transitive, so that indeed the action of $\QAut \Pi$ on $I$ is transitive. In general, we rather construct a unitary $2$-category $\cC(\Pi)$, where the $0$-cells are indexed by the quantum orbits of the action of $\QAut \Pi$ on $I$.

We actually perform all our constructions of the quantum automorphism group $\QAut \Pi$ and of the unitary $2$-category $\cC(\Pi)$ relative to any graph category $\cD$ (in the sense of \cite[Definition 8.1]{MR19}, see Section \ref{sec.bi-labeled-graphs}) containing the graph category $\cP$ of planar bi-labeled graphs. So, for any such graph category $\cD$ and for any connected locally finite graph $\Pi$, we obtain a locally compact quantum group $\bG$ and a unitary $2$-category $\cC(\cD,\Pi)$. We prove in Proposition \ref{prop.link-with-equivariant-bimodules} that $\cC(\cD,\Pi)$ is equivalent with the category of $\bG$-equivariant $\ell^\infty(I)$-bimodules of finite type. When $\cD$ is the graph category of planar bi-labeled graphs, $\bG$ equals the quantum automorphism group $\QAut \Pi$. When $\cD$ is the graph category of all bi-labeled graphs, $\bG$ equals the classical automorphism group $\Aut \Pi$. In general, we have $\Aut \Pi \subset \bG \subset \QAut \Pi$ as closed quantum subgroups.

It follows from \cite[Proposition 6.1]{VV18} that if $\Pi$ is the Cayley graph of a countable group $\Gamma$ with a finite symmetric generating set $S = S^{-1} \subset \Gamma$, then $\cC(\Pi)$ admits a natural fiber functor and thus becomes the representation category of a compact quantum group $\bG$. Combining Theorem \ref{thm.quantization-discrete-group} and Proposition \ref{prop.link-with-equivariant-bimodules}, we obtain the following description of this compact quantum group. Moreover, we prove in Proposition \ref{prop.identify-with-quantum-isometry-group} that $\bG$ is also isomorphic to the quantum isometry group introduced in \cite{BhS10} associated with the spectral triple given by $(\Gamma,S)$.

\begin{letterthm}\label{thm.identify-categories}
Let $\Gamma$ be a countable group with finite symmetric generating set $S = S^{-1} \subset \Gamma$. We denote by $\bG$ the universal compact quantum group where $\cO(\bG)$ is generated by the entries of an $S \times S$ unitary representation $U$ such that for every $n \geq 1$, the vector $\xi_n \in \ell^2(S^n)$
$$\xi_n(s_1,\ldots,s_n) = \begin{cases} 1 &\quad\text{if $s_1 \cdots s_n = e$ in $\Gamma$,}\\ 0 &\quad\text{otherwise,}\end{cases}$$
is invariant under the $n$-fold tensor power of $U$.

Denote by $\Pi$ the Cayley graph of $(\Gamma,S)$ with quantum automorphism group $\QAut \Pi$. Then the unitary tensor category of $\QAut \Pi$-equivariant $\ell^\infty(\Gamma)$-bimodules of finite type is equivalent with the representation category $\Rep \bG$.
\end{letterthm}

Note that $\bG$ is of Kac type and note that, using $\xi_2$, one can show that the fundamental representation $U$ is isomorphic to its conjugate $\overline{U}$.

By construction, $\widehat{\Gamma}$ is a closed quantum subgroup of the compact quantum group $\bG$ in Theorem~\ref{thm.identify-categories}, and this identification is given by the surjective $*$-homomorphism $\pi : \cO(\bG) \to \C[\Gamma] : \pi(u_{st}) = \delta_{s,t} \, s$. We prove that there is no intermediate group $\cO(\bG) \to \C[\Lambda] \to \C[\Gamma]$ and therefore view the dual $\widehat{\bG}$ as a quantization of $\Gamma$, in the following sense.

\begin{letterdef}\label{def.quantization}
We say that a discrete quantum group $\bGhat$, realized as the dual of a compact quantum group $\bG$, is a \emph{quantization} of the countable group $\Gamma$ if we are given a surjective Hopf $*$-algebra homomorphism $\pi : \cO(\bG) \to \C[\Gamma]$ that does not factor through surjective Hopf $*$-algebra homomorphisms $\cO(\bG) \to \C[\Lambda] \to \C[\Gamma]$ unless the second arrow is an isomorphism.
\end{letterdef}

The prototype examples of quantizations in the sense of Definition \ref{def.quantization} (see Example \ref{ex.quantizations}) are $\widehat{A_u(n)} \to \F_n$, induced by $u_{ij} \mapsto \delta_{i,j} \, a_i$, where $u_{ij}$ are the entries of the fundamental $n$-dimensional unitary representation of the universal unitary quantum group $A_u(n)$ and where $a_1,\ldots,a_n$ are free generators of $\F_n$. The second example is $\widehat{A_o(n)} \to (\Z/2\Z)^{\ast n}$, induced by $u_{ij} \mapsto \delta_{i,j} \, a_i$, where this time $a_i$ are the free generators of the free product $(\Z/2\Z)^{\ast n}$ satisfying $a_i^2 = e$.

In Section \ref{sec.quantum-isomorphism}, we define when two connected locally finite graphs $\Pi$ and $\Pi'$ are \emph{quantum isomorphic}. Quantum isomorphism of finite graphs can be defined in many equivalent ways, see \cite{LMR17,BCEHPSW18,MR19,MR20}, and has become a key notion in quantum information theory. Three of these definitions have an immediate counterpart for connected locally finite graphs: \emph{quantum isomorphism} given by the existence of a magic unitary on a Hilbert space intertwining the adjacency matrices (see Definition \ref{def.quantum-iso}), \emph{algebraic quantum isomorphism} given by the existence of nonzero $*$-algebra generated by the entries of a magic unitary intertwining the adjacency matrices (see Definition \ref{def.algebraic-iso}) and \emph{planar isomorphism} given by the equality of pointed homomorphism counts from finite planar graphs (see Definition \ref{def.planar-isomorphism}). We then prove the following result, generalizing \cite[Theorem 4.9]{BCEHPSW18} and \cite[Theorem 7.16]{MR19} from finite graphs to connected locally finite graphs.

\begin{letterthm}[{See Theorem \ref{thm.quantum-iso-criteria}}]
For all connected locally finite graphs $\Pi$ and $\Pi'$, the notions of quantum isomorphism, algebraic quantum isomorphism and planar isomorphism coincide.
\end{letterthm}

\section{Locally compact quantum automorphism groups of connected locally finite graphs}

\subsection{Bi-labeled graphs and graph categories, following \cite{MR19,MR20}}\label{sec.bi-labeled-graphs}

Fix a connected locally finite graph $\Pi$ with vertex set $I$. For $i, j \in I$, we write $i \sim j$ if $(i,j)$ is an edge in $\Pi$.

Following \cite[Definition 3.1]{MR19}, a \emph{bi-labeled graph} $(K,x,y) \in \cG(n,m)$ is a finite graph $K = (V,E)$ together with tuples $x \in V^n$ and $y \in V^m$ of vertices. We define the subsets $\cG_c(n,m) \subset \cG_1(n,m) \subset \cG_2(n,m) \subset \cG(n,m)$ in the following way.

We define $\cG_c(n,0) = \cG_c(0,m) = \emptyset$ for all $n,m \geq 0$. For all $n, m \geq 1$, we define $\cG_c(n,m)$ as the set of bi-labeled graphs $(K,x,y)$ such that $K$ is connected.

We define $\cG_1(n,0) = \cG_1(0,m) = \emptyset$ for all $n,m \geq 0$. For all $n, m \geq 1$, we define $\cG_1(n,m)$ as the set of bi-labeled graphs $(K,x,y)$ such that every connected component of $K$ intersects $x$ and intersects $y$.

Finally, we define $\cG_2(0,0) = \emptyset$ and, for all $n,m \geq 0$ with $n+m \geq 1$, we define $\cG_2(n,m)$ as the set of bi-labeled graphs $(K,x,y)$ such that every connected component of $K$ intersects $x \cup y$.

Since $\Pi$ is locally finite, for every $\cK = (K,x,y) \in \cG_2(n,m)$, the definition of the $\Pi$-homo\-mor\-phism matrix $T^\cK$ in \cite[Definition 3.4]{MR19} makes sense. This matrix $T^\cK$ is the $I^n \times I^m$ matrix with values in $\N \cup \{0\}$ defined by
$$T^\cK_{ij} = \# \bigl\{ \vphi : V(K) \to I \bigm| \vphi \;\;\text{is a graph homomorphism and}\;\; \vphi(x) = i \; , \; \vphi(y) = j \bigr\} \; ,$$
for all $i \in I^n$ and $j \in I^m$. Here, a map $\vphi : V(K) \to I$ is called a graph homomorphism if $(\vphi(v),\vphi(w)) \in E(\Pi)$ for all $(v,w) \in E(K)$. By convention, when $n=0$ or $m=0$, we view $T^\cK$ as a row indexed by $I^m$, resp.\ a column indexed by $I^n$.

When $\cK \in \cG_1(n,m)$, the matrix $T^\cK$ has the following finiteness properties: for every $i \in I^n$, there are only finitely many $j \in I^m$ with $T^\cK_{ij} \neq 0$, and for every $j \in I^m$, there are only finitely many $i \in I^n$ with $T^\cK_{ij} \neq 0$. Since we only assumed that $\Pi$ is locally finite, the matrix $T^\cK$ need not define a bounded operator from $\ell^2(I^m)$ to $\ell^2(I^n)$. For instance, if $\cK = (K,x,y) \in \cG(1,1)$ is defined by $V(K) = \{0,1\}$, $E(K) = \{(0,1),(1,0)\}$ and $x_1 = y_1 = 0$, as illustrated in Figure \ref{fig.simple-K}, then $T^\cK$ is the diagonal matrix with $T^{\cK}_{ii} = \deg i$ for all $i \in I$.
\begin{figure}[h]
    \centering
    \begin{tikzpicture}
        \filldraw (0,0) circle (1.5pt) node[anchor=south]{$x_1$} node[anchor=north]{$y_1$};
        \filldraw (1.4,0) circle (1.5pt);
        \draw (0,0) -- (1.4,0);
    \end{tikzpicture}
    \caption{The bi-labeled graph $\cK \in \cG(1,1)$ with $T^{\cK}_{ii} = \deg i$ for all $i \in I$}\label{fig.simple-K}
\end{figure}

In \cite[Section 3.1]{MR19}, several operations on bi-labeled graphs were defined. The \emph{composition} $\cK_1 \circ \cK_2$ of $\cK_1 \in \cG(n,k)$ and $\cK_2 \in \cG(k,m)$ is a bi-labeled graph in $\cG(n,m)$ such that, in case $\cK_1$ and $\cK_2$ belong to $\cG_1$, we have $T^{\cK_1 \circ \cK_2} = T^{\cK_1} \, T^{\cK_2}$, by \cite[Lemma 3.21]{MR19}. Note that the matrix product at the right hand side makes sense because of the finiteness properties of $T^\cK$. When $\cK_1$ and $\cK_2$ are connected, $\cK_1 \circ \cK_2$ remains connected.

The \emph{tensor product} $\cK \ot \cK'$ of $\cK \in \cG(n,m)$ and $\cK' \in \cG(n',m')$ belongs to $\cG(n+n',m+m')$ and consists of putting $\cK$ and $\cK'$ next to each other. We have that $T^{\cK \ot \cK'} = T^\cK \ot T^{\cK'}$. When $\cK, \cK' \in \cG_i$ for $i \in \{1,2\}$, then also $\cK \ot \cK' \in \cG_i$. Note however that $\cK \ot \cK'$ is not connected. Below, we will also define a relative tensor product that preserves connectedness. The \emph{transpose} $\cK^*$ of $\cK = (K,x,y)$ is defined as $\cK^* = (K,y,x)$ and satisfies $T^{\cK^*} = (T^\cK)^*$. All $\cG_c$, $\cG_1$ and $\cG_2$ are stable under this transpose operation.

For all $n,m \geq 0$, we define the basic bi-labeled graph $\cM^{n,m}$ consisting of a single vertex, say $0$, no edges and labeling $x_1=\cdots=x_n=0=y_1=\cdots=y_m$. We sometimes write $1$ instead of $\cM^{1,1}$, because the associated matrix is the identity matrix.

Following \cite[Definition 8.1]{MR19}, a \emph{graph category} $\cD$ is a collection of subsets $\cD(n,m) \subset \cG(n,m)$ that contains $\cM^{1,1}$ and $\cM^{2,0}$ and that is closed under composition, tensor product and transpose. We then denote $\cD_c(n,m) = \cD(n,m) \cap \cG_c(n,m)$ and similarly with $\cD_1$ and $\cD_2$.

We finally recall from \cite[Definition 5.4]{MR19} the definition of the graph category $\cP$ of \emph{planar bi-labeled graphs}. A bi-labeled graph $(K,x,y)$ belongs to $\cP(n,m)$ if the graph $K$ is planar and if the labelings $x$ and $y$ are lying in an appropriate way ``at the outside'' of $K$, see \cite[Definition 5.4]{MR19}. By \cite[Theorem 6.7]{MR19}, $\cP$ is the smallest graph category containing $\cM^{1,0}$, $\cM^{1,2}$ and the bi-labeled graph $\cK = (K,x,y) \in \cG(1,1)$ with $V(K) = \{0,1\}$, $E(K) = \{(0,1),(1,0)\}$ and $x_1 = 0$, $y_1 = 1$, whose associated matrix $T^\cK$ is the adjacency matrix of $\Pi$.
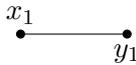
\begin{figure}[h]
    \centering
    \begin{tikzpicture}
        \filldraw (0,0) circle (1.5pt) node[anchor=south]{$x_1$};
        \filldraw (1.4,0) circle (1.5pt) node[anchor=north]{$y_1$};
        \draw (0,0) -- (1.4,0);
    \end{tikzpicture}
    \caption{The bi-labeled graph $\cK \in \cG(1,1)$ such that $T^{\cK}$ is the adjacency matrix of $\Pi$}
\end{figure}

\subsection{\boldmath Unitary $2$-categories associated with connected locally finite graphs}\label{sec.2-category}

Given a countable set $I$, we consider the category of Hilbert $\ell^\infty(I)$-bimodules, consisting of Hilbert spaces equipped with commuting normal representations of $\ell^\infty(I)$. Denoting by $p_i \in \ell^\infty(I)$ the canonical minimal projections, a Hilbert $\ell^\infty(I)$-bimodule $\cH$ is the same thing as an $I \times I$ graded Hilbert space, by considering the closed subspaces $\cH_{i,j} = p_i \cdot \cH \cdot p_j$. Together with $\ell^\infty(I)$-bimodular bounded operators as morphisms and the relative tensor product
\begin{equation}\label{eq.relative-tensor-product-Hilbert}
\cH \ot_{I} \cK = \bigoplus_{i \in I} \bigl( \cH \cdot p_i \ot p_i \cdot \cK) \; ,
\end{equation}
we obtain a C$^*$-tensor category. We view $\cH \ot_{I} \cK$ as a closed subspace of $\cH \ot \cK$. When $T : \cH \to \cH'$ and $S : \cK \to \cK'$ are bounded $\ell^\infty(I)$-bimodular operators, the ordinary tensor product $T \ot S$ preserves these closed subspaces. We denote by $T \ot_I S$ the restriction of $T \ot S$ to $\cH \ot_I \cK$, defining in this way the tensor product of morphisms.

Note that $\ell^2(I)$ with its obvious $\ell^\infty(I)$-bimodule structure serves as the identity object. Every Hilbert $\ell^\infty(I)$-bimodule $\cH$ has a natural conjugate $\overline{\cH}$, with bimodule structure $f \cdot \overline{\xi} \cdot g = \overline{g^* \cdot \xi \cdot f^*}$ for all $\xi \in \cH$ and $f,g \in \ell^\infty(I)$. We say that an $\ell^\infty(I)$-bimodule is of \emph{finite type} if $p_i \cdot \cH$ and $\cH \cdot p_i$ are finite-dimensional for all $i \in I$. Whenever $\cH$ is of finite type, we have the natural morphism
$$\ell^2(I) \to \cH \ot_{I} \overline{\cH} : e_i \mapsto \sum_{j \in I} \; \sum_{\xi \in \onb(p_i \cdot \cH \cdot p_j)} (\xi \ot \overline{\xi}) \; ,$$
where we denoted by $\onb(p_i \cdot \cH \cdot p_j)$ any choice of orthonormal basis of $p_i \cdot \cH \cdot p_j$.

Assume now that $G$ is a locally compact group and that $G \actson I$ is a continuous, transitive action with compact stabilizers, i.e.\ $I = G/K$ for some compact open subgroup $K \subset G$. Following \cite[Proposition 2.2]{AV16}, we consider the category of $G$-equivariant Hilbert $\ell^\infty(I)$-bimodules of finite type, i.e.\ Hilbert $\ell^\infty(I)$-bimodules $\cH$ of finite type equipped with a unitary representation $\pi : G \to \cU(\cH)$ satisfying
$$\pi(g)(p_i \cdot \cH \cdot p_j) = p_{g \cdot i} \cdot \cH \cdot p_{g \cdot j} \quad\text{for all $g \in G$, $i,j \in I$.}$$
Together with $G$-equivariant $\ell^\infty(I)$-bimodular bounded operators as morphisms and the relative tensor product $\cH \ot_{I} \cK$, we obtain the unitary tensor category $\cC(G \actson I)$ introduced in \cite{AV16}.

If we drop the transitivity assumption for the action $G \actson I$, we rather find a unitary $2$-category, defined as follows. Denote by $(I_a)_{a \in \cE}$ the partition of $I$ into orbits for $G \actson I$. We consider $\cE$ as the set of $0$-cells. For all $a,b \in \cE$, the class of $1$-cells $\cC_{a-b}(G \actson I)$ consists of the $G$-equivariant Hilbert $\ell^\infty(I_a)$-$\ell^\infty(I_b)$-bimodules of finite type. The morphisms between two such $1$-cells are the $G$-equivariant $\ell^\infty(I_a)$-$\ell^\infty(I_b)$-bimodular bounded operators. The tensor product of $\cH \in \cC_{a-b}(G \actson I)$ and $\cK \in \cC_{b-c}(G \actson I)$ is defined as $\cH \ot_{I_b} \cK$.

In this section we define, given any connected locally finite graph $\Pi$ with vertex set $I$ and given any graph category $\cD$ containing all planar bi-labeled graphs, a similar unitary $2$-category $\cC(\cD,\Pi)$ of Hilbert bimodules. When $\cD$ is the graph category of all bi-labeled graphs, we will prove in Proposition \ref{prop.link-with-equivariant-bimodules} that $\cC(\cD,\Pi)$ is equivalent with the unitary $2$-category $\cC(G \actson I)$ defined above, with $G = \Aut \Pi$. When $\cD = \cP$ is the graph category of all planar bi-labeled graphs, we will prove that $\cC(\cP,\Pi) \cong \cC(\bG \actson I)$, where $\bG = \QAut \Pi$ is the quantum automorphism group of $\Pi$, which we still have to construct. We however first construct the unitary $2$-category $\cC(\cP,\Pi)$ and use it as a tool to construct $\QAut \Pi$ and to prove that $\QAut \Pi$ is a locally compact quantum group.

For the rest of this section, fix a connected locally finite graph $\Pi$ with vertex set $I$ and fix a graph category $\cD$ containing all planar bi-labeled graphs. For all $n,m \geq 0$, define
$$\cL(n,m) = \{ (K,x,y) \in \cD_c(n+1,m+1) \mid x_0 = y_0 \;\;\text{and}\;\; x_n = y_m \} \; .$$
Notice that we made a shift of notation from $n$-tuples of labels $(x_1,\ldots,x_n)$ to $(n+1)$-tuples of labels $(x_0,\ldots,x_n)$. This will be more natural in a context of relative tensor products over $\ell^\infty(I)$ and $\ell^\infty(I)$-bimodular maps.

For every $\cK = (K,x,y) \in \cL(n,m)$, the $I^{n+1} \times I^{m+1}$ matrix $T^\cK$ is $\ell^\infty(I)$-bimodular, in the sense that $T^\cK_{ij} = 0$ if $i_0 \neq j_0$ or if $i_n \neq j_m$. Since $K$ is supposed to be connected, once we fix one of the indices $i_s$ or $j_t$, there are only finitely many choices for the other indices such that $T^\cK_{ij} \neq 0$. In particular, $T^\cK$ defines and will be viewed as an $\ell^\infty(I)$-bimodular map from $\cF(I^{m+1})$, the space of finitely supported functions on $I^{m+1}$, to $\cF(I^{n+1})$. Here, we equip $\cF(I^{n+1})$ with the bimodule structure $(F \cdot \xi \cdot F')(i) = F(i_0) \, \xi(i) \, F'(i_n)$. As mentioned above, if the degree of the vertices of $\Pi$ is not uniformly bounded, then $T^\cK$ need not extend to a bounded operator from $\ell^2(I^{m+1})$ to $\ell^2(I^{n+1})$.

Note that $\cL$ is closed under composition and transpose of bi-labeled graphs. As before, $T^{\cK^*} = (T^\cK)^*$ and $T^{\cK \circ \cK'} = T^{\cK} \, T^{\cK'}$. When $\cK \in \cL(n,m)$ and $\cK' \in \cL(n',m')$, the \emph{relative tensor product} $\cK \ot_r \cK' \in \cL(n+n',m+m')$ is defined as
$$\cK \ot_r \cK' = (1^{\otimes n} \otimes \cM^{1,2} \otimes 1^{\otimes m}) \circ (\cK \ot \cK') \circ (1^{\otimes n} \otimes \cM^{2,1} \otimes 1^{\otimes m}) \; .$$
Since the planar bi-labeled graphs are contained in $\cD$, we get that $\cK \ot_r \cK' \in \cL(n+n',m+m')$. By construction, $T^{\cK \ot_r \cK'} = T^\cK \ot_I T^{\cK'}$. In matrix notation, this means that
$$T^{\cK \ot_r \cK'}_{ij} = T^\cK_{i_0\cdots i_n,j_0 \cdots j_m} \, T^{\cK'}_{i_n \cdots i_{n+n'},j_m \cdots j_{m+m'}} \; .$$
Viewing $\cF(I^{n+n'+1}) = \cF(I^{n+1}) \ot_I \cF(I^{n'+1})$, this relative tensor product is also compatible with the tensor product of $\ell^\infty(I)$-bimodular maps considered above.

We define $\Mor(n,m)$ as the linear span of all matrices/operators $T^\cK$ with $\cK \in \cL(n,m)$. By construction, $\Mor(n,m)^* = \Mor(m,n)$, $\Mor(n,k) \Mor(k,m) \subset \Mor(n,m)$ and $T \ot_I S \in \Mor(n+n',m+m')$ for all $T \in \Mor(n,m)$ and $S \in \Mor(n',m')$.

Since elements of $\Mor(n,m)$ are $\ell^\infty(I)$-bimodular, a special role is played by $\Mor(0,0)$ and $\Mor(1,1)$, since they form $*$-algebras of functions on $I$, resp.\ $I \times I$. We define the equivalence relation $\approx$ on $I$ (not to be confused with the edge relation $i \sim j$) in the following way:
\begin{equation}\label{eq.equiv-rel-on-I}
i \approx j \quad\text{if and only if}\quad T_{ii} = T_{jj} \;\;\text{for all $T \in \Mor(0,0)$.}
\end{equation}
We denote by $\cE$ the set of equivalence classes of $\approx$, with corresponding partition $(I_a)_{a \in \cE}$ of $I$. Later on, in Corollary \ref{cor.about-rel-E}, we will interpret $\cE$ as the set of quantum orbits of $\Pi$, meaning that $\cE$ is a singleton if and only if the graph $\Pi$ is quantum vertex transitive.

Note that every graph automorphism $\si \in \Aut \Pi$ acts on $I^{n+1}$ diagonally and that all $T \in \Mor(n,m)$ are $\Aut \Pi$-equivariant. In particular, all $T \in \Mor(0,0)$ define $\Aut \Pi$-invariant functions on $I$. So when $\Pi$ is a vertex transitive graph, $\cE$ is a singleton.

For future reference, we gather the notations introduced so far in the following definition, and add several other vector spaces of matrices/operators that play a key role in this paper.

\begin{definition}\label{def.mor-spaces}
Let $\Pi$ be a connected locally finite graph. Let $\cD$ be a graph category containing all planar bi-labeled graphs. We define
\begin{align*}
& \cL(n,m) = \{ (K,x,y) \in \cD(n+1,m+1) \mid K \;\;\text{is connected,}\;\; x_0 = y_0 \;\;\text{and}\;\; x_n = y_m \} \; ,\\
& \Mor(n,m) = \text{linear span of all $I^{n+1} \times I^{m+1}$ matrices $T^\cK$ with $\cK \in \cL(n,m)$} \; .
\end{align*}
Denote by $(I_a)_{a \in \cE}$ the equivalence classes of the equivalence relation defined in \eqref{eq.equiv-rel-on-I}. For every $a \in \cE$, we denote by $1_a \in \ell^\infty(I)$ the indicator function of $I_a \subset I$. We define for all $a,b \in \cE$,
\begin{align*}
& \Mor_{a-}(n,m) = \Mor(n,m)(1_a \ot 1^{\ot m}) = (1_a \ot 1^{\ot n})\Mor(n,m) \; ,\\
& \Mor_{-b}(n,m) = \Mor(n,m)(1^{\ot m} \ot 1_b) = (1^{\ot n} \ot 1_b)\Mor(n,m) \; ,\\
& \Mor_{a-b}(n,m) = \Mor(n,m)(1_a \ot 1^{\ot (m-1)} \ot 1_b) = (1_a \ot 1^{\ot (n-1)} \ot 1_b)\Mor(n,m) \; ,
\end{align*}
with the convention that $\Mor_{a-b}(0,m) = 0 = \Mor_{a-b}(n,0)$ when $a \neq b$.
\end{definition}

The following lemma contains all that is needed to define $\cC(\cD,\Pi)$ as a concrete unitary $2$-category of $\ell^\infty(I_a)$-$\ell^\infty(I_b)$-bimodules of finite type.

\begin{lemma}\phantomsection\label{lem.crucial-technical-Mor-lemma}
\begin{enumlist}
\item For every $T \in \Mor(n,n)$, the functions
\begin{align*}
& (\id \ot \Tr_{n})(T) : I \to \C : i \mapsto \sum_{i_1,\ldots,i_{n} \in I} T_{ii_1\cdots i_{n},ii_1\cdots i_{n}} \quad\text{and}\\
& (\Tr_{n} \ot \id)(T) : I \to \C : i \mapsto \sum_{i_0,\ldots,i_{n-1} \in I} T_{i_0\cdots i_{n-1}i,i_0\cdots i_{n-1}i}
\end{align*}
belong to $\Mor(0,0)$, and the above sums have only finitely many nonzero terms.
\item The elements of $\Mor_{a-}(n,m)$ and $\Mor_{-b}(n,m)$ define bounded $\ell^\infty(I)$-bimodular operators from $\ell^2(I^{m+1})$ to $\ell^2(I^{n+1})$.
\item For every $T \in \Mor_{a-}(n,m)$, the range projection $P$ of $T$ belongs to $\Mor_{a-}(n,n)$ and $P \, \Mor_{a-}(n,n) \, P$ is a finite dimensional C$^*$-algebra. A similar statement holds for $T \in \Mor_{-b}(n,m)$.
\item We have that $\Mor_{a-b}(n,m) \subset \Mor_{a-}(n,m)$ and $\Mor_{a-b}(n,m) \subset \Mor_{-b}(n,m)$.
\item We have $\Mor_{a-b}(n,m)^* = \Mor_{a-b}(m,n)$ and $\Mor_{a-b}(n,k) \Mor_{a-b}(k,m) \subset \Mor_{a-b}(n,m)$. We also have that
$$\Mor_{a-b}(n,m) \ot_{I_b} \Mor_{b-c}(n',m') \subset \Mor_{a-c}(n+n',m+m') \; .$$
\end{enumlist}

For every $i \in I^{n+1}$, we define $\ibar \in I^{n+1}$ as $\ibar = (i_n, \ldots, i_0)$. For every $I^{n+1} \times I^{m+1}$ matrix $T$, we denote by $\Ttil$ the $I^{m+1} \times I^{n+1}$ matrix given by $\Ttil_{ij} = T_{\jbar \, \ibar}$.

\begin{enumlist}[resume]
\item For every $T \in \Mor_{a-b}(n,m)$, we have that $\Ttil \in \Mor_{b-a}(m,n)$.

\item Let $P \in \Mor_{a-b}(n,n)$ be a self-adjoint projection. Then $\Ptil$ is a self-adjoint projection in $\Mor_{b-a}(n,n)$. Defining
\begin{equation}\label{eq.adjoint-object}
\begin{split}
& s_{i_0 \cdots i_{2n},j} = \begin{cases} P_{i_0\cdots i_n, i_{2n} i_{2n-1} \cdots i_n} &\quad\text{if $i_0 = i_{2n} = j$,}\\ 0 &\quad\text{otherwise,}\end{cases} \\
& t_{i_0 \cdots i_{2n},j} = \begin{cases} P_{i_n\cdots i_{2n},i_n i_{n-1} \cdots i_0} &\quad\text{if $i_0 = i_{2n} = j$,}\\ 0 &\quad\text{otherwise,}\end{cases}
\end{split}
\end{equation}
we have $s \in \Mor_{a-a}(2n,0)$ with $ss^* \leq P \ot_{I_b} \Ptil$, $t \in \Mor_{b-b}(2n,0)$ with $tt^* \leq \Ptil \ot_{I_a} P$ and
$$(s^* \ot_{I_a} 1) (1 \ot_{I_b} t) = P \quad\text{and}\quad (t^* \ot_{I_b} 1)(1 \ot_{I_a} s) = \Ptil \; .$$

\item The formulas
\begin{align*}
& \Tr_\ell : \Mor_{a-b}(n,n) \to \C : \Tr_\ell(T) = (\id \ot \Tr_{n})(T)_i \quad\text{for all $i \in I_a$,}\\
& \Tr_r : \Mor_{a-b}(n,n) \to \C : \Tr_r(T) = (\Tr_{n} \ot \id)(T)_i \quad\text{for all $i \in I_b$}
\end{align*}
provide well-defined tracial functionals on $\Mor_{a-b}(n,n)$. If $P \in \Mor_{a-b}(n,n)$ is a self-adjoint projection and if $s \in \Mor_{a-a}(2n,0)$ and $t \in \Mor_{b-b}(2n,0)$ are defined as in 7, then for all $T \in P \Mor_{a-b}(n,n) P$,
$$\Tr_\ell(T) = (s^* (T \ot 1^{\ot n}) s)_i \;\;\text{for all $i \in I_a$, and}\quad \Tr_r(T) = (t^* (1^{\ot n} \ot T) t)_i \;\;\text{for all $i \in I_b$.}$$
\end{enumlist}
\end{lemma}

Once Lemma \ref{lem.crucial-technical-Mor-lemma} is proven, we can define the unitary $2$-category $\cC(\cD,\Pi)$ in the following way.

\begin{definition}\label{def.my-unitary-2-category}
Let $\Pi$ be a connected locally finite graph. Let $\cD$ be a graph category containing all planar bi-labeled graphs. We define $\cC(\cD,\Pi)$ as the unitary $2$-category with $\cE$ as the set of $0$-cells.
For all $a,b \in \cE$, the set of $1$-cells $\cC_{a-b}(\cD,\Pi)$ consists of the (formal finite direct sums of) self-adjoint projections $P$ in $\Mor_{a-b}(n,n)$ with $n \geq 0$. The space of morphisms between the self-adjoint projections $P \in \Mor_{a-b}(n,n)$ and $Q \in \Mor_{a-b}(m,m)$ is defined as $P \Mor_{a-b}(n,m) Q$. The tensor product of the self-adjoint projections $P \in \Mor_{a-b}(n,n)$ and $Q \in \Mor_{b-c}(m,m)$ is defined as $P \ot_{I_b} Q$ and the tensor product of morphisms is defined similarly.
\end{definition}

Note that we can identify a self-adjoint projection $P \in \Mor_{a-b}(n,n)$ with its range, which is a finite type $\ell^\infty(I_a)$-$\ell^\infty(I_b)$-subbimodule of $\ell^2(I^{n+1})$. In this way, the tensor product in Definition \ref{def.my-unitary-2-category} is compatible with the tensor product of bimodules and $\ell^\infty(I)$-bimodular operators discussed above.

By \eqref{eq.adjoint-object}, every object $\cH \in \cC_{a-b}(\cD,\Pi)$ has a conjugate $\overline{\cH} \in \cC_{b-a}(\cD,\Pi)$. Lemma \ref{lem.crucial-technical-Mor-lemma} provides all needed properties to check that $\cC(\cD,\Pi)$ is a well-defined unitary $2$-category.

Also note that the morphisms $s$ and $t$ in \eqref{eq.adjoint-object} are solutions for the conjugate equations but need not be standard solutions, see Definition \ref{def.dimensions} and the comments following it.

\begin{proof}[{Proof of Lemma \ref{lem.crucial-technical-Mor-lemma}}]
1.\ By symmetry and linearity, it suffices to consider $(\id \ot \Tr_n)(T^\cK)$ for $\cK \in \cL(n,n)$. Since $\cK$ is connected, for every $i \in I$, in the sum defining $(\id \ot \Tr_n)(T^\cK)_i$, there are only finitely many nonzero terms. Define the planar bi-labeled graph $\cR \in \cP(2n,1)$ by
$$\cR = (1^{\ot n} \ot \cM^{1,0} \ot 1^{\ot (n-1)}) \circ (1^{\ot (n-1)} \ot \cM^{2,0} \ot 1^{\ot (n-2)}) \circ \cdots \circ (1 \ot \cM^{2,0}) \; ,$$
as illustrated in Figure \ref{fig.R-with-n-3}.
\begin{figure}[h]
    \centering
    \begin{tikzpicture}
        \filldraw (0,0) circle (1.5pt) node[anchor=south]{$x_0$} node[anchor=north]{$y_0$};
        \filldraw (1.4,0) circle (1.5pt) node[anchor=south]{$x_1$} node[anchor=north]{$x_5$};
        \filldraw (2.8,0) circle (1.5pt) node[anchor=south]{$x_2$} node[anchor=north]{$x_4$};
        \filldraw (4.2,0) circle (1.5pt) node[anchor=west]{$x_3$};
    \end{tikzpicture}
    \caption{The bi-labeled graph $\cR \in \cP(2n,1)$ for $n=3$}\label{fig.R-with-n-3}
\end{figure}

Then, $\cK' = \cR^* \circ (\cK \ot 1^{\ot (n-1)}) \circ \cR$ belongs to $\cD(1,1)$. By construction, $\cK' \in \cL(0,0)$ and $T^{\cK'} = (\id \ot \Tr_n)(T^\cK)$.

2.\ Take $T \in \Mor_{a-}(n,m)$ and write $T = (1_a \ot 1^{\ot n})S$ with $S \in \Mor(n,m)$. For every $i \in I$, we have that $(\id \ot \Tr_n)(SS^*)_i < +\infty$. By 1 and the definition of $\approx$, the function
$$i \mapsto (\id \ot \Tr_n)(SS^*)_i$$
is constant on $I_a$ and we denote this constant by $\kappa_a$. Thus, $(\id \ot \Tr_n)(TT^*)_i$ equals $\kappa_a$ when $i \in I_a$ and equals $0$ otherwise. Viewing $T$ as the direct sum of the operators $T_i = (p_i \ot 1) T$ acting on $\ell^2(I^n)$ and using that the operator norm is bounded above by the Hilbert-Schmidt norm, we get that
$$\|T\|^2 = \sup_{i \in I} \|T_i\|^2 \leq \sup_{i \in I} \Tr_n(T_i T_i^*) = \kappa_a \; .$$
So, $T$ is a bounded operator. By symmetry, also the elements of $\Mor_{-b}(n,m)$ define bounded operators.

3.\ We first prove that for every $i \in I_a$, the map
$$\Mor_{a-}(n,m) \to B(\ell^2(I^{m+1}),(p_i \ot 1^{\ot n}) \ell^2(I^{n+1}))  : T \mapsto (p_i \ot 1^{\ot n}) T$$
is injective. If $T \in \Mor_{a-}(n,m)$, it follows from the previous paragraph that the map $i \mapsto (\id \ot \Tr_n)(TT^*)_{i}$ is constant on $I_a$. So if $(p_i \ot 1^{\ot n}) T = 0$, we get that $(\id \ot \Tr_n)(TT^*) = 0$. Because $\Tr_n$ is the usual trace on bounded operators, which is faithful, we conclude that $T = 0$.

Then fix $T \in \Mor_{a-}(n,m)$ and write $T = (1_a \ot 1^{\ot n})S$ with $S \in \Mor(n,m)$. Denote by $P \in B(\ell^2(I^{n+1}))$ the range projection of $T$. Fix $i \in I_a$. Since the bi-labeled graphs in $\cL(n,m)$ are assumed to be connected, the range of $(p_i \ot 1^{\ot n}) S = (p_i \ot 1^{\ot n})T$ is a finite dimensional subspace $\cH_0$ of $\ell^2(I^{n+1})$. By the claim above, the map
$$T \, \Mor_{a-}(m,m) \, T^* \to B(\cH_0) : X \mapsto X (p_i \ot 1^{\ot n})$$
is an injective $*$-homomorphism. Since $\cH_0$ is finite dimensional, $T \, \Mor_{a-}(m,m) \, T^*$ is a finite dimensional C$^*$-algebra for all $T \in \Mor_{a-}(n,m)$. Since $P$ belongs to this C$^*$-algebra, we get in particular that $P \in \Mor_{a-}(n,n)$. Applying the previous statement to $P$ instead of $T$, we then also get that $P \, \Mor_{a-}(n,n) \, P$ is a finite dimensional C$^*$-algebra.

4.\ Fix $a,b \in \cE$ and fix $T \in \Mor_{a-}(n,m)$. By symmetry, it suffices to prove that $(1^{\ot n} \ot 1_b) T \in \Mor_{a-}(n,m)$. Fix $i \in I_a$. Since the range of $(p_i \ot 1^{\ot n})T$ is contained in $\ell^2(W)$ for some finite subset $W \subset I^{n+1}$, the set
$$\cE_0 = \{c \in \cE \mid (p_i \ot 1^{\ot (n-1)} \ot 1_c) T \neq 0 \}$$
is finite. Recall that $\Mor(0,0)$ is a $*$-algebra of functions on $I$ containing the constant function $1$ and recall that by definition of $\approx$, for all $c \in \cE \setminus \{b\}$, there exists an element $S \in \Mor(0,0)$ such that $S_i \neq S_j$ when $i \in I_c$ and $j \in I_b$. Combining both, we can choose $S \in \Mor(0,0)$ such that $S 1_b = 1_b$ and $S 1_c = 0$ for all $c$ in the finite set $\cE_0 \setminus \{b\}$. Write $R = (1^{\ot n} \ot S) T$. Then, $R \in \Mor_{a-}(n,m)$.

We claim that $R = (1^{\ot n} \ot 1_b) T$. First note that $(1^{\ot n} \ot 1_b) R = (1^{\ot n} \ot 1_b S) T = (1^{\ot n} \ot 1_b) T$. Then fix $d \in \cE \setminus \{b\}$. To prove the claim, it suffices to prove that $(1^{\ot n} \ot 1_d)R = 0$. As in the previous paragraph, we can choose $F \in \Mor(0,0)$ such that $F 1_d = 1_d$ and $F 1_b = 0$. Since $S 1_b = 1_b$ and $S 1_c = 0$ for all $c \in \cE_0 \setminus \{b\}$, it follows from the definition of $\cE_0$ that
\begin{align*}
(p_i \ot 1^{\ot n})(1^{\ot n} \ot F) R &= (1^{\ot n} \ot F S) (p_i \ot 1^{\ot n})T = \sum_{c \in \cE_0} (1^{\ot n} \ot FS 1_c) (p_i \ot 1)T \\
&= (1^{\ot n} \ot F 1_b) (p_i \ot 1)T = 0 \; .
\end{align*}
Since $N \mapsto (p_i \ot 1^{\ot n}) N$ is injective on $\Mor_{a-}(n,m)$ and $(1^{\ot n} \ot F) R$ belongs to $\Mor_{a-}(n,m)$, we conclude that $(1^{\ot n} \ot F) R = 0$. Then also
$$(1^{\ot n} \ot 1_d) R = (1^{\ot n} \ot 1_d F) R =  0 \; .$$
So the claim is proven. Since $R \in \Mor_{a-}(n,m)$, we have shown that $(1^{\ot n} \ot 1_b) T \in \Mor_{a-}(n,m)$.

5.\ The first two properties are immediate. Using 4, we get that
\begin{align*}
\Mor_{a-b}(n,m) \ot_{I_b} \Mor_{b-c}(n',m') & \subset \Mor_{a-}(n,m) \ot_I \Mor_{-c}(n',m') \\
&= (\Mor(n,m) \ot_I \Mor(n',m'))(1_a \ot 1^{\ot (m+m'-1)} \ot 1_c) \\
& \subset \Mor_{a-c}(n+n',m+m') \; .
\end{align*}

6.\ For every $n \geq 1$, define the planar bi-labeled graph $\cR_n \in \cP(2n,0)$ by
$$\cR_n = (1^{\ot (n-1)} \ot \cM^{2,0} \ot 1^{\ot (n-1)}) \circ (1^{\ot (n-2)} \ot \cM^{2,0} \ot 1^{\ot (n-2)}) \circ \cdots \circ \cM^{2,0} \; ,$$
as illustrated in Figure \ref{fig.Rn-with-n-3}.
\begin{figure}[h]
    \centering
    \begin{tikzpicture}
        \filldraw (0,0) circle (1.5pt) node[anchor=south]{$x_0$} node[anchor=north]{$x_5$};
        \filldraw (1.4,0) circle (1.5pt) node[anchor=south]{$x_1$} node[anchor=north]{$x_4$};
        \filldraw (2.8,0) circle (1.5pt) node[anchor=south]{$x_2$} node[anchor=north]{$x_3$};
    \end{tikzpicture}
    \caption{The bi-labeled graph $\cR_n \in \cP(2n,0)$ for $n=3$}\label{fig.Rn-with-n-3}
\end{figure}

Whenever $\cK = (K,x,y) \in \cG(n,m)$, the bi-labeled graph
$$\cKtil = (\cR_n^* \ot 1^{\ot m}) \circ (1^{\ot n} \ot \cK \ot 1^{\ot m}) \circ (1^{\ot n} \ot \cR_m)$$
is isomorphic with $(K,\overline{y},\overline{x})$. In particular, $T^{\cKtil} = \widetilde{T^\cK}$. Since $\cP \subset \cD$, it follows that for all $\cK \in \cL(n,m)$, $\cKtil \in \cL(m,n)$. Thus, for every $T \in \Mor(n,m)$, the matrix $\Ttil$ belongs to $\Mor(m,n)$.

7.\ Define the planar bi-labeled graph $\cS_n \in \cP(2n+1,1)$ by
\begin{multline*}
\cS_n = (1^{\ot n} \ot \cM^{1,0} \ot 1^{\ot n}) \circ (1^{\ot (n-1)} \ot \cM^{2,0} \ot 1^{\ot (n-1)}) \circ (1^{\ot (n-2)} \ot \cM^{2,0} \ot 1^{\ot (n-2)}) \circ \cdots \\ \circ (1 \ot \cM^{2,0} \ot 1) \circ \cM^{2,1} \; .,
\end{multline*}
as illustrated in Figure \ref{fig.Sn-with-n-3}.
\begin{figure}[h]
    \centering
    \begin{tikzpicture}
        \filldraw (0,0) circle (1.5pt) node[anchor=south]{$x_0$} node[anchor=north east]{$y_0$} node[anchor=north west]{$x_7$};
        \filldraw (1.4,0) circle (1.5pt) node[anchor=south]{$x_1$} node[anchor=north]{$x_6$};
        \filldraw (2.8,0) circle (1.5pt) node[anchor=south]{$x_2$} node[anchor=north]{$x_5$};
        \filldraw (4.2,0) circle (1.5pt) node[anchor=south]{$x_3$} node[anchor=north]{$x_4$};
    \end{tikzpicture}
    \caption{The bi-labeled graph $\cS_n \in \cP(2n+1,1)$ for $n=3$}\label{fig.Sn-with-n-3}
\end{figure}

Whenever $\cK \in \cL(n,n)$, we find that $(\cK \ot 1^{\ot n}) \circ \cS_n \in \cL(2n,0)$. The $ij$ matrix coefficient of the corresponding element in $\Mor(2n,0)$ equals $T^\cK_{i_0\cdots i_n, i_{2n} i_{2n-1} \cdots i_n}$ if $i_0 = i_{2n} = j$ and equals $0$ otherwise. It thus follows that the element $s$ in \eqref{eq.adjoint-object} belongs to $\Mor_{a-a}(2n,0)$. We similarly get that $t \in \Mor_{b-b}(2n,0)$. The remaining formulas for $s$ and $t$ follow from a direct computation.

8.\ By 1 and 4, for every $T \in \Mor_{a-b}(n,n)$, the function $i \mapsto (\id \ot \Tr_n)(T)_i$ is constant on $I_a$. Fix $i \in I_a$. Since the elements of $\Mor_{a-b}(n,n)$ are $\ell^\infty(I)$-bimodular, we can view $T \mapsto T(p_i \ot 1^{\ot n})$ as a $*$-homomorphism from $\Mor_{a-b}(n,n)$ to the space of finite rank operators on $\ell^2(I^n)$. Then $\Tr_\ell$ is the composition of this $*$-homomorphism with the usual trace. Thus, $\Tr_\ell$ is a trace on $\Mor_{a-b}(n,n)$. Similarly, $\Tr_r$ is a trace on $\Mor_{a-b}(n,n)$. The remaining formulas follow from a direct computation.
\end{proof}

\begin{definition}\label{def.dimensions}
For every minimal projection $P \in \Mor_{a-b}(n,n)$, thus representing an irreducible object in $\cC_{a-b}(\cD,\Pi)$, we define the left and right dimension as
$$d_\ell(P) = \Tr_\ell(P) \quad\text{and}\quad d_r(P) = \Tr_r(P) \; .$$
The categorical dimension of $P$, in the unitary $2$-category $\cC(\cD,\Pi)$, is thus given by $d(P) = \sqrt{d_\ell(P) \, d_r(P)}$.
We also define $\rho(P) = d_r(P) d_\ell(P)^{-1}$.
\end{definition}

For later use, we record two formulas for the ratio $\rho(P)$ between the left and right dimension. These will be used to give an explicit formula for the modular element of the quantum automorphism group of $\Pi$, defined as the Radon-Nikodym derivative between the left and right Haar functionals.

\begin{lemma}\label{def.dim-proportion-tensor-product}
Let $P \in \Mor_{a-b}(n,n)$, $Q \in \Mor_{b-c}(m,m)$ and $R \in \Mor_{a-c}(k,k)$ be minimal projections such that $(P \ot_{I_b} Q) \Mor_{a-c}(n+m,k) R \neq \{0\}$, meaning that $R$ is isomorphic with a subobject of $P \ot_{I_b} Q$ in $\cC_{a-c}(\cD,\Pi)$. Then, $\rho(R) = \rho(P) \, \rho(Q)$.
\end{lemma}
\begin{proof}
Since $R$ is a minimal projection, we can choose $V \in (P \ot_I Q) \Mor_{a-c}(n+m,k) R$ such that $V^* V = R$. Take $s_P \in (P \ot_I \widetilde{P}) \Mor_{a-a}(2n,0)$ and $t_P \in (\widetilde{P} \ot_I P) \Mor_{b-b}(2n,0)$ as in Lemma \ref{lem.crucial-technical-Mor-lemma}. Similarly choose $s_Q$ and $t_Q$. Write $\lambda = \rho(P) \, \rho(Q)$.

By point 7 of Lemma \ref{lem.crucial-technical-Mor-lemma}, $\rho(P)^{1/4} s_P$ and $\rho(P)^{-1/4} t_P$ form a standard solution for the conjugate equations for $P$. Similarly, $\rho(Q)^{1/4} s_Q$ and $\rho(Q)^{-1/4} t_Q$ form a standard solution for the conjugate equations for $Q$. Write $s = (1 \ot_I s_Q \ot_I 1) s_P$ and $t = (1 \ot_I t_P \ot_I 1) t_Q$. Then, $\lambda^{1/4} s$ and $\lambda^{-1/4} t$ form a standard solution for the conjugate equations of $S=P \ot_I Q$. Fix $i \in I_a$. The categorical trace on $S \, \Mor_{a-c}(n+m,n+m) \, S$ is thus given by $T \mapsto \lambda^{1/2} (s^*(T \ot_I 1)s)_i$. Therefore, we get that
\begin{align*}
(d_\ell(R) \, d_r(R))^{1/2} &= d(R) = \lambda^{1/2} (s^*(VV^* \ot_I 1) s)_i = \lambda^{1/2} \Tr_\ell(VV^*) \\
& = \lambda^{1/2} \Tr_\ell(V^*V) = \lambda^{1/2} \, d_\ell(R) \; .
\end{align*}
So, $\rho(R) = d_r(R)/d_\ell(R) = \lambda$.
\end{proof}

Since $\Mor_{a-b}(1,1) \subset \ell^\infty(I_a \times I_b)$, we can identify projections in $\Mor_{a-b}(1,1)$ with subsets of $I_a \times I_b$. For every subset $W \subset I_a \times I_b$, we denote by $1_W$ its indicator function.

\begin{lemma}\label{def.formula-for-rho}
Up to multiplication by a positive scalar, there exists a unique function $\mu : I \to (0,+\infty)$ such that for all $a,b \in \cE$, all minimal projections $1_W \in \Mor_{a-b}(1,1)$ and all $(i,j) \in W$, we have $\rho(1_W) = \mu_j / \mu_i$.

If $P \in \Mor_{a-b}(n,n)$ is a minimal projection, then the range of $P$ is contained in $\ell^2(W)$, where $W = \bigl\{i \in I^{n+1} \bigm| i_0 \in I_a, i_n \in I_b, \rho(P) = \mu_{i_n}/\mu_{i_0}\bigr\}$.
\end{lemma}
\begin{proof}
For every $(i,j) \in I \times I$, there exist unique $a,b \in \cE$ such that $i \in I_a$ and $j \in I_b$, and there then exists a unique minimal projection $1_W \in \Mor_{a-b}(1,1)$ with $(i,j) \in W$. We define $C(i,j) := \rho(1_W)$.

Fix $i,k,j \in I$. We claim that $C(i,k) \, C(k,j) = C(i,j)$. Take the unique $a,b,c \in \cE$ and minimal projections $1_W \in \Mor_{a-b}(1,1)$ and $1_{W'} \in \Mor_{b-c}(1,1)$ such that $(i,k) \in W$ and $(k,j) \in W'$. Let $1_V \in \Mor_{a-c}(1,1)$ be the unique minimal projection with $(i,j) \in V$. Define the planar bi-labeled graph $\cK = (K,x,y) \in \cP(3,2)$ by $V(K) = \{0,1,2\}$, $E(K) = \emptyset$, $x_i = i$ for $i \in \{0,1,2\}$ and $y_0=0$, $y_1 = 2$. For all connected $\cK_1 \in \cP_c(3,3)$ and $\cK_2 \in \cP_c(2,2)$, we have that $\cK_1 \circ \cK \circ \cK_2$ is connected. Therefore, $(1_W \ot_I 1_{W'}) T^\cK 1_V$ belongs to $\Mor_{a-c}(2,1)$ and it is nonzero, because the $(ikj,ij)$ entry of this matrix equals $1$. By Lemma \ref{def.dim-proportion-tensor-product}, we find that $\rho(1_V) = \rho(1_W) \, \rho(1_{W'})$. This proves the claim.

By the claim above, $C(k,j) = C(i,j) C(i,k)^{-1}$ for all $i,k,j \in I$. Fixing a vertex $e \in I$, we define $\mu_s = C(e,s)$ for all $s \in I$ and get that $C(k,j) = \mu_j \mu_k^{-1}$ for all $j,k \in I$. Uniqueness of $\mu$ up to a scalar multiple is obvious.

We can partition $I \times I$ into subsets $W_j \subset I \times I$ such that $W_j \subset I_{a_j} \times I_{b_j}$ and such that $1_{W_j}$ is a minimal projection in $\Mor_{a_j-b_j}(1,1)$. Let $P \in \Mor_{a-b}(n,n)$ be a minimal projection. Define $X = \{i \in I_a \times I^{n-1} \times I_b \mid P(i) \neq 0\}$, where $(i) \in \ell^2(I^{n+1})$ denotes the basis vector that is equal to $1$ on $i$ and equal to $0$ elsewhere. By definition, the range of $P$ is contained in $\ell^2(X)$. Take $i \in X$. Take $j_1,\ldots,j_n$ such that $(i_{k-1},i_k) \in W_{j_k}$ for all $k \in \{1,\ldots,n\}$. Since $i_0 \in I_a$ and $i_n \in I_b$, we have $a_{j_1} = a$ and $b_{j_n} = b$. Thus, $Q := 1_{W_{j_1}} \ot_I \cdots \ot_I 1_{W_{j_n}}$ is a projection in $\Mor_{a-b}(n,n)$ and, by definition, $Q(i) = (i)$. Therefore $P Q \neq 0$. By Lemma \ref{def.dim-proportion-tensor-product}, we get that
$$\rho(P) = \rho(1_{W_{j_1}}) \cdots \rho(1_{W_{j_n}}) = C(i_0,i_1) \cdots C(i_{n-1},i_n) = \frac{\mu_{i_1}}{\mu_{i_0}} \cdots \frac{\mu_{i_n}}{\mu_{i_{n-1}}} = \frac{\mu_{i_n}}{\mu_{i_0}} \; .$$
This concludes the proof of the lemma.
\end{proof}

\begin{remark}\label{rem.path-approach-to-2-category}
We have chosen to take the $\ell^\infty(I)$-bimodules $\ell^2(I^{n+1})$ as basic objects, even though they are not of finite type, so that the actual objects in $\cC_{a-b}(\cD,\Pi)$ have to be defined as strict submodules. This choice will be helpful in constructing the quantum automorphism group of $\Pi$ later in the paper. It is however possible to describe $\cC(\cD,\Pi)$ with spaces of finite paths in $\Pi$ as basic objects. We explain this construction here and it will be useful in Section \ref{sec.quantizations-discrete-groups} to construct fiber functors, and their corresponding compact quantum groups, for $\cC(\cD,\Pi)$.

Define $\cK_n = (K,x,y) \in \cP(n+1,n+1)$ with $V(K) = \{0,\ldots,n\}$, $(i,j) \in E(K)$ iff $|i-j| = 1$ and $x_i = y_i = i$ for all $i \in \{0,\ldots,n\}$, as illustrated in Figure \ref{fig.graph-Kn}.
\begin{figure}[h]
    \centering
    \begin{tikzpicture}
        \filldraw (0,0) circle (1.5pt) node[anchor=south]{$x_0$} node[anchor=north]{$y_0$};
        \filldraw (1.4,0) circle (1.5pt) node[anchor=south]{$x_1$} node[anchor=north]{$y_1$};
        \filldraw (2.8,0) circle (1.5pt) node[anchor=south]{$x_2$} node[anchor=north]{$y_2$};
        \filldraw (4.8,0) circle (1.5pt) node[anchor=south]{$x_{n-1}$} node[anchor=north]{$y_{n-1}$};
        \filldraw (6.2,0) circle (1.5pt) node[anchor=south]{$x_n$} node[anchor=north]{$y_n$};
        \draw (0,0) -- (1.4,0);
        \draw (1.4,0) -- (2.8,0);
        \draw (2.8,0) -- (3.4,0);
        \draw (4.2,0) -- (4.8,0);
        \draw (4.8,0) -- (6.2,0);
        \draw (3.8,0) node{$\cdots$};
    \end{tikzpicture}
    \caption{The bi-labeled graph $\cK_n = (K,x,y) \in \cP(n+1,n+1)$}\label{fig.graph-Kn}
\end{figure}

Then, $P_{n,a,b} := T^{\cK_n}(1_a \ot 1^{\ot (n-1)} \ot 1_b) \in \Mor_{a-b}(n,n)$ is the orthogonal projection of $\ell^2(I_a \times I^{n-1} \times I_b)$ onto the $\ell^2$-space of the set of paths of length $n$ from a vertex in $I_a$ to a vertex in $I_b$. We can view this as an object $P_{n,a,b} \in \cC_{a-b}(\cD,\Pi)$. By concatenation of paths, $P_{n,a,b} \ot_{I_b} P_{m,b,c}$ embeds into $P_{n+m,a,c}$ and the latter is isomorphic with the direct sum of all $P_{n,a,b} \ot_{I_b} P_{m,b,c}$, $b \in \cE$. By reversion of paths, we get that $\overline{P_{n,a,b}} = P_{n,b,a}$.

For every $n$-tuple $d = (d_1,\ldots,d_n)$ of integers $d_i \geq 0$, write $D = d_1 + \cdots + d_n$ and define the planar bi-labeled graph $\cK_d = (K,x,y) \in \cP(D+1,n+1)$ where $V(K) = \{0,\ldots,D\}$, $(i,j) \in E(K)$ iff $|i-j| = 1$, $x_i = i$ for all $i \in \{0,\ldots,D\}$, while $y_0 = 0$ and $y_i = d_1 + \cdots + d_i$ for all $i \in \{1,\ldots,n\}$, as illustrated in Figure \ref{fig.graph-Kd}.
\begin{figure}[h]
    \centering
    \begin{tikzpicture}
        \filldraw (0,0) circle (1.5pt) node[anchor=south]{$x_0$} node[anchor=north]{$y_0$};
        \filldraw (1,0) circle (1.5pt) node[anchor=south]{$x_1$};
        \filldraw (2,0) circle (1.5pt) node[anchor=south]{$x_2$};
        \filldraw (3,0) circle (1.5pt) node[anchor=south]{$x_3$} node[anchor=north]{$y_1$};
        \filldraw (4,0) circle (1.5pt) node[anchor=south]{$x_4$};
        \filldraw (5,0) circle (1.5pt) node[anchor=south]{$x_5$} node[anchor=north]{$y_2$};
        \filldraw (6,0) circle (1.5pt) node[anchor=south]{$x_6$};
        \filldraw (7,0) circle (1.5pt) node[anchor=south]{$x_7$};
        \filldraw (8,0) circle (1.5pt) node[anchor=south]{$x_8$};
        \filldraw (9,0) circle (1.5pt) node[anchor=south]{$x_9$} node[anchor=north]{$y_3$};
        \draw (0,0) -- (1,0);
        \draw (1,0) -- (2,0);
        \draw (2,0) -- (3,0);
        \draw (3,0) -- (4,0);
        \draw (4,0) -- (5,0);
        \draw (5,0) -- (6,0);
        \draw (6,0) -- (7,0);
        \draw (7,0) -- (8,0);
        \draw (8,0) -- (9,0);
    \end{tikzpicture}
    \caption{The bi-labeled graph $\cK_d = (K,x,y) \in \cP(10,4)$ with $d = (3,2,4)$}\label{fig.graph-Kd}
\end{figure}

Then, $T^{\cK_d} \in P_D \Mor(D,n)$ and, for all $a,b \in \cE$, the right support of $T^{\cK_d} (1_a \ot 1^{\ot (n-1)} \ot 1_b)$ equals $\ell^2(W)$ with $i \in W$ iff $i_0 \in I_a$, $i_n \in I_b$ and for all $k \in \{1,\ldots,n\}$, there exists in $\Pi$ a path of length $d_k$ from $i_{k-1}$ to $i_k$. By varying $d_1,\ldots,d_n$, we conclude that every irreducible object in $\cC_{a-b}(\cD,\Pi)$ is contained in one of the $P_{n,a,b}$, $n \geq 0$.

We could thus take the path spaces $P_{n,a,b}$ as basic objects of the unitary $2$-category $\cC(\cD,\Pi)$. By construction, the space of morphisms from $P_{m,a,b}$ to $P_{n,a,b}$ is given by the linear span of the operators $T^{\cK}(1_a \ot 1^{\ot (m-1)} \ot 1_b)$ where $\cK = (K,x,y) \in \cD(n+1,m+1)$ is connected, with $x_0 = y_0$, $x_n = y_m$ and with $x$ and $y$ being paths in $K$.
\end{remark}

\subsection{\boldmath A $*$-algebra associated with the unitary $2$-category $\cC(\cD,\Pi)$}\label{sec.construction-B}

We still fix a connected locally finite graph $\Pi$ with vertex set $I$ and a graph category $\cD$ containing all planar bi-labeled graphs. We denote by $\cC(\cD,\Pi)$ the unitary $2$-category defined in Definition \ref{def.my-unitary-2-category}. By construction, $\cC(\cD,\Pi)$ is a concrete unitary $2$-category of $\ell^\infty(I)$-bimodules of finite type.

In \cite[Proposition 3.12]{DCT14}, an \emph{$I$-partial compact quantum group} is associated to any such unitary $2$-category and this should be viewed as a quantum groupoid with classical unit space $I$. This $I$-partial compact quantum group of \cite{DCT14} is in particular a $*$-algebra $\cB$ with positive Haar functionals. We repeat this construction here in a concrete way, defining $\cB$ by generators and relations, in such a way that in the next section \ref{sec.quantum-aut-group}, we can easily define an algebraic quantum group $(\cA,\Delta)$ and its Haar functionals as a corner of $\cB$.

Denote by $\cUB$ the free vector space with basis $\sqcup_{n \geq 0} (I^{n+1} \times I^{n+1})$. We denote the basis vectors as $F_n(i,j)$ for $n \geq 0$ and $i,j \in I^{n+1}$. We view $F_n$ as an $I^{n+1} \times I^{n+1}$ matrix with entries in $\cUB$.

For every $i \in I^{n+1}$, we define $\ibar \in I^{n+1}$ by $\ibar = (i_n, \cdots,i_0)$. We turn $\cUB$ into a $*$-algebra by defining
\begin{align*}
& F_n(i,j) \, F_m(k,l) = \begin{cases} F_{n+m}(i_0 \cdots i_n k_1 \cdots k_m,j_0 \cdots j_n l_1 \cdots l_m) &\quad\text{if $i_n=k_0$ and $j_n = l_0$,}\\ 0 &\quad\text{otherwise,}\end{cases}\\
& F_n(i,j)^* = F_n(\ibar,\jbar) \; .
\end{align*}
Note that, by definition, $(F_0(i,j))_{(i,j) \in I \times I}$ is a family of mutually orthogonal, self-adjoint idempotents in $\cUB$. Also note that if $I$ is infinite, then $\cUB$ is non-unital.

For the following lemma, recall that $\cD_1(n,m)$ denotes the set of bi-labeled graphs $(K,x,y)$ in $\cD(n,m)$ with the property that every connected component of $K$ intersects $x$ and intersects $y$.
Similarly to the notation $\cL(n,m)$, we define
$$\cL_1(n,m) = \{ (K,x,y) \in \cD_1(n+1,m+1) \mid x_0=y_0 \;\;\text{and}\;\; x_n=y_m\} \; .$$
Recall the equivalence relation $\approx$ on $I$ introduced in \eqref{eq.equiv-rel-on-I}, with equivalence classes $(I_a)_{a \in \cE}$. We denote by $d(i,j)$ the distance in the graph $\Pi$ between vertices $i,j \in I$.

\begin{lemma}\phantomsection\label{lem.ideal-UB}
\begin{enumlist}
\item The linear span
\begin{equation}\label{eq.ideal-I-in-UB}
\cI = \lspan \bigl\{(F_n T^\cK - T^\cK F_m)_{ij} \bigm| n,m \geq 0, \cK \in \cL_1(n,m), i \in I^{n+1}, j \in I^{m+1}\bigr\}
\end{equation}
is a $*$-ideal in $\cUB$.

\item If $n \geq 0$ and $i,j \in I^{n+1}$ such that $i_k \not\approx j_k$ for some $k \in \{0,\ldots,n\}$, then $F_n(i,j) \in \cI$.

\item If $n \geq 0$ and $i,j \in I^{n+1}$ such that $d(i_{k-1},i_k) \neq d(j_{k-1},j_k)$ for some $k \in \{1,\ldots,n\}$, then $F_n(i,j) \in \cI$.

\item For all $n,m \geq 0$, $a,b \in \cE$ and $T \in \Mor_{a-b}(n,m)$, the entries of $F_n T - T F_m$ belong to $\cI$.
\end{enumlist}
\end{lemma}\vspace{0.5ex}

\begin{definition}\label{def.star-algebra-B}
We define $\cB$ as the quotient of $\cUB$ by the $*$-ideal $\cI$ defined in \eqref{eq.ideal-I-in-UB}.
\end{definition}

\begin{proof}[{Proof of Lemma \ref{lem.ideal-UB}}]
1.\ As in the proof of Lemma \ref{lem.crucial-technical-Mor-lemma}.6, we get that for every $\cK = (K,x,y)$ in $\cL_1(n,m)$, the bi-labeled graph $(K,\overline{y},\overline{x})$ belongs to $\cL_1(m,n)$. Taking its transpose, we find that $(K,\overline{x},\overline{y}) \in \cL_1(n,m)$. This operation shows that $\cI^* = \cI$. For every $\cK \in \cL_1(n,m)$ and for every $k \geq 1$, we have that $\cK \ot 1^{\ot k}$ and $1^{\ot k} \ot \cK$ belong to $\cL_1(n+k,m+k)$. It follows that $\cI \, F_k(i,j) \subset \cI$ and $F_k(i,j) \, \cI \subset \cI$ for all $k \geq 0$ and $i,j \in I^{k+1}$. So, $\cI$ is a $*$-ideal in $\cUB$.

2.\ Take $n \geq 0$, $i,j \in I^{n+1}$ and $k \in \{0,\ldots,n\}$ such that $i_k \not\approx j_k$. Take $S \in \Mor(0,0)$ with $S_{i_k} \neq S_{j_k}$. Then,
$$(S_{j_k} - S_{i_k}) F_n(i,j) = \bigl(F_n (1^{\ot k} \ot S \ot 1^{\ot (n-k)}) - (1^{\ot k} \ot S \ot 1^{\ot (n-k)}) F_n\bigr)_{ij} \in \cI \; ,$$
so that $F_n(i,j) \in \cI$.

3.\ Assume that $F_n(i,j) \not\in \cI$ and fix $k \in \{1,\ldots,n\}$. We have to prove that $d(i_{k-1},i_{k}) = d(j_{k-1},j_{k})$. Fix an integer $d \geq 0$. Define $\cK = (K,x,y) \in \cL_1(n,n)$ with $V(K) = \{0,1,\ldots,n+d-1\}$, with $(v,w) \in E(K)$ iff $v,w \in \{k-1,k,\ldots,k+d-1\}$ and $|v-w|=1$, and with $x_s = y_s = s$ if $0 \leq s \leq k-1$ and $x_s = y_s = s+d-1$ if $k \leq s \leq n$.

Then $T^\cK$ is a diagonal matrix with $T^\cK_{ii} \geq 0$ for all $i \in I^{n+1}$ and $T^\cK_{ii} > 0$ iff there exists in $\Pi$ a path of length $d$ from $i_{k-1}$ to $i_{k}$.

Since $(F_n T^\cK - T^\cK F_n)_{ij} \in \cI$ and since we assumed that $F_n(i,j) \not\in \cI$, the following property holds for every $d \geq 0$~: there exists in $\Pi$ a path of length $d$ from $i_{k-1}$ to $i_{k}$ iff there exists in $\Pi$ a path of length $d$ from $j_{k-1}$ to $j_{k}$. Since this holds for all $d \geq 0$, we must have that $d(i_{k-1},i_{k}) = d(j_{k-1},j_{k})$.

4.\ It follows in particular from 2 that the entries of $F_n (1_a \ot 1^{\ot (n-1)} \ot 1_b) - (1_a \ot 1^{\ot (n-1)} \ot 1_b) F_n$ belong to $\cI$. In combination with the definition of $\cI$, statement 4 follows immediately.
\end{proof}

From now on, we also (and mainly) view $F_n$ as an $I^{n+1} \times I^{n+1}$ matrix with entries in $\cB$.

In the next section, we will prove that the underlying multiplier Hopf $*$-algebra $(\cA,\Delta)$ of the quantum automorphism group $\Pi$ can be identified with a corner of $\cB$ (when $\cB$ is defined using $\cD = \cP$). We now construct a faithful positive functional on $\cB$, whose restriction to this corner will turn out to be the left invariant Haar functional on $(\cA,\Delta)$.

As before, we denote by $\cF(I^{n+1}) \subset \ell^2(I^{n+1})$ the subspace of finitely supported functions $I^{n+1} \to \C$. Recall that every $T \in \Mor(n,m)$ defines a linear map from $\cF(I^{m+1})$ to $\cF(I^{n+1})$.

\begin{lemma}\label{lem.ibf}
Let $a,b \in \cE$, $k \geq 0$ and let $P \in \Mor_{a-b}(k,k)$ be a minimal projection, thus realizing an irreducible object in $\cC_{a-b}(\cD,\Pi)$.

For every $n \geq 0$, there exists a subset $\ibf_{a-b}(n,P) \subset \Mor_{a-b}(n,k) P$ with the following properties.
\begin{enumlist}
\item For all $V,W \in \ibf_{a-b}(n,P)$, we have that $W^* V = 0$ if $V \neq W$ and $W^* V = P$ if $V = W$.
\item In $B(\ell^2(I^{n+1}))$, we have that $\sum_{V \in \ibf_{a-b}(n,P)} VV^*$ equals the projection onto the closed linear span of $\Mor_{a-b}(n,k)P \ell^2(I^{k+1})$.
\item For all $\xi \in \cF(I^{n+1})$, there are only finitely many $V \in \ibf_{a-b}(n,P)$ such that $V^* \xi \neq 0$.
\end{enumlist}
\end{lemma}

We read $\ibf$ as ``isometric basis of finite type''.

\begin{proof}
For every $n$-tuple $d = (d_1,\ldots,d_n)$ of integers $d_i \geq 0$, define the planar bi-labeled graph $\cD_d = (K,x,y) \in \cP(n+1,n+1)$ with $V(K) = \{0,1,\ldots,d_1+\cdots+d_n\}$, with $(v,w) \in E(K)$ iff $|v-w| = 1$, and with labeling $x_0 = y_0 = 0$, $x_k = y_k = d_1 + \cdots + d_k$ for all $k \in \{1,\ldots,n\}$, as illustrated in Figure \ref{fig.this-graph-Dd}.
\begin{figure}[h]
    \centering
    \begin{tikzpicture}
        \filldraw (0,0) circle (1.5pt) node[anchor=south]{$x_0$} node[anchor=north]{$y_0$};
        \filldraw (1,0) circle (1.5pt);
        \filldraw (2,0) circle (1.5pt);
        \filldraw (3,0) circle (1.5pt) node[anchor=south]{$x_1$} node[anchor=north]{$y_1$};
        \filldraw (4,0) circle (1.5pt);
        \filldraw (5,0) circle (1.5pt) node[anchor=south]{$x_2$} node[anchor=north]{$y_2$};
        \filldraw (6,0) circle (1.5pt);
        \filldraw (7,0) circle (1.5pt);
        \filldraw (8,0) circle (1.5pt);
        \filldraw (9,0) circle (1.5pt) node[anchor=south]{$x_3$} node[anchor=north]{$y_3$};
        \draw (0,0) -- (1,0);
        \draw (1,0) -- (2,0);
        \draw (2,0) -- (3,0);
        \draw (3,0) -- (4,0);
        \draw (4,0) -- (5,0);
        \draw (5,0) -- (6,0);
        \draw (6,0) -- (7,0);
        \draw (7,0) -- (8,0);
        \draw (8,0) -- (9,0);
    \end{tikzpicture}
    \caption{The bi-labeled graph $\cD_d = (K,x,y) \in \cP(4,4)$ with $d = (3,2,4)$}\label{fig.this-graph-Dd}
\end{figure}

The matrix $D^d := T^{\cD_d}$ is diagonal, $D^d_{ii} \geq 0$ for all $i \in I^{n+1}$ and $D^d_{ii} > 0$ iff for every $k \in \{1,\ldots,n\}$, there exists in $\Pi$ a path of length $d_k$ from $i_{k-1}$ to $i_k$.

For every integer $\lambda \geq 0$, define the subset $W(n,\lambda) \subset I^{n+1}$ consisting of all $i \in I^{n+1}$ such that for all $k \in \{1,\ldots,n\}$, the distance in $\Pi$ from $i_{k-1}$ to $i_k$ is at most $\lambda$. For all $a,b \in \cE$, denote by $Q(a,b,n,\lambda) \in \ell^\infty(I^{n+1})$ the indicator function of the set $\{i \in W(n,\lambda) \mid i_0 \in I_a, i_n \in I_b\}$. Since $Q(a,b,n,\lambda)$ is the range projection of
$$\sum_{d : 0 \leq d_k \leq \lambda} D^d (1_a \ot 1^{\ot (n-1)} \ot 1_b) \; ,$$
it follows from Lemma \ref{lem.crucial-technical-Mor-lemma} that $Q(a,b,n,\lambda) \in \Mor_{a-b}(n,n)$.

Fix $a,b \in \cE$ and fix $n \geq 0$. Write $R_0 = Q(a,b,n,0)$ and $R_s = Q(a,b,n,s) - Q(a,b,n,s-1)$ for all $s \geq 1$. To prove the lemma, it now suffices to choose finite bases of orthogonal isometries for the finite dimensional morphism spaces $R_s \Mor_{a-b}(n,k) P$ and to define $\ibf_{a-b}(n,P)$ as the union of these bases.
\end{proof}

Let $\Irr_{a-b}$ be a choice of representatives for all irreducible objects in $\cC_{a-b}(\cD,\Pi)$ and realize every $\al \in \Irr_{a-b}$ by a minimal projection $P_\al \in \Mor_{a-b}(k_\al,k_\al)$ with $k_\al \geq 0$. For every $\al \in \Irr_{a-b}$ and $n \geq 0$, we write $\ibf_{a-b}(n,\al) := \ibf_{a-b}(n,P_\al)$. We mean by this \emph{any} choice of basis satisfying the conclusion of Lemma \ref{lem.ibf}. We often write sums over $\ibf_{a-b}(n,\al)$ and one checks easily that such sums are independent of the choice of basis.

\begin{lemma}\label{lem.basis-is-complete}
For all $n \geq 0$ and $\xi \in \cF(I^{n+1})$, we have that
\begin{equation}\label{eq.finite-sum-is-nice}
\sum_{a,b \in \cE} \; \sum_{\al \in \Irr_{a-b}} \; \sum_{V \in \ibf_{a-b}(n,\al)} VV^* \xi = \xi
\end{equation}
and the sum on the left has only finitely many nonzero terms.
\end{lemma}

\begin{proof}
It follows immediately from the construction in Lemma \ref{lem.ibf} that
$$\sum_{a,b \in \cE} \; \sum_{\al \in \Irr_{a-b}} \; \sum_{V \in \ibf_{a-b}(n,\al)} VV^* = 1$$
with strong convergence in $B(\ell^2(I^{n+1}))$. We thus only have to prove that the sum in \eqref{eq.finite-sum-is-nice} has, for every fixed $\xi \in \cF(I^{n+1})$, only finitely many nonzero terms.

Since $\xi$ is finitely supported, we can take a finite subset $\cE_0 \subset \cE$ such that $(1_a \ot 1^{\ot (n-1)} \ot 1_b)\xi = 0$ unless $a,b \in \cE_0$. So we only have nonzero terms in \eqref{eq.finite-sum-is-nice} for $a,b \in \cE_0$. Fix $a,b \in \cE_0$. Using the notation introduced in the proof of Lemma \eqref{lem.ibf}, we can take $\lambda \geq 0$ large enough such that $Q(a,b,n,\lambda) \xi = (1_a \ot 1^{\ot (n-1)} \ot 1_b)\xi$. Since $Q(a,b,n,\lambda)$ is an orthogonal projection in $\Mor_{a-b}(n,n)$, the subset $\cJ \subset \Irr_{a-b}$ of $\al$ such that $P_\al \, \Mor_{a-b}(k_\al,n) \, Q(a,b,n,\lambda)$ is nonzero, is finite. When $\al \not\in \cJ$, we have that $V^* \xi = 0$ for all $V \in \Mor_{a-b}(n,\al)$. Once we also fix $\al \in \cJ$, it follows from Lemma \ref{lem.ibf} that there are only finitely many $V \in \ibf_{a-b}(n,\al)$ with $V^* \xi \neq 0$.
\end{proof}

For all $n \geq 0$ and $\xi,\eta \in \cF(I^{n+1})$, we write
\begin{equation}\label{eq.notation-sesquilinear}
F_n(\xi,\eta) = \sum_{i,j \in I^{n+1}} \overline{\xi(i)} \, \eta(j) \, F_n(i,j) \; .
\end{equation}
Then Lemma \ref{lem.ideal-UB}.4 translates to the property $F_n(T \xi,\eta) = F_m(\xi,T^* \eta)$ for all $T \in \Mor_{a-b}(n,m)$, $\xi \in \cF(I^{m+1})$ and $\eta \in \cF(I^{n+1})$.

For all $a,b \in \cE$ and $\al \in \Irr_{a-b}$, define $H_\al \subset \cF(I^{k_\al + 1})$ as $H_\al = P_\al(\cF(I^{k_\al + 1}))$. Define
$$\cB_\al = \lspan\{ F_{k_\al}(\xi,\eta) \mid \xi,\eta \in H_\al \} \; .$$
By Lemma \ref{lem.ibf}, the linear map
$$\theta_\al : \cB \to \overline{H_\al} \otalg H_\al : \theta_\al(F_n(\xi,\eta)) = \sum_{V \in \ibf_{a-b}(n,\al)} (\overline{V^* \xi} \ot V^* \eta)$$
is well-defined.

Viewing $F_n$ as a linear map from $\overline{\cF(I^{n+1})} \otalg \cF(I^{n+1})$ to $\cB$, we define $\pi_\al = F_{k_\al} \circ \theta_\al$. The following properties are now a consequence of Lemma \ref{lem.basis-is-complete}.
\begin{enumlist}
\item We have $\pi_\al \circ \pi_\al = \pi_\al$. Also, $\theta_\al \circ \pi_\beta = 0$ and $\pi_\al \circ \pi_\beta = 0$ whenever $\al \neq \be$. We have $\cB_\al = \pi_\al(\cB)$.
\item The restrictions $\theta_\al : \cB_\al \to \overline{H_\al} \otalg H_\al$ and $F_{k_\al} : \overline{H_\al} \otalg H_\al \to \cB_\al$ are each other's inverse.
\item For every $x \in \cB$, we have that
$$x = \sum_{a,b \in \cE} \; \sum_{\al \in \Irr_{a-b}} \pi_\al(x)$$
and this sum has only finitely many nonzero terms.
\end{enumlist}

\begin{remark}
We have thus in particular found a vector space isomorphism between $\cB$ and the direct sum of the vector spaces $(\overline{H_\al} \otalg H_\al)_{a,b \in \cE, \al \in \Irr_{a-b}}$. It would be possible as well to take this vector space as the starting point to define $\cB$. The weight of the construction would then shift to providing a well-defined product and $*$-operation on this vector space, very much in the spirit of the Tannaka-Krein theorem in \cite{Wor88}.
\end{remark}

For every $a \in \cE$, we represent the identity object $\eps_a$ of $\cC_{a-a}(\cD,\Pi)$ by the minimal projection $1_a \in \Mor_{a-a}(0,0) = \C 1_a$, so that $H_{\eps_a} = \cF(I_a)$. It then follows that the orthogonal projections $(F_0(i,j))_{i,j \in I_a}$ are all nonzero in $\cB$ and form a vector space basis of the $*$-algebra $\cB_{\eps_a}$. So, all projections $F_0(i,j)$, $i \approx j$, are nonzero in $\cB$. They are orthogonal and linearly span the diagonal subalgebra $\cB_0$. We then get the well-defined conditional expectation
$$\pi_0 : \cB \to \cB_0 : \pi_0(x) = \sum_{a \in \cE} \pi_{\eps_a}(x) \; .$$

\begin{definition}\label{def.functional-vphi}
We uniquely define $\vphi_0 : \cB_0 \to \C : \vphi_0(F_0(i,j)) = 1$ whenever $i \approx j$, and then define $\vphi : \cB \to \C : \vphi = \vphi_0 \circ \pi_0$.
\end{definition}

Recall from Definition \ref{def.dimensions}, the left and right dimension of the irreducible objects in $\cC_{a-b}(\cD,\Pi)$.

\begin{proposition}\label{prop.vphi-positive-faithful}
For all $x,y \in \cB$, we have
\begin{align}
& \vphi(x y^*) = \sum_{a,b \in \cE} \; \sum_{\al \in \Irr_{a-b}} \; d_\ell(\al)^{-1} \; \langle \theta_\al(x),\theta_\al(y) \rangle \; ,\label{eq.first-formula}\\
& \vphi(y^* x) = \sum_{a,b \in \cE} \; \sum_{\al \in \Irr_{a-b}} \; d_r(\al)^{-1} \; \langle \theta_\al(x),\theta_\al(y) \rangle \; .\label{eq.second-formula}
\end{align}
In particular, $\vphi$ is a positive faithful functional on $\cB$.

The bijective linear map $\rho : \cB \to \cB : \rho(x) = d_r(\al) d_\ell(\al)^{-1} x$ for all $x \in \cB_\al$ and $\al \in \Irr_{a-b}$ satisfies
$$\rho(xy) = \rho(x) \rho(y) \quad\text{and}\quad \vphi(xy) = \vphi(y \rho(x)) \quad\text{for all $x,y \in \cB$.}$$
\end{proposition}

\begin{proof}
We start by proving \eqref{eq.first-formula}. Let $a,b,c,d \in \cE$, $\al \in \Irr_{a-b}$ and $\be \in \Irr_{c-d}$. By linearity, it suffices to prove \eqref{eq.first-formula} for $x \in \cB_\al$ and $y \in \cB_\beta$. When $a \neq c$ or $b \neq d$, both the left and the right hand side of \eqref{eq.first-formula} are zero. So, we assume that $c=a$ and $d=b$, so that $\be \in \Irr_{a-b}$. Again by linearity, we may assume that $x = F_{k_\al}(\xi,\eta)$ and $y = F_{k_\be}(\xi',\eta')$ with $\xi,\eta \in \cH_\al$ and $\xi',\eta' \in \cH_\beta$.

For every $m \geq 0$, denote by $J_m$ the anti-unitary operator on $\ell^2(I^{m+1})$ given by $(J_m \xi)(i) = \overline{\xi(\ibar)}$ for all $\xi \in \ell^2(I^{m+1})$ and $i \in I^{m+1}$. When $\xi \in \cF(I^{m+1})$ and $\xi' \in \cF(I^{m'+1})$, we define $\xi \ot_I \xi' \in \cF(I^{m+m'+1})$ by $(\xi \ot_I \xi')(i) = \xi(i_0\cdots i_m) \, \xi'(i_m \cdots i_{m+m'})$. Write $n=k_\al + k_\be$. By definition,
\begin{equation}\label{eq.berekening}
\begin{split}
x y^* &= F_{k_\al}(\xi,\eta) \, F_{k_\be}(\xi',\eta')^* = F_{k_\al}(\xi,\eta) \, F_{k_\be}(J_{k_\be}(\xi'),J_{k_\be}(\eta')) \\
& = F_n(\xi \ot_I J_{k_\be}(\xi'),\eta \ot_I J_{k_\be}(\eta')) \; .
\end{split}
\end{equation}
We now use the notations and results from Lemma \ref{lem.crucial-technical-Mor-lemma}. Note that $J_{k_\be} P_\be J_{k_\be} = \widetilde{P_\be}$. By Lemma \ref{lem.crucial-technical-Mor-lemma}, the minimal projection $\widetilde{P_\be}$ in $\Mor_{b-a}(k_\be,k_\be)$ is a representative for the conjugate of $\be$. So, if $\al \neq \be$, we have
$$(P_\al \ot_I \widetilde{P_\be}) \, \Mor_{a-a}(n,0) = \{0\} \; .$$
It follows from \eqref{eq.berekening} that $\pi_0(x y^*) = 0$, so that $\vphi(x y^*)=0$. If $\al \neq \be$, the right hand side of \eqref{eq.first-formula} is obviously $0$.

We finally consider the case where $\be = \al$. Then, $(P_\al \ot_I \widetilde{P_\al}) \, \Mor_{a-a}(n,0)$ is one-dimensional. Lemma \ref{lem.crucial-technical-Mor-lemma} provides the isometry $d_\ell(\al)^{-1/2} s$ in this intertwiner space. Therefore,
$$\pi_0(x y^*) = d_\ell(\al)^{-1} F_0\bigl(s^*(\xi \ot_I J_{k_\al}(\xi')),s^*(\eta \ot_I J_{k_\al}(\eta'))\bigr) \; ,$$
so that
\begin{align*}
\vphi(x y^*) &= d_\ell(\al)^{-1} \sum_{i,j \in I_a} \langle s(i), (\xi \ot_I J_{k_\al}(\xi')) \rangle \, \langle (\eta \ot_I J_{k_\al}(\eta')) , s(j) \rangle \\
&=d_\ell(\al)^{-1} \sum_{i,j \in I^{k_\al+1}} \overline{\xi(i)} \, \xi'(i) \, \eta(j) \, \overline{\eta'(j)} = d_\ell(\al)^{-1} \, \langle \overline{\xi} \ot \eta , \overline{\xi'} \ot \eta' \rangle \; .
\end{align*}
This last expression is equal to the right hand side of \eqref{eq.first-formula}. The formula \eqref{eq.second-formula} is proven analogously.

Since the subspaces $\cB_\al$ with $\al \in \Irr_{a-b}$ and $a,b \in \cE$ are linearly independent and span $\cB$, the linear map $\rho$ is well-defined and bijective. Combining \eqref{eq.first-formula} and \eqref{eq.second-formula}, we find that $\vphi(xy) = \vphi(y \rho(x))$ for all $x,y \in \cB$. Then, for all $x,y,z \in \cB$,
$$\vphi(z \rho(x) \rho(y)) = \vphi(y z \rho(x)) = \vphi(xyz) = \vphi(z \rho(xy)) \; .$$
Since $\vphi$ is faithful, it follows that $\rho(x) \rho(y) = \rho(xy)$.
\end{proof}

\begin{lemma}\label{lem.modular-vphi}
Let $\mu : I \to (0,+\infty)$ be the map defined in Lemma \ref{def.formula-for-rho}. Then,
$$\rho(F_n(i,j)) = \mu_{i_n} \, \mu_{i_0}^{-1} \, F_n(i,j) = \mu_{j_n} \, \mu_{j_0}^{-1} \, F_n(i,j) \; .$$
In particular, if $\mu_{i_n} \, \mu_{i_0}^{-1} \neq \mu_{j_n} \, \mu_{j_0}^{-1}$, then $F_n(i,j) = 0$ in $\cB$.
\end{lemma}
\begin{proof}
Since $F_n(i,j) = F_1(i_0 i_1,j_0 j_1) \, F_1(i_1 i_2,j_1 j_2) \, \cdots \, F_1(i_{n-1} i_n, j_{n-1} j_n)$ and since $\rho$ is multiplicative, it suffices to consider the case $n=1$.

If $i_0 \not\approx j_0$ or $i_1 \not\approx j_1$, we have that $F_1(i,j) = 0$ by Lemma \ref{lem.ideal-UB} and there is nothing to prove. So, we may assume that $i_0,j_0 \in I_a$ and $i_1,j_1 \in I_b$ for some $a,b \in \cE$. Take the unique minimal projections $1_W,1_{W'} \in \Mor_{a-b}(1,1)$ such that $(i_0,i_1) \in W$ and $(j_0,j_1) \in W'$. If $W \neq W'$, we find that $F_1(i,j) = F_1(1_W(i),j) = F_1(i,1_W(j)) = 0$ in $\cB$. If $W = W'$, we find that $F_1(i,j) \in \cB_{1_W}$, so that $\rho(F_1(i,j)) = \rho(1_W) \, F_1(i,j)$. By Lemma \ref{def.formula-for-rho}, we get that $\rho(1_W) = \mu_{i_1} \, \mu_{i_0}^{-1}$ and also that $\rho(1_W) = \rho(1_{W'}) = \mu_{j_1} \, \mu_{j_0}^{-1}$.
\end{proof}

Recall that $\cD_2(n,m)$ denotes the set of bi-labeled graphs $(K,x,y) \in \cD(n,m)$ such that every connected component of $K$ intersects $x \cup y$. We denote
$$\cL_2(n,m) = \{(K,x,y) \in \cD_2(n+1,m+1) \mid x_0 = y_0 \;\;\text{and}\;\; x_n=y_m \} \; .$$

\begin{lemma}\label{lem.further-relations}
For every $\cK \in \cL_2(n,m)$, $i \in I^{n+1}$, $j \in I^{m+1}$, the equality
$$\sum_{k \in I^{m+1}} T^\cK_{ik} \, F_m(k,j) = \sum_{k \in I^{n+1}} T^\cK_{kj} \, F_n(i,k)$$
holds in $\cB$ and both sides have only finitely many nonzero terms in $\cB$.
\end{lemma}
\begin{proof}
By Lemma \ref{lem.ideal-UB}.3, both sides of the sum have only finitely many nonzero terms in $\cB$.

Fix $\cK \in \cL_2(n,m)$, $i \in I^{n+1}$ and $j \in I^{m+1}$. As in the proof of Lemma \ref{lem.ibf}, we can take connected $\cK_1 \in \cL(n,n)$ and $\cK_2 \in \cL(m,m)$ such that $T^{\cK_1}$ and $T^{\cK_2}$ are diagonal matrices with $T^{\cK_1}_{ii} > 0$ and $T^{\cK_2}_{jj} > 0$. Then, $\cK_1 \circ \cK \circ \cK_2 \in \cL(n,m)$. Therefore, the following equalities of matrices hold over $\cB$.
\begin{align*}
T^{\cK_1} F_n T^{\cK} T^{\cK_2} &= F_n T^{\cK_1} T^{\cK} T^{\cK_2} = F_n T^{\cK_1 \circ \cK \circ \cK_2} = T^{\cK_1 \circ \cK \circ \cK_2} F_m \\
& = T^{\cK_1} T^{\cK} T^{\cK_2} F_m = T^{\cK_1} T^{\cK} F_m T^{\cK_2} \; .
\end{align*}
Taking the $ij$-component and dividing by $T^{\cK_1}_{ii} \, T^{\cK_2}_{jj}$, the lemma follows.
\end{proof}

\begin{remark}
For all edges $e,f \in E(\Pi)$, we can define the element $E_{e,f} \in \cB$ by $E_{e,f} = F_1(i_0i_1,j_0j_1)$ with $e=(i_0,i_1)$ and $f=(j_0,j_1)$. The elements $E_{e,f}$ generate $\cB$ and $E_{e,f}^* = E_{\overline{e},\overline{f}}$. Also the defining relations for $\cB$ can be written in terms of the generators $E_{e,f}$ and matrices indexed by paths in $\Pi$. Although we do not need this in this paper, we provide the following outline how to do this.

We denote by $P_n \subset I^{n+1}$ the set of all paths of length $n$ in $\Pi$. By convention, $P_0 = I$. We denote by $\cUP$ the free vector space with basis vectors $E_n(i,j)$, $n \geq 0$, $i,j \in P_n$. The same formulas defining the $*$-algebra structure on $\cUB$ also define a $*$-algebra structure on $\cUP$. We define $\cP$ as the quotient of $\cUP$ by the relations $E_n T^\cK = T^\cK E_m$ for all $n,m \geq 0$ and $\cK = (K,x,y)$ in $\cD(n+1,m+1)$ such that $x_0 = y_0$, $x_n = y_m$ and $(x_0,\ldots,x_n)$, $(y_0,\ldots,y_m)$ are paths in $K$. One checks that the well-defined $*$-homomorphism $E_n(i,j) \mapsto F_n(i,j)$ actually is a $*$-isomorphism of $\cP$ onto $\cB$.

When $i,j \in P_n$, we write $e_k = (i_{k-1},i_k)$ and $f_k = (j_{k-1},j_k)$ and note that
$$E_n(i,j) = E_{e_1,f_1} \cdots E_{e_n,f_n} \; .$$
In this way, we have concretely identified $\cB$ with a $*$-algebra generated by the elements $E_{e,f}$ and a set of relations.
\end{remark}

\subsection{Locally compact quantum automorphism groups of connected locally finite graphs}\label{sec.quantum-aut-group}

We still fix a connected locally finite graph $\Pi$ with vertex set $I$ and a graph category $\cD$ containing all planar bi-labeled graphs. In this section, we associate to $\cD$ and $\Pi$ an algebraic quantum group, i.e.\ a multiplier Hopf $*$-algebra $(\cA,\Delta)$ with positive faithful integrals. In the next section, we then prove that $(\cA,\Delta)$ satisfies the universal property of Theorem \ref{thm.lc-qaut} when $\cD = \cP$ is the category of planar bi-labeled graphs.

Denote by $\cUA$ the free vector space with basis $\sqcup_{n \geq 1} (I^n \times I^n)$. We denote the basis vectors of $\cUA$ as $U_n(i,j)$ for all $n \geq 1$ and $i,j \in I^n$. We view $U_n$ as an $I^n \times I^n$ matrix with entries in $\cUA$. We turn $\cUA$ into a $*$-algebra by defining
$$U_n(i,j)^* = U_n(\ibar,\jbar) \quad\text{and}\quad U_n(i,j) \, U_m(i',j') = U_{n+m}(ii',jj') \; .$$

Recall that $\cD_1(n,m)$ denotes the set of bi-labeled graphs $(K,x,y)$ in $\cD(n,m)$ with the property that every connected component of $K$ intersects $x$ and intersects $y$.

\begin{lemma}\label{lem.ideal-UA}
Define $\cI = \lspan\bigl\{ (U_n T^\cK - T^\cK U_m)_{ij} \bigm| n,m \geq 1, \cK \in \cD_1(n,m), i \in I^n, j \in I^m \bigr\}$. Then $\cI$ is a $*$-ideal in $\cUA$.
\end{lemma}
\begin{proof}
The same argument as for Lemma \ref{lem.ideal-UB}.1 works.
\end{proof}

\begin{definition}\label{def.alg-A}
We define the $*$-algebra $\cA$ as the quotient $\cUA / \cI$, where $\cI$ is the $*$-ideal defined in Lemma \ref{lem.ideal-UA}.
\end{definition}

We also view $U_n$ as an $I^n \times I^n$ matrix with entries in $\cA$. We still denote by $d$ the distance in the graph $\Pi$.

\begin{lemma}\phantomsection\label{lem.strict-convergence}
\begin{enumlist}
\item If $n \geq 2$, $i,j \in I^n$ and if there exists a $k \in \{1,\ldots,n-1\}$ such that $d(i_k,i_{k+1}) \neq d(j_k,j_{k+1})$, then $U_n(i,j) = 0$ in $\cA$.

In particular, if $n \geq 1$, $i \in I^n$, $k \in \{1,\ldots,n\}$ and $s \in I$ are fixed, then the set
$$\{j \in I^n \mid j_k = s \;\;\text{and}\;\; U_n(i,j) \neq 0 \;\text{in $\cA$} \}$$
is finite. An analogous statement holds if we fix $j \in I^n$ and one of the $i_k$.
\item Let $n,m \geq 1$ and $\cK = (K,x,y) \in \cD(n,m)$. Assume that every connected component of $K$ intersects $x \cup y$ (i.e.\ $\cK \in \cD_2(n,m)$) and assume that at least one connected component of $K$ intersects $x$ and intersects $y$. Then for all $i \in I^n$ and $j \in I^m$ the equality
\begin{equation}\label{eq.infinite-but-not}
\sum_{k \in I^n} U_n(i,k) T^\cK_{kj} = \sum_{k \in I^m} T^\cK_{ik} U_m(k,j)
\end{equation}
holds in $\cA$ and both sides of the sum only have finitely many terms that are nonzero in $\cA$.
\item If $n \geq 1$ and if $\cK = (K,x,\emptyset) \in \cD(n,0)$ is such that every connected component of $K$ intersects $x$, then
\begin{equation}\label{eq.strict-convergence-in-A}
\sum_{j \in I^n} U_n(i,j) T^\cK_j = T^\cK_i \, 1 \quad\text{strictly.}
\end{equation}
\item The $*$-algebra $\cA$ is nondegenerate: if $x \in \cA$ and $xy = 0$ for all $y \in \cA$, then $x = 0$.
\item The elements $u_{ij} = U_1(i,j)$ for $i,j \in I$ satisfy the relations in Theorem \ref{thm.lc-qaut}.1.
\end{enumlist}
\end{lemma}
\begin{proof}
1.\ The proof is identical to the proof of Lemma \ref{lem.ideal-UB}.3.

2.\ The proof is identical to the proof of Lemma \ref{lem.further-relations}.

3.\ This follows from 2 because the bi-labeled graphs $\cK \ot 1$ and $1 \ot \cK$ satisfy the assumption of 2. Since we have not yet shown that $\cA$ is nondegenerate, we have to interpret \eqref{eq.strict-convergence-in-A} in the following way: if we fix $a \in \cA$ and if we multiply the terms in \eqref{eq.strict-convergence-in-A} either on the left or on the right by $a$, there remain only finitely many nonzero terms and their sum equals $T^\cK_i \, a$.

4.\ Assume that $x \in \cA$ is such that $xy = 0$ for all $y \in \cA$. Applying 3 to $\cK = \cM^{1,0}$ and fixing $i \in I$, we find that $\sum_{j \in I} x \, U_1(i,j) = x$. But the left hand side is zero. So, $x = 0$. Now that we have proven that $\cA$ is nondegenerate, we can view $\cA$ as a $*$-subalgebra of $M(\cA)$ and we can interpret \eqref{eq.strict-convergence-in-A} as strict convergence in $M(\cA)$.

5.\ This follows by applying 3 to $\cM^{1,0}$ and $\cM^{0,1}$, and by applying the defining relations of $\cA$ to the connected planar bi-labeled graphs $\cM^{2,1}$, $\cM^{1,2}$ and $\cK = (K,x,y) \in \cP_c(1,1)$ given by $V(K) = \{1,2\}$, $E(K) = \{(1,2),(2,1)\}$, $x_1 = 1$ and $y_1 = 2$.
\end{proof}

\begin{proposition}\phantomsection\label{prop.A-mult-hopf}
\begin{enumlist}
\item There is a unique $*$-homomorphism $\Delta : \cA \to M(\cA \ot \cA)$ satisfying
\begin{equation}\label{eq.def-comult-A}
\Delta(U_n(i,j)) = \sum_{k \in I^n} (U_n(i,k) \ot U_n(k,j)) \quad\text{strictly, for all $n \geq 1$, $i,j \in I^n$.}
\end{equation}
\item There is a unique $*$-anti-automorphism $S : \cA \to \cA$ and a unique $*$-homomorphism $\eps : \cA \to \C$ satisfying
$$S(U_n(i,j)) = U_n(\jbar,\ibar) \quad\text{and}\quad \eps(U_n(i,j)) = \delta_{i,j} \quad\text{for all $n \geq 1$, $i,j \in I^n$.}$$
\item The pair $(\cA,\Delta)$ is a multiplier Hopf $*$-algebra in the sense of \cite[Definition 2.4]{VD92}. Its antipode is $S$ and its co-unit is $\eps$.
\end{enumlist}
\end{proposition}
\begin{proof}
1.\ It follows from Lemma \ref{lem.strict-convergence}.1 that multiplying the terms of the sum in \eqref{eq.def-comult-A} on the left or on the right by $a \ot 1$ or $1 \ot a$ for some element $a \in \cA$, only finitely many nonzero terms remain. So, we find a well-defined linear map
$$\Psi : \cUA \to M(\cA \ot \cA) : \Psi(U_n(i,j)) = \sum_{k \in I^n} (U_n(i,k) \ot U_n(k,j)) \quad\text{strictly.}$$
We have to prove that $\Psi(\cI) = \{0\}$, where $\cI$ is the $*$-ideal defined in Lemma \ref{lem.ideal-UA}. Fix $n,m \geq 1$, $\cK \in \cD_1(n,m)$ and $i \in I^n$, $j \in I^m$. We have to prove that $\Psi((U_n T^\cK)_{ij}) = \Psi((T^\cK U_m)_{ij})$.

We say that a matrix $(X_{ij})_{i,j \in J}$ over an index set $J$ and with entries in an algebra is of finite type if for all $i \in J$, there are at most finitely many $j \in J$ with $X_{ij} \neq 0$, and for all $j \in J$, there are at most finitely many $i \in J$ with $X_{ij} \neq 0$. The product of matrices of finite type is well-defined and associative.

Fix $a,b \in \cA$. By Lemma \ref{lem.strict-convergence}.1, the formulas $X_{ik} = a U_n(i,k) \ot 1$ and $Y_{kj} = 1 \ot b U_{k,j}$ define matrices of finite type with entries in $M(\cA \ot \cA)$. Note that $X T^\cK = T^\cK X$ and $Y T^\cK = T^\cK Y$. Using the definition of $\Psi$, we get that $(XY)_{ij} = (a \ot b) \Psi(U_n(i,j))$ for all $i,j \in I^n$. It follows that
$$(a \ot b) \Psi((U_n T^\cK)_{ij}) = (XYT^\cK)_{ij} = (XT^\cK Y)_{ij} = (T^\cK XY)_{ij} = (a \ot b) \Psi((T^\cK U_n )_{ij})\; .$$
Since this holds for all $a,b \in \cA$, we conclude that $\Psi((U_n T^\cK)_{ij}) = \Psi((T^\cK U_m)_{ij})$. It follows that $\Psi(\cI) = \{0\}$.

By passing to the quotient, we define $\Delta : \cA \to M(\cA \ot \cA)$. It is easy to check that $\Delta$ is a $*$-homomorphism.

2.\ The same argument showing that $\cI^* = \cI$ shows that $S$ is a well-defined $*$-anti-automorphism. It is easy to check that $\eps$ is well-defined.

3.\ The coassociativity of $\Delta$ follows immediately from the definition of $\Delta$. We have already seen in 1 that $\Delta(\cA) (1 \ot \cA)$ and $(\cA \ot 1)\Delta(\cA)$ are subspaces of $\cA \ot \cA$. By Lemma \ref{lem.strict-convergence}.3, we get for all $i,j \in I^n$ that
$$\sum_{k \in I^n} U_n(i,k) U_n(\jbar,\overline{k}) = \delta_{i,j} 1 = \sum_{k \in I^n} U_n(\overline{k},\ibar) U_n(k,j) \quad\text{strictly.}$$
So, for all $a \in \cA$, $i,j \in I^n$, we get that
$$\sum_{k \in I^n} a U_n(i,k) S(U_n(k,j)) = \delta_{i,j} a = \sum_{k \in I^n} S(U_n(i,k)) U_n(k,j) a \; .$$
It follows that $(\cA,\Delta)$ is a multiplier Hopf $*$-algebra with antipode $S$ and counit $\eps$.
\end{proof}

To prove that the multiplier Hopf $*$-algebra $(\cA,\Delta)$ admits positive left and right invariant functionals, we will construct a $*$-isomorphism between $\cA$ and a corner of the $*$-algebra $\cB$ constructed in the previous section. Using the positive faithful functional $\vphi$ on $\cB$, we will find the left Haar integral and the right Haar integral on $(\cA,\Delta)$.

From now on, we write $u_{ij} = U_1(i,j)$ for all $i,j \in I$. Note that the elements $(u_{ij})_{i,j \in I}$ generate $\cA$.

Fix a base vertex $e \in I$. Recall the equivalence relation $\approx$ on $I$ defined in \eqref{eq.equiv-rel-on-I} with equivalence classes $(I_a)_{a \in \cE}$. Take the unique $a \in \cE$ such that $e \in I_a$. With the same proof as for Lemma \ref{lem.ideal-UB}.2, we have that $u_{ij} = 0$ if $i \not\approx j$. So, by Lemma \ref{lem.strict-convergence}.5, we get that
$$\sum_{i \in I_a} u_{ie} = 1 \quad\text{strictly.}$$

Define the $*$-subalgebra $\cB_e \subset \cB$ as the linear span of $F_0(i,e) \cB F_0(j,e)$, $i,j \in I_a$. Since $(F_0(i,j))_{(i,j) \in I \times I}$ is an orthogonal family of idempotents in $\cB$, we should consider $\cB_e$ as a corner of $\cB$.

\begin{lemma}\label{lem.iso-A-B}
The map
\begin{equation}\label{eq.map-Theta}
\Theta_e : \cA \to \cB_e : \Theta_e(U_n(i,j)) = \sum_{s,t \in I_a} F_{n+1}(sit,eje)
\end{equation}
is a $*$-isomorphism with inverse $\Theta_e^{-1}(F_n(i,j)) = U_{n+1}(i,j)$ whenever $i_0,i_n \in I_a$ and $j_0=j_n=e$.
\end{lemma}

Note that by Lemma \ref{lem.ideal-UB}.3, the sum in \eqref{eq.map-Theta} has only finitely many nonzero terms.

\begin{proof}
Since for every $\cK \in \cD_1(n,m)$, we have that $1 \ot T^\cK \ot 1 \in \cL_1(n,m)$, while $\cL_1(n,m) \subset \cD_1(n+1,m+1)$, it is immediate that the formulas for $\Theta_e$ and its inverse respect the defining relations for $\cA$ and $\cB$. Because
$$\sum_{s,t \in I_a} U_{n+2}(sit,eje) = \Bigl(\sum_{s \in I_a} u_{se}\Bigr) \, U_n(i,j) \, \Bigl(\sum_{t \in I_a} u_{te} \Bigr) = U_n(i,j) \; ,$$
we indeed get that $\Theta_e$ and $\Theta_e^{-1}$ are each other's inverse.
\end{proof}

\begin{definition}\label{def.left-Haar}
Using the notation of Definition \ref{def.functional-vphi}, we define the functional $\vphi_e : \cA \to \C : \vphi_e = \vphi \circ \Theta_e$. By Proposition \ref{prop.vphi-positive-faithful} and Lemma \ref{lem.modular-vphi}, $\vphi_e$ is a positive, faithful, tracial functional on $\cA$.
\end{definition}

Using the notation of Lemma \ref{lem.ibf}, we also get the following explicit formula for $\vphi_e$ on a set of elements linearly spanning $\cA$. When $i \in I^n$, we denote by $(i) \in \ell^2(I^n)$ the corresponding basis vector, i.e.\ the function that is equal to $1$ on $i$ and equal to $0$ elsewhere.

\begin{equation}\label{eq.formula-left-Haar}
\vphi_e\bigl( u_{se} U_n(i,j) u_{te} \bigr) = \delta_{s,t} \, \sum_{V \in \ibf_{a-a}(n+1,\eps_a)} \langle V(s),(sis)\rangle \, \langle (eje),V(e) \rangle
\end{equation}
for all $s,t \in I_a$, $i,j \in I^n$, $n \geq 1$.

\begin{theorem}\label{thm.left-Haar}
The functional $\vphi_e$ is left invariant: for all $a \in \cA$, we have that $(\id \ot \vphi_e)\Delta(a) = \vphi_e(a)1$.
Also, $\psi_e := \vphi_e \circ S$ is a positive, faithful, right invariant functional on $(\cA,\Delta)$. Thus, $(\cA,\Delta)$ is an algebraic quantum group in the sense of \cite{VD96} and \cite[Definition 1.2]{KVD96}.
\end{theorem}
\begin{proof}
It suffices to prove that for all $a \in \cA$ and all $b \in \cA$ of the form $b = u_{se} U_n(i,j) u_{te}$, we have that
$$(\id \ot \vphi_e)((a \ot 1)\Delta(b)) = a \vphi_e(b) \; .$$
By definition, the left hand side is given by the following sum, with only finitely many nonzero terms.
$$\sum_{r_1,r_2 \in I, k \in I^n} a \, u_{sr_1} \, U_n(i,k) \, u_{tr_2} \; \vphi_e\bigl(u_{r_1e} \, U_n(k,j) \, u_{r_2 e}\bigr) \; .$$
Using \eqref{eq.formula-left-Haar}, this expression equals
\begin{multline}\label{eq.we-compute}
\sum_{r \in I, k \in I^n} a \, u_{sr} \, U_n(i,k) \, u_{tr} \; \sum_{V \in \ibf_{a-a}(n+1,\eps_a)} \langle V(r),(rkr)\rangle \, \langle (eje),V(e) \rangle \\
= \sum_{r \in I, k \in I^n, V \in \ibf_{a-a}(n+1,\eps_a)} a \, U_{n+2}(sit,rkr) \, \langle V(r),(rkr)\rangle \, \langle (eje),V(e) \rangle \; .
\end{multline}
Similarly to \eqref{eq.notation-sesquilinear}, we introduce the notation
$$U_n(\xi,\eta) = \sum_{i,j \in I^n} \overline{\xi(i)} \, \eta(j) \, U_n(i,j)$$
for all finitely supported functions $\xi,\eta \in \cF(I^n) \subset \ell^2(I^n)$. By definition of $\cA$, we have that $U_n(T^\cK \xi,\eta) = U_m(\xi,(T^\cK)^* \eta)$ for all $\cK \in \cD_1(n,m)$, $\xi \in \cF(I^m)$ and $\eta \in \cF(I^n)$.

So we find that the expression in \eqref{eq.we-compute} equals
\begin{multline}\label{eq.we-compute-further}
\sum_{r \in I, V \in \ibf_{a-a}(n+1,\eps_a)} a \, U_{n+2}\bigl((sit),V(r)\bigr) \, \langle (eje),V(e) \rangle \\
= \sum_{r \in I, V \in \ibf_{a-a}(n+1,\eps_a)} a \, U_{1}\bigl(V^*(sit),(r)\bigr) \, \langle (eje),V(e) \rangle \; .
\end{multline}
Every $V$ appearing in \eqref{eq.we-compute-further} is $\ell^\infty(I)$-bimodular, so that $V^*(sit) = \delta_{s,t}  \, \langle V^*(sit),(s) \rangle \, (s)$. Therefore, the expression in \eqref{eq.we-compute-further} equals
$$\delta_{s,t} \sum_{r \in I, V \in \ibf_{a-a}(n+1,\eps_a)} a \, U_{1}(s,r) \, \langle (s), V^*(sis) \rangle \, \langle (eje),V(e) \rangle \; .$$
Since $\sum_{r \in I} U_1(r,s) = 1$ strictly, we finally find equality with
$$\delta_{s,t} \sum_{V \in \ibf_{a-a}(n+1,\eps_a)} a \, \langle (s), V^*(sis) \rangle \, \langle (eje),V(e) \rangle = a \vphi_e(b) \; .$$
So, $\vphi_e$ is left invariant. Since $S$ is a $*$-anti-automorphism satisfying $\Delta \circ S = \sigma \circ (S \ot S) \circ \Delta$, where $\sigma$ is the flip automorphism, we get that $\psi_e$ is positive, faithful and right invariant.
\end{proof}

By \cite[Theorem 3.7]{VD96}, the left invariant functional on $(\cA,\Delta)$ is unique up to a scalar multiple. So, changing the base vertex $e \in I$ to $f \in I$, we must have that $\vphi_f$ is a multiple of $\vphi_e$. This multiple can be explicitly computed using the map $\mu : I \to (0,+\infty)$ of Lemma \ref{def.formula-for-rho}.

\begin{proposition}\label{prop.dependence-base-point}
We have that $\vphi_e = \mu_f \, \mu_e^{-1} \, \vphi_f$ for all $e,f \in I$.
\end{proposition}
\begin{proof}
Take the unique $a,b \in \cE$ such that $e \in I_a$ and $f \in I_b$. Fix $n \geq 1$ and $i,j \in I^n$ and $s \in I_a$. We make the following direct computation, making use of Lemma \ref{lem.further-relations}, Proposition \ref{prop.vphi-positive-faithful} and Lemma \ref{lem.modular-vphi}.
\begin{align*}
\vphi_e(u_{se} U_n(i,j)) &= \vphi(F_{n+1}(sis,eje)) = \sum_{r \in I_b} \vphi(F_{n+2}(sirs,ejfe)) \\
&= \sum_{r \in I_b} \vphi(F_{n+1}(sir,ejf) F_1(rs,fe)) = \sum_{r \in I_b} \vphi(\rho^{-1}(F_1(rs,fe)) F_{n+1}(sir,ejf)) \\
&= \frac{\mu_f}{\mu_e} \sum_{r \in I_b} \vphi(F_1(rs,fe) F_{n+1}(sir,ejf)) = \frac{\mu_f}{\mu_e} \sum_{r \in I_b} \vphi(F_{n+2}(rsir,fejf)) \\
&= \frac{\mu_f}{\mu_e} \vphi_f(U_{n+1}(si,ej)) =  \frac{\mu_f}{\mu_e} \vphi_f(u_{se} U_n(i,j))\; .
\end{align*}
This concludes the proof of the proposition.
\end{proof}

We can now interpret the equivalence relation $\approx$ on $I$ as really giving the quantum orbits under the quantum automorphism group, in the sense of \cite[Definition 3.1]{LMR17}.

\begin{corollary}\label{cor.about-rel-E}
For all $i,j,e \in I$, we have that
$$\vphi_e(u_{ij}) = \begin{cases} \mu_j \, \mu_e^{-1} &\quad\text{if $i \approx j$,}\\ 0 &\quad\text{if $i \not\approx j$.}\end{cases}$$
In particular, $u_{ij} \neq 0$ if and only if $i \approx j$.
\end{corollary}
\begin{proof}
Applying Proposition \ref{prop.dependence-base-point} to $f = j$, we find that $\vphi_e(u_{ij}) = \mu_j \, \mu_e^{-1} \, \vphi_j(u_{ij})$. When $i \not\approx j$, we have that $u_{ij} = 0$. When $i \approx j$, we have by definition that $\vphi_j(u_{ij}) = 1$ and thus $\vphi_e(u_{ij}) = \mu_j \, \mu_e^{-1}$.
\end{proof}

Note that, in the classical case where $(\cA,\Delta)$ corresponds to the ordinary automorphism group of $\Pi$, Corollary \ref{cor.about-rel-E} is saying that the numbers $\mu_j$ are equal to $\lambda(\Stab j)$, where $\lambda$ is a left Haar measure on $\Aut \Pi$.

By \cite[Proposition 3.10]{VD96}, we know that there exists an element $\delta \in M(\cA)$ such that $\psi_e(a) = \vphi_e(a\delta)$ for all $a \in \cA$. As in Proposition \ref{prop.dependence-base-point}, we can give a concrete formula for this \emph{modular element} of the algebraic quantum group $(\cA,\Delta)$.

We first need the following lemma.

\begin{lemma}\label{lem.iszero}
If $i,j,s,r \in I$ and $\mu_s \mu_i^{-1} \neq \mu_r \mu_j^{-1}$, then $u_{ij} u_{sr} = 0$ in $\cA$.
\end{lemma}
\begin{proof}
Fix a base vertex $e \in I$ and take $a \in \cE$ with $e \in I_a$. By definition,
$$\Theta_e(u_{ij} u_{sr}) = \sum_{t,l \in I_a} F_1(ti,ej) \, F_1(is,jr) \, F_1(sl,re) \; .$$
If $\mu_s \mu_i^{-1} \neq \mu_r \mu_j^{-1}$, it follows from Lemma \ref{lem.modular-vphi} that $F_1(is,jr) = 0$. So, $\Theta_e(u_{ij} u_{sr})=0$, implying that $u_{ij} u_{sr} = 0$.
\end{proof}

\begin{proposition}\label{prop.modular-element}
There exists a unique element $\delta \in M(\cA)$ such that for all $i,j \in I$, we have that $\delta u_{ij} = u_{ij} \delta = \mu_i \mu_j^{-1} u_{ij}$. The element $\delta$ belongs to the center of $M(\cA)$. For all $i,j \in I$, we have that
\begin{equation}\label{eq.formula-delta}
\sum_{k \in I} \frac{\mu_k}{\mu_j} u_{kj} = \delta = \sum_{k \in I} \frac{\mu_i}{\mu_k} u_{ik} \quad\text{strictly.}
\end{equation}
For every $e \in I$ and every $a \in \cA$, we have that $\psi_e(a) = \vphi_e(a \delta)$. So, $\delta$ is the modular element of $(\cA,\Delta)$.
\end{proposition}
\begin{proof}
Fix $e \in I$. Since for every $a \in \cA$, there are only finitely many $s \in I$ such that $a u_{se} \neq 0$ or $u_{se} a \neq 0$, we have a well-defined element $\delta \in M(\cA)$ such that $\sum_{s \in I} \mu_s \mu_e^{-1} u_{se} = \delta$ strictly.

Let $i,j \in I$ be arbitrary. By definition,
$$\delta \, u_{ij} = \sum_{s \in I} \frac{\mu_s}{\mu_e} u_{se} u_{ij} \; .$$
By Lemma \ref{lem.iszero}, if $u_{se} u_{ij} \neq 0$, we have $\mu_i \mu_s^{-1} = \mu_j \mu_e^{-1}$ and thus $\mu_s \mu_e^{-1} = \mu_i \mu_j^{-1}$. It follows that
$$\delta \, u_{ij} = \sum_{s \in I} \frac{\mu_i}{\mu_j} u_{se} u_{ij} = \frac{\mu_i}{\mu_j} u_{ij} \; .$$
We similarly find that $u_{ij} \, \delta = \mu_i \mu_j^{-1} u_{ij}$.

Since $\delta u_{ij} = u_{ij} \delta$ for all $i,j \in I$, we have that $\delta \, a = a \, \delta$ for all $a \in \cA$. It then follows that $\delta$ belongs to the center of $M(\cA)$.

Summing the equalities $\delta \, u_{kj} = \mu_k \mu_j^{-1} u_{kj}$ and $\delta \, u_{ik} = \mu_i \mu_k^{-1} u_{ik}$ over $k \in I$, we find \eqref{eq.formula-delta}.

To prove that $\psi_e(a) = \vphi_e(a \delta)$ for all $e \in I$, $a \in \cA$, we first note that in the same way as we defined the antipode $S : \cA \to \cA$, there exists a unique $*$-anti-automorphism $\kappa : \cB \to \cB$ satisfying $\kappa(F_n(i,j)) = F_n(\jbar,\ibar)$ for all $i,j \in I^{n+1}$, $n \geq 0$. Defining the anti-unitary operators $J_n$ on $\ell^2(I^{n+1})$ by $(J_n \xi)(i) = \overline{\xi(\ibar)}$, we have
$$\kappa(F_n(\xi,\eta)) = F_n(J_n \eta, J_n \xi) \quad\text{for all $\xi,\eta \in \cF(I^{n+1})$.}$$
Note that the map $V \mapsto J_n V J_0$ transforms any isometric basis of finite type $\ibf_{a-a}(n,\eps_a)$ into another isometric basis of finite type for $\Mor_{a-a}(n,0)$. We thus get by definition that $\vphi \circ \kappa = \vphi$.

Fix $e \in I$ and take $a \in \cE$ with $e \in I_a$. Since $\psi_e = \vphi_e \circ S$ and $\vphi_e = \vphi \circ \Theta_e$, we can make the following computation, using that $\vphi = \vphi \circ \kappa$ and using Proposition \ref{prop.dependence-base-point}. Let $n \geq 1$, $i,j \in I^n$ and $r \in I_a$ be arbitrary.
\begin{align*}
\psi_e(U_n(i,j) u_{er}) &= \vphi_e(u_{re} U_n(\jbar,\ibar)) = \vphi(F_{n+1}(r \jbar r, e \ibar e)) = \vphi(\kappa(F_{n+1}(r \jbar r, e \ibar e))) \\ &= \vphi(F_{n+1}(e i e, rjr))
= \vphi_r(U_n(i,j) u_{er}) = \mu_e \mu_r^{-1} \vphi_e(U_n(i,j)  u_{er}) \\ &= \vphi_e(U_n(i,j) \mu_e \mu_r^{-1} u_{er}) =\vphi_e(U_n(i,j) u_{er} \delta) \; .
\end{align*}
Since $i,j \in I^n$ and $r \in I_a$ are arbitrary, we conclude that $\psi_e(a) = \vphi_e(a \delta)$ for all $a \in \cA$.
\end{proof}

Quite naturally, we have a coaction of $(\cA,\Delta)$ on $\cF(I)$, the $*$-algebra of finitely supported functions on $I$. Note that the multiplier algebra $M(\cF(I))$ is the $*$-algebra of all functions from $I$ to $\C$.

\begin{proposition}\label{prop.coaction-on-I}
There is a unique nondegenerate $*$-homomorphism
$$\al : \cF(I) \to M(\cA \ot \cF(I)) : \al(p_i) = \sum_{j \in I} (u_{ij} \ot p_j) \quad\text{strictly, for all $i \in I$.}$$
We have that $\al$ is a coaction, meaning that $(\Delta \ot \id) \circ \al = (\id \ot \al) \circ \al$. A function $F \in M(\cF(I))$ satisfies $\al(F) = 1 \ot F$ if and only if $F$ is constant on each of the sets $I_a \subset I$, $a \in \cE$.
\end{proposition}
\begin{proof}
It is immediate to define the coaction $\al$. For every $F \in M(\cF(I))$ and $i,j \in I$, we have that $\al(F) (u_{ij} \ot p_j) = F(i) (u_{ij} \ot p_j)$. Therefore, $\al(F) = 1 \ot F$ if and only if for all $i,j$ with $u_{ij} \neq 0$, we have that $F(i) = F(j)$. By Corollary \ref{cor.about-rel-E}, this holds if and only if $F(i) = F(j)$ for all $i \approx j$.
\end{proof}

Corollary \ref{cor.about-rel-E} and Proposition \ref{prop.coaction-on-I} both say in a way that the equivalence classes $(I_a)_{a \in \cE}$ can be viewed as the quantum orbits of the quantum automorphism group $(\cA,\Delta)$ of $\Pi$ (w.r.t.\ the graph category $\cD$). The following result is thus quite natural and straightforward. We include it for later reference since in Examples \ref{ex.compact-noncompact} and \ref{ex.compact-nonunimodular}, we will provide graphs $\Pi$ such that $\Aut \Pi$ is compact, but $\QAut \Pi$ is noncompact. We use here as definition of compactness that the reduced C$^*$-algebra is unital.

\begin{proposition}\label{prop.criterion-compact}
The following statements are equivalent.
\begin{enumlist}
\item The locally compact quantum group associated with $(\cA,\Delta)$ is compact.
\item For every $a \in \cE$, the quantum orbit $I_a$ is finite.
\item There exists an $a \in \cE$ such that the quantum orbit $I_a$ is finite.
\item The $*$-algebra $\cA$ is unital.
\end{enumlist}
\end{proposition}
\begin{proof}
Fix a vertex $e \in I$ and let $H = L^2(\cA,\vphi_e)$ be the GNS Hilbert space of the left invariant functional $\vphi_e$ with faithful nondegenerate GNS-representation $\pi : \cA \to B(H)$. Denote by $A$ the norm closure of $\pi(\cA)$.

1 $\Rightarrow$ 2. By assumption, $A$ is unital. Fix $j \in I$. Since $(\pi(u_{ij}))_{i \in I}$ is a family of orthogonal projections in $A$ summing up strictly to $1$, there are only finitely many nonzero projections in this family. Since $\pi$ is faithful, this precisely says that the quantum orbit of $j$ is finite.

2 $\Rightarrow$ 3 is trivial and 3 $\Rightarrow$ 4 follows because for $j \in I_a$, $\sum_{i \in I_a} u_{ij}$ is the identity of $\cA$.

4 $\Rightarrow$ 1. When $\cA$ is unital, certainly $A$ is unital.
\end{proof}

The following result provides a sufficient condition for the unimodularity of $(\cA,\Delta)$. In Example \ref{ex.compact-nonunimodular}, we will see that the transitivity assumption is essential. In Example \ref{ex.nonunimodular-unimodular}, we will see that the converse of Proposition \ref{prop.unimodular} does not hold: there are quantum vertex transitive graphs $\Pi$ such that $\QAut \Pi$ is unimodular, but $\Aut \Pi$ is not unimodular.

\begin{proposition}\label{prop.unimodular}
If the classical automorphism group $\Aut \Pi$ is unimodular and acts transitively on $I$, then also $(\cA,\Delta)$ is unimodular.
\end{proposition}
\begin{proof}
Since $\Pi$ is vertex transitive, the set $\cE$ is a singleton. Take $W \in I \times I$ such that $1_W$ is a minimal projection in $\Mor(1,1)$. Since all elements of $\Mor(n,n)$ commute with the diagonal $\Aut \Pi$ action, it follows from the mass transport principle for unimodular graphs (see e.g.\ \cite[Corollary 8.8]{LP16}) that for all $i,j \in I$,
$$\sum_{k \in I} 1_W(i,k) = \sum_{k \in I} 1_W(k,j) \; .$$
This means that $\dim_\ell(1_W) = \dim_r(1_W)$, so that $\rho(1_W) = 1$. By Lemma \ref{def.formula-for-rho}, the function $\mu$ must be constant. So, $\delta = 1$ and $(\cA,\Delta)$ is unimodular.
\end{proof}

\subsection{Proof of Theorem \ref{thm.lc-qaut}}

\begin{proof}[{Proof of Theorem \ref{thm.lc-qaut}}]
We apply Theorem \ref{thm.left-Haar} to $\cD = \cP$. We have thus found an algebraic quantum group $(\cA,\Delta)$ generated by elements $(u_{ij})_{i,j \in I}$ satisfying all the properties listed in Theorem \ref{thm.lc-qaut}.

It only remains to prove the universal property. So assume that $\cA_1$ is a $*$-algebra generated by elements $(v_{ij})_{i,j \in I}$ satisfying the relations in Theorem \ref{thm.lc-qaut}. We have to prove the existence of a $*$-homomorphism $\Phi : \cA \to \cA_1$ satisfying $\Phi(u_{ij}) = v_{ij}$ for all $i,j \in I$.

Write $V_n(i,j) = v_{i_1 j_1} \cdots v_{i_n j_n}$ for all $n \geq 1$ and $i,j \in I^n$. We view $V_n$ as an $I^n \times I^n$ matrix with entries in $\cA_1$. It suffices to prove that $V_n T^\cK = T^\cK V_m$ for all $\cK \in \cP_1(n,m)$ and $n,m \geq 1$. Denote
$$\cP_0(n,m) = \{\cK \in \cP_1(n,m) \mid V_n T^\cK = T^\cK V_m \} \; .$$
Note that $\cP_0$ is closed under tensor products and compositions of bi-labeled graphs. By the properties of the generators $v_{ij}$, we have that $\cM^{1,2}$ and $\cM^{2,1}$ belong to $\cP_0$. Obviously, the identity $\cM^{1,1}$ belongs to $\cP_0$.

Multiplying the equality $\sum_{k : k \sim j} v_{ik} = \sum_{k : k \sim i} v_{kj}$ on the left by $v_{is}$ and on the right by $v_{tj}$, we find that
\begin{equation}\label{eq.reformulate-with-rel}
\rel(s,j) v_{is} v_{tj} = \rel(i,t) v_{is} v_{tj} \quad\text{where $\rel(k,l) = 1$ if $k \sim l$ and $\rel(k,l) = 0$ if $k \not\sim l$.}
\end{equation}
Define the interval graph $I_k$ for $k \geq 2$ with vertex set $\{1,\ldots,k\}$ and $i \sim j$ iff $|i-j| = 1$. Then \eqref{eq.reformulate-with-rel} is saying that $\cJ_2 := (I_2,(1,2),(1,2))$ belongs to $\cP_0(2,2)$. Define for $n \geq 2$, $\cJ_n = (I_n,(1,\ldots,n),(1,\ldots,n))$. Then,
$$\cJ_{n+1} = (1^{\ot (n-1)} \ot \cM^{1,2} \ot 1) \circ (\cJ_n \ot \cJ_2) \circ (1^{\ot (n-1)} \ot \cM^{2,1} \ot 1)$$
for all $n \geq 2$. By induction, we conclude that $\cJ_n \in \cP_0(n,n)$ for all $n \geq 2$.

Whenever $1 \leq n,m \leq k$, $1 \leq x_1 < x_2 < \cdots < x_n \leq k$ and $1 \leq y_1 < y_2 < \cdots < y_m \leq k$, we obtain the planar bi-labeled graph $(I_k,x,y)$. We denote by $\cI$ the set of all these bi-labeled graphs $(I_k,x,y)$, with $k$ and $x,y$ arbitrary as described. When $(I_k,x,y) \in \cI$ is arbitrary, we first write $(V_k T^{\cJ_k})_{ij} = (T^{\cJ_k} V_k)_{ij}$ for all $i,j \in I^k$, which holds because $\cJ_k \in \cP_0(k,k)$. If we now sum over all the indices $i_s \not\in x$ and $j_t \not\in y$ and use that $\sum_{i \in I} v_{ij} = 1$ strictly and $\sum_{j \in I} v_{ij} = 1$ strictly, we conclude that $(I_k,x,y) \in \cP_0$. So, $\cI \subset \cP_0$.

As in \cite[Theorem 6.7]{MR19}, one proves that $\cP_1$ is the smallest set of bi-labeled graphs that contains $\cI \cup \{\cM^{1,2},\cM^{2,1},\cM^{1,1}\}$ and that is closed under composition and tensor products. Thus, $\cP_0 = \cP_1$.
\end{proof}

Following Definition \ref{def.qAut}, we denote by $\bG = \QAut \Pi$ the locally compact quantum group arising from Theorem \ref{thm.lc-qaut}. Using \eqref{eq.canonical-closed-quantum-subgroup}, we view the classical automorphism group $\Aut \Pi$ as a closed quantum subgroup of $\QAut \Pi$. We say that $\Pi$ \emph{has quantum symmetry} if $\QAut \Pi \neq \Aut \Pi$.

\begin{remark}\label{rem.closed-quantum-subgroup}
The $*$-homomorphism $\pi : \cA \to \cO(G)$ in \eqref{eq.canonical-closed-quantum-subgroup} is surjective and is a morphism of multiplier Hopf $*$-algebras. In general, a surjective morphism of multiplier Hopf $*$-algebras $\pi : \cA \to \cA_1$ always identifies the corresponding locally compact quantum group $\bG_1$ as a closed quantum subgroup of $\bG$ in the strictest sense of the word (see \cite[Definition 2.5]{Vae04}). For this, it suffices to note that we may view $\cA$, resp.\ $\cA_1$, as a dense subspace of the predual $L(\bG)_*$, resp.\ $L(\bG_1)_*$, and that the faithful normal $*$-homomorphism $L(\bG_1) \to L(\bG)$ is therefore well-defined because $\pi(\cA) = \cA_1$.
\end{remark}

The $*$-algebra $\cA$ in Theorem \ref{thm.lc-qaut} has a universal C$^*$-algebra completion $C_{0,f}(\bG)$ and a reduced C$^*$-algebra completion $C_{0,r}(\bG)$, where the latter is given by the GNS representation w.r.t.\ the Haar functionals. We have a canonical surjective $*$-homomorphism $C_{0,f}(\bG) \to C_{0,r}(\bG)$. Of course, if $\Pi$ has no quantum symmetry, we have that $\bG = \Aut \Pi$ and both the full and reduced C$^*$-algebra equal $C_0(\Aut \Pi)$. Recall that $\bG$ is said to be \emph{co-amenable} if the above $*$-homomorphism $C_{0,f}(\bG) \to C_{0,r}(\bG)$ is faithful. Recall that the dual $\bGhat$ is said to be \emph{amenable} if $L^\infty(\bGhat) = L(\bG)$ admits an invariant mean. Co-amenability of $\bG$ implies amenability of $\bGhat$. It is an open problem whether the converse implication holds in general, but it does hold for compact quantum groups by \cite[Theorem 3.8]{Tom03}.

By \cite[Theorem 3.2]{Cra16}, amenability passes to closed quantum subgroups. So, the following result is an immediate consequence of \cite[Theorem 2.2]{Sch19b}. In particular for $d \geq 3$, the $d$-regular tree $\Pi$ has a lot of quantum symmetry, in the sense that $\QAut \Pi$ is not even co-amenable. Of course, $\QAut \Pi$ is also not amenable, because already $\Aut \Pi$ is nonamenable.

\begin{proposition}
Let $\Pi$ be a connected locally finite graph with vertex set $I$. Assume that $I = I_1 \sqcup I_2$ is a partition into two nonempty subsets. For all $k \in \{1,2\}$, write $G_k = \{\si \in \Aut \Pi \mid \si(i) = i \;\;\text{for all $i \in I_k$}\}$.
\begin{enumlist}
\item If both $G_1$ and $G_2$ are nontrivial groups, then $\Pi$ has quantum symmetry.
\item If $G_1$ and $G_2$ are both nontrivial groups and if at least one of them has at least three elements, then $\QAut \Pi$ is not co-amenable and the dual of $\QAut \Pi$ is not amenable.
\end{enumlist}
\end{proposition}

\begin{remark}\label{rem.compare-voigt}
In \cite[Definition 7.1]{Voi22}, a \emph{quantum automorphism} of a graph $\Pi$ with vertex set $I$ is defined as a magic unitary $U \in \cU(\ell^2(I) \ot \cH)$ that commutes with the adjacency matrix of $\Pi$. By definition, there is a bijective correspondence between such quantum automorphisms and nondegenerate $*$-representations $\pi$ of the $*$-algebra $\cA$ in Theorem \ref{thm.lc-qaut}. The tensor product between two such quantum automorphisms $\pi_1$ and $\pi_2$ is then given by $(\pi_1 \ot \pi_2) \circ \Delta$.

The discrete quantum automorphism group $\bG_d$ defined in \cite[Definition 7.2]{Voi22} is then associated with the finite-dimensional nondegenerate $*$-representations of $\cA$. So by definition, its compact dual $\widehat{\bG_d}$ is equal to the Bohr compactification of $\widehat{\QAut \Pi}$. Note that the $1$-dimensional $*$-representations of $\cA$ precisely correspond to the elements of $\Aut \Pi$, while finite-dimensional $*$-representations of $\cA$ are related to matrix models of $\QAut \Pi$, as discussed in \cite{Voi22}. Although we have no concrete example in mind, it is in principle possible that $\cA$ has very few finite-dimensional $*$-representations. For instance, the intersection of all their kernels could be nonzero. In principle, this can already happen with finite graphs. In such cases, the homomorphism from $\widehat{\QAut \Pi}$ to its Bohr compactification is not injective.
\end{remark}

\subsection{\boldmath Recovering the classical automorphism group of $\Pi$}

Let now $\cD = \cG$ be the graph category of all bi-labeled graphs and consider the associated multiplier Hopf $*$-algebra $(\cA,\Delta)$. Define $\cK = (K,x,y) \in \cG_1(2,2)$ by $V(K) = \{1,2\}$, $E(K) = \emptyset$, $x_1=1=y_2$ and $x_2 = 2 = y_1$. Then the relation $U_2 T^\cK = T^\cK U_2$ says that $u_{ij} u_{kl} = u_{kl} u_{ij}$, so that $\cA$ is abelian. We thus find that $\cA$ is the $*$-algebra $\cO(G)$ of locally constant functions on a totally disconnected, second countable, locally compact group $G$ that acts properly and continuously on the graph $\Pi$. Since the $*$-algebra $\cO(\Aut \Pi)$ satisfies the defining relations of $\cA$, we have proven the following.

\begin{proposition}\label{prop.link-with-classical-case}
If $\cD = \cG$ is the graph category of all bi-labeled graphs, the associated algebraic quantum group $(\cA,\Delta)$ is isomorphic with the algebra of locally constant functions on the automorphism group $\Aut \Pi$ with its natural comultiplication.
\end{proposition}

\subsection{Remarks}

\begin{remark}
When $\Pi$ is a locally finite graph with finitely many connected components, the automorphism group $\Aut \Pi$ with the topology of pointwise convergence is locally compact. It is thus tempting to also define $\QAut \Pi$ as a locally compact quantum group. This is however typically impossible. To see this, assume that $\Pi$ is the disjoint union of two connected locally finite graphs $\Pi_0$ and $\Pi_1$. Denote by $\cA_i$ the multiplier Hopf $*$-algebra defining $\QAut \Pi_i$ for $i =0,1$. Then the hypothetical $\QAut \Pi$ would admit the co-free product of $\QAut \Pi_0$ and $\QAut \Pi_1$ (i.e.\ the free product of the $*$-algebras $\cA_0$ and $\cA_1$, which amounts to the free product of the duals of $\QAut \Pi_i$, $i=0,1$) as a closed quantum subgroup. If one of the $\QAut \Pi_i$ is noncompact and the other is nontrivial, such a co-free product is not well-defined as a locally compact quantum group, for the same reason that the free product of two nontrivial locally compact groups, with at least one of them nondiscrete, is not well-defined as a locally compact group. Therefore, also $\QAut \Pi$ is not well-defined as a locally compact quantum group.

Once $\Pi$ is no longer connected, $\QAut \Pi$ can only be well-defined if all connected components of $\Pi$ have a compact quantum automorphism group. In particular, the quantum automorphism group of any finite graph is well-defined, as was done in \cite{Bic99,Ban03}.
\end{remark}

\begin{remark}
Let $\Pi$ be a connected locally finite graph with vertex set $I$. Denote by $(\cA,\Delta)$ the multiplier Hopf $*$-algebra defining $\QAut \Pi$. Fix a base vertex $e \in I$. It would be natural to try to define the stabilizer of $e$ as a compact open quantum subgroup of $\QAut \Pi$, in the sense of \cite{KKS15}. The projection $p = u_{ee}$ in $\cA$ is \emph{group-like}, meaning that $\Delta(p) (1 \ot p) = p \ot p = \Delta(p)(p \ot 1)$, but $p$ is typically not central, as required by \cite[Definition 4.1]{KKS15}. Nevertheless, the linear span of the projections $(u_{ie})_{i \in I}$ forms a co-ideal $\cD$ in $(\cA,\Delta)$. This means that $\Delta(\cD) \subset M(\cA \ot \cD)$. This co-ideal plays the role of the homogeneous space $\QAut \Pi / \Stab e$ and equals $\cD = \{x \in \cA \mid \Delta(x)(1 \ot p) = x \ot p = (1 \ot p) \Delta(x)\}$.

The algebraic quantum groups defined in \cite{SVV22} as Schlichting completions of discrete quantum groups have, by construction, a natural compact open quantum subgroup. Therefore in \cite{SVV22}, the Haar functionals can be constructed by combining the Haar state on this compact quantum subgroup with the ``counting measure'' on the homogeneous space. Since in our setting, the stabilizer of a vertex is not a compact quantum subgroup, we have to make a detour via Sections \ref{sec.2-category} and \ref{sec.construction-B} to define the Haar functionals on $\QAut \Pi$.
\end{remark}

\section{Fiber functors and quantizations of discrete groups}\label{sec.quantizations-discrete-groups}

We now assume that $\Gamma \subset \Aut \Pi$ is a subgroup that is acting simply transitively on the set $I$ of vertices of $\Pi$. This is equivalent with saying that $\Pi$ is the Cayley graph of $\Gamma$ w.r.t.\ a finite generating set $S \subset \Gamma$ satisfying $S=S^{-1}$. Recall that this Cayley graph has vertex set $\Gamma$, edges $E = \{(g,h) \in \Gamma \times \Gamma \mid g^{-1} h \in S\}$, and admits the action of $\Gamma$ by left translation.

In particular, $\Pi$ is vertex transitive so that $\cE$ is a singleton and $\cC(\cD,\Pi)$ is a unitary tensor category. In this section, we prove that $\cC(\cD,\Pi)$ has a natural fiber functor and we identify the dual of the associated compact quantum group as a quantization of $\Gamma$.

We use the ``path approach'' to $\cC(\cD,\Pi)$ as explained in Remark \ref{rem.path-approach-to-2-category}. We denote by $P_n \in \cC(\cD,\Pi)$ the (orthogonal projection onto the) $\ell^2$-space of the set of paths of length $n$ in $\Pi$.

Recall that we denoted by $\cL(n,m)$ the set of connected bi-labeled graphs $\cK = (K,x,y) \in \cD_c(n,m)$ satisfying $x_0 = y_0$ and $x_n = y_m$. To every $\cK \in \cL(n,m)$ we associated the matrix $T^\cK$ and we defined $\Mor(n,m)$ as the linear span of all $T^\cK$, $\cK \in \cL(n,m)$. We denote by $\cL_p(n,m) \subset \cL(n,m)$ the subset of bi-labeled graphs in which both $x$ and $y$ are paths in $K$. As explained in Remark \ref{rem.path-approach-to-2-category}, the linear span of all $T^\cK$, $\cK \in \cL_p(n,m)$, equals the space of morphisms from $P_m$ to $P_n$.

Denoting by $S = S^{-1}$ the given finite generating set of $\Gamma$, we define for every $\cK \in \cL_p(n,m)$ the $S^n \times S^m$ matrix $R^\cK$ by
\begin{equation}\label{eq.R-K-matrix}
R^\cK_{st} = T^\cK_{ij} \quad\text{with $i = (e,s_1, s_1 s_2,\ldots,s_1 \cdots s_n)$ and $j = (e,t_1,t_1 t_2,\ldots,t_1 \cdots t_m)$.}
\end{equation}
Note that $R^\cK_{st} = 0$ if $s_1 \cdots s_n \neq t_1 \cdots t_m$. Also note that $R^\cK_{st}$ counts the number of labelings of the edges $(v,w) \in E(K)$ by elements $\psi(v,w) \in S$ subject to the following constraints:
\begin{itemlist}
\item $\psi(x_{i-1},x_i) = s_i$ for all $i \in \{1,\ldots,n\}$ and $\psi(y_{i-1},y_i) = t_i$ for all $i \in \{1,\ldots,m\}$,
\item for every path $(v_0,\ldots,v_k)$ in $K$ with $v_0 = v_k$, we have that $\psi(v_0,v_1) \cdots \psi(v_{k-1},v_k) = e$~; in particular, $\psi(w,v) = \psi(v,w)^{-1}$ for all $(v,w) \in E(K)$.
\end{itemlist}\

\begin{proposition}\label{prop.fiber-functor-compact-quantum}
The maps $P_n \mapsto \ell^2(S^n)$ and $T^\cK \mapsto R^\cK$ define a unitary fiber functor on $\cC(\cD,\Pi)$.

The resulting compact quantum group $\bG$ is of Kac type, has a fundamental self-conjugate unitary representation $V = (v_{st})_{s,t \in S}$ and the space of morphisms from the $m$-fold to the $n$-fold tensor power of $V$ is given by the linear span of the matrices $R^\cK$, $\cK \in \cL_p(n,m)$.

The map $v_{st} \mapsto \delta_{s,t} \, s$ extends to a surjective Hopf $*$-algebra homomorphism from the polynomial algebra $\cO(\bG)$ to the group algebra $\C[\Gamma]$. It turns $\bGhat$ into a quantization of $\Gamma$ in the sense of Definition \ref{def.quantization}.
\end{proposition}
\begin{proof}
The first two statements hold by construction. To prove the last statement, define for every $n \geq 1$,
$$\cR_n = \{(s_1,\ldots,s_n) \in S^n \mid s_1 \cdots s_n = e \;\;\text{in $\Gamma$}\;\} \; .$$
Define $\cK = (K,x,y) \in \cL_p(n,0)$ by $K = \Z / n\Z$, $i \sim j$ iff $i-j \in \{\pm 1\}$, $x_i = i + n\Z$ for all $i=0,\ldots,n$ and $y_0 = n \Z$, as illustrated in Figure \ref{fig.circle-graph}.
\begin{figure}[h]
    \centering
    \begin{tikzpicture}
    \filldraw (-90:1.5cm) circle (1.5pt) node[anchor=north]{$x_0 , y_0 , x_6$};
    \filldraw (-30:1.5cm) circle (1.5pt) node[anchor=north west]{$x_5$};
    \filldraw (30:1.5cm) circle (1.5pt) node[anchor=south west]{$x_4$};
    \filldraw (90:1.5cm) circle (1.5pt) node[anchor=south]{$x_3$};
    \filldraw (150:1.5cm) circle (1.5pt) node[anchor=south east]{$x_2$};
    \filldraw (210:1.5cm) circle (1.5pt) node[anchor=north east]{$x_1$};
    \draw (0,0) circle (1.5cm);
    \end{tikzpicture}
    \caption{The bi-labeled graph $\cK \in \cL_p(n,0)$ with $n=6$}\label{fig.circle-graph}
\end{figure}

By definition, $R^\cK_s = 1$ if $s \in \cR_n$ and $R^\cK_s = 0$ if $s \not\in \cR_n$. Since $\Gamma$ is generated by $S$ with relations $\bigsqcup_{n=1}^\infty \cR_n$, the conclusion follows from Lemma \ref{lem.quantization} below.
\end{proof}

In the special case where $\cD = \cP$, the graph category of planar bi-labeled graphs, we find the following description for the compact quantum group $\bG$ of Proposition \ref{prop.fiber-functor-compact-quantum}. In combination with Proposition \ref{prop.link-with-equivariant-bimodules} below, the following result will provide the proof of Theorem \ref{thm.identify-categories} stated in the introduction.

\begin{theorem}\label{thm.quantization-discrete-group}
Let $\Gamma$ be a countable group with finite symmetric generating set $S = S^{-1} \subset \Gamma$. The compact quantum group $\bG$ that is associated in Proposition \ref{prop.fiber-functor-compact-quantum} to the Cayley graph of $(\Gamma,S)$ and the graph category $\cP$ of planar bi-labeled graphs equals the universal compact quantum group $\bG$ generated by the entries of an $S \times S$ unitary representation $U$ such that for every $n \geq 1$, the vector $\xi_n \in \ell^2(S^n)$
$$\xi_n(s_1,\ldots,s_n) = \begin{cases} 1 &\quad\text{if $s_1 \cdots s_n = e$ in $\Gamma$,}\\ 0 &\quad\text{otherwise,}\end{cases}$$
is invariant under the $n$-fold tensor power of $U$.
\end{theorem}

\begin{definition}\label{def.planar-quantization-discrete-group}
For every countable group $\Gamma$ with finite symmetric generating set $S = S^{-1} \subset \Gamma$, we call the dual $\bGhat$ of the compact quantum group defined in Theorem \ref{thm.quantization-discrete-group} the \emph{planar quantization of $(\Gamma,S)$}.
\end{definition}

In Corollary \ref{cor.characterize-trivial-planar-quantization} below, we will see that it almost never happens that the surjective $*$-homo\-mor\-phism $\pi : \cO(\bG) \to \C[\Gamma]$ is an isomorphism.

\begin{proof}[{Proof of Theorem \ref{thm.quantization-discrete-group}}]
We denote as above by $\cP_p$ the set of planar bi-labeled graphs in which the labelings form paths. It then suffices to check, as in \cite[Theorem 6.7]{MR19}, that $\cP_p$ is the smallest set of bi-labeled graphs that contains all the ``circular'' planar bi-labeled graphs $\cK$ introduced in the proof of Proposition \ref{prop.fiber-functor-compact-quantum} (and giving rise to the invariant vector $\xi_n$) and that is closed under composition and relative tensor product
$$\cK \ot_r \cK' = (1^{\ot n} \ot \cM^{1,2} \ot 1^{\ot n'}) \circ (\cK \ot \cK') \circ (1^{\ot n} \ot \cM^{2,1} \ot 1^{\ot n'}) \; .$$
\end{proof}

In the following result, we prove that the compact quantum group $\bG$ of Theorem \ref{thm.quantization-discrete-group} can be characterized as well as the universal compact quantum group acting in a trace preserving way on the reduced C$^*$-algebra $C^*_r(\Gamma)$ such that $\lspan \{\lambda_s \mid s \in S\}$ is an invariant subspace. This then identifies $\bG$ with the quantum isometry group defined in \cite[Theorems 2.2 and 2.6]{BhS10}, associated with the spectral triple given by $(\Gamma,S)$.

\begin{proposition}\label{prop.identify-with-quantum-isometry-group}
Let $\Gamma$ be a countable group with finite symmetric generating set $S = S^{-1} \subset \Gamma$. Denote by $\bG$ the compact quantum group defined in Theorem \ref{thm.quantization-discrete-group} with fundamental unitary representation $U = (u_{st})_{s,t \in S}$. Denote by $\ell(g)$ the word length (w.r.t.\ $S$) and consider the reduced group C$^*$-algebra $C^*_r(\Gamma)$ generated by the group of unitaries $(\lambda_g)_{g \in \Gamma}$. Define $\cW_n = \lspan \{\lambda_g \mid \ell(g) = n \}$.

\begin{enumlist}
\item The formula
$$\al(\lambda_t) = \sum_{s \in S} \lambda_s \ot u_{st} \quad\text{for all $t \in S$,}$$
uniquely extends to a trace preserving action of $\bG$ on $C^*_r(\Gamma)$.

\item We have that $\al(\cW_n) \subset \cW_n \ot \cO(\bG)$ for all $n \geq 0$.

\item If $\be : C^*_r(\Gamma) \to C^*_r(\Gamma) \ot C(\bH)$ is any continuous, trace preserving action of a compact quantum group $\bH$ on $C^*_r(\Gamma)$ such that $\be(\cW_1) \subset \cW_1 \ot C(\bH)$, there is a unique Hopf $*$-algebra homomorphism $\pi : \cO(\bG) \to \cO(\bH)$ such that $\be = (\id \ot \pi) \circ \al$.
\end{enumlist}
\end{proposition}

\begin{proof}
1.\ In order to prove that $\al$ uniquely extends to a continuous $*$-homomorphism $\al : C^*_r(\Gamma) \to C^*_r(\Gamma) \ot C(\bG)$, we start by proving that $\al(\lambda_{t_1}) \cdots \al(\lambda_{t_n}) = 1 \ot 1$ whenever $t \in S^n$ and $t_1 \cdots t_n = e$. This means that we have to prove that
\begin{equation}\label{eq.equal-1}
\sum_{s \in S^n} \lambda_{s_1 \cdots s_n} \ot u_{s_1 t_1} \cdots u_{s_n t_n} = 1 \ot 1 \; .
\end{equation}
Defining, for every $g \in \Gamma$,
\begin{equation}\label{eq.def-y}
y_g = \sum_{s \in S^n : s_1 \cdots s_n = g} u_{s_1 t_1} \cdots u_{s_n t_n} \; ,
\end{equation}
we thus have to prove that $y_e = 1$ and that $y_g = 0$ if $g \neq e$.

We use the context and notation of Proposition \ref{prop.fiber-functor-compact-quantum} for the special case where $\cD = \cP$ is the graph category of planar bi-labeled graphs. As in the beginning of this section, we consider the morphism spaces $\Mor(n,m)$ in the unitary tensor category $\cC(\cP,\Pi)$, where $\Pi$ is the Cayley graph of $(\Gamma,S)$. Writing $T_n = T^{\cK_d} \in \Mor(n,1)$ where $d = (n)$ and $\cK_d$ is defined in Figure \ref{fig.graph-Kd}, we conclude that $T_n T_n^* \in \Mor(n,n)$. Applying the fiber functor of Proposition \ref{prop.fiber-functor-compact-quantum}, we get that for every $n \geq 1$, the matrix $X_n \in B(\ell^2(S^n))$ defined by
$$X_{n,st} = \begin{cases} 1 &\quad\text{if $s_1 \cdots s_n = t_1 \cdots t_n$,} \\ 0&\quad\text{otherwise,}\end{cases}$$
is an endomorphism of the $n$-fold tensor power $U^{\ot n}$ of $U$.

Fix $g \in \Gamma$. If $g$ cannot be written as a product of $n$ elements in $S$, obviously $y_g = 0$. If $g = r_1 \cdots r_n$ for some $r \in S^n$, we find that
$$y_g = (X_n \, U^{\ot n})_{rt} = (U^{\ot n} \, X_n)_{rt} = (U^{\ot n} \, \xi_n)_r = \xi_n(r) \; .$$
By definition, $\xi_n(r) = 1$ when $g = e$ and $\xi_n(r) = 0$ if $g \neq e$. So, \eqref{eq.equal-1} is proven.

Since $\xi_2$ is an invariant vector of $U^{\ot 2}$ and $U$ is unitary, we find that $u_{st}^* = u_{s^{-1}t^{-1}}$ for all $s,t \in S$. It follows that $\al(\lambda_t)^* = \al(\lambda_{t^{-1}})$ for all $t \in S$. Since $t t^{-1} = e = t^{-1} t$, the previous paragraphs imply that $\al(\lambda_t) \al(\lambda_t)^* = 1 \ot 1 = \al(\lambda_t)^* \al(\lambda_t)$. So, $\al(\lambda_t)$ is a unitary for all $t \in S$.

If now $s_1 \cdots s_n = t_1 \cdots t_m$ for some $s \in S^n$ and $t \in S^m$, also $s_1 \cdots s_n t_m^{-1} \cdots t_1^{-1} = e$ and we conclude using the previous paragraphs that
$$\al(\lambda_{s_1}) \cdots \al(\lambda_{s_n}) = \al(\lambda_{t_1}) \cdots \al(\lambda_{t_m}) \; .$$
It follows that $\al$ uniquely extends to a well-defined $*$-homomorphism from $\C[\Gamma]$ to $\C[\Gamma] \ot \cO(\bG)$.

Denote by $\tau$ the canonical trace on $C^*_r(\Gamma)$. Since $\xi_n$ is an invariant vector for $U^{\ot n}$, we immediately find that $(\tau \ot \id)(\al(x)) = \tau(x) 1$ for all $x \in \C[\Gamma]$. Therefore, $\al$ extends to a $*$-homomorphism from $C^*_r(\Gamma)$ to $C^*_r(\Gamma) \ot C(\bG)$ and thus defines a continuous action of $\bG$ on $C^*_r(\Gamma)$ by construction.

2.\ Let $t \in S^n$ such that $h := t_1 \cdots t_n$ has length $k$. We must prove that $\al(\lambda_h) \in \cW_k \ot \cO(\bG)$. Again defining $y_g$ by \eqref{eq.def-y}, we have to prove that $y_g = 0$ whenever $\ell(g) \neq k$.

Since $\Gamma$ is the vertex set of $\Pi$, we have $\Mor(1,1) \subset \ell^\infty(\Gamma \times \Gamma)$. Write $V_k = \{(g,h) \mid \ell(g^{-1}h) = k\}$. It follows from Example \ref{rem.path-approach-to-2-category} that $1_{V_k} \in \Mor(1,1)$. Using $T_n \in \Mor(n,1)$ as in part 1 of the proof and applying the fiber functor to $T_n 1_{V_k} T_n^*$, we conclude that the matrix $X_{n,k} \in B(\ell^2(S^n))$ defined by
$$X_{n,k,st} = \begin{cases} 1 &\quad\text{if $s_1 \cdots s_n = t_1 \cdots t_n$ and $\ell(s_1 \cdots s_n) = k$,} \\ 0&\quad\text{otherwise,}\end{cases}$$
is an endomorphism of the $n$-fold tensor power $U^{\ot n}$ of $U$.

Fix $g \in \Gamma$ with $\ell(g) \neq k$. Write $\ell = \ell(g)$. If $g$ cannot be written as a product of $n$ elements in $S$, obviously $y_g = 0$. If $g = r_1 \cdots r_n$ for some $r \in S^n$, we find that
$$y_g = (X_{n,\ell} \, U^{\ot n})_{rt} = (U^{\ot n} \, X_{n,\ell})_{rt} = 0 \; .$$

3.\ Define $v_{st} \in C(\bH)$ such that $\be(\lambda_t) = \sum_{s \in S} \lambda_s \ot v_{st}$. Since $\be$ is a continuous trace preserving action, $\be$ defines a unitary representation of $\bH$ on $\ell^2(\Gamma)$ for which $\cW_1$ is an invariant subspace. So the matrix $V = (v_{st})_{s,t \in S}$ is a unitary representation of $\bH$. Expressing that
$$(\tau \ot \id)\be(\lambda_{s_1} \cdots \lambda_{s_n}) = \begin{cases} 1 &\quad\text{if $s_1 \cdots s_n = e$,}\\ 0 &\quad\text{otherwise,}\end{cases}$$
we conclude that the vector $\xi_n$ defined in Theorem \ref{thm.quantization-discrete-group} is invariant under the $n$-fold tensor power of $V$. So by definition, there is a unique Hopf $*$-algebra homomorphism $\pi : \cO(\bG) \to \cO(\bH)$ satisfying $\pi(u_{st}) = v_{st}$ for all $s,t \in S$. By construction, $\be = (\id \ot \pi) \circ \al$.
\end{proof}

Let $\Gamma$ be a countable group generated by a finite symmetric subset $S = S^{-1} \subset \Gamma$. Denote by $\F_S$ the free group with free generators $a_s$, $s \in S$. Given a subset $\cR \subset \sqcup_{n=1}^\infty S^n$, we say that $\Gamma$ is generated by $S$ with relations $\cR$ if the kernel of the natural homomorphism $\F_S \to \Gamma : a_s \mapsto s$ equals the smallest normal subgroup of $\F_S$ containing $a_s a_{s^{-1}}$ for all $s \in S$ and $a_{s_1} \cdots a_{s_n}$ for all $(s_1,\ldots,s_n) \in \cR$.

\begin{lemma}\label{lem.quantization}
Assume that $\Gamma$ is a countable group generated by a finite symmetric subset $S = S^{-1} \subset \Gamma$ and relations $\cR \subset \sqcup_{n=1}^\infty S^n$. Assume that $(s,s^{-1}) \in \cR$ for all $s \in S$.

Let $\bG$ be a compact quantum group of Kac type with polynomial algebra $\cO(\bG)$. Assume that $\cO(\bG)$ is generated by the entries $u_{st}$ of an $S \times S$ unitary representation $U$ of $\bG$.

Assume that $\pi : \cO(\bG) \to \C[\Gamma]$ is a surjective Hopf $*$-algebra homomorphism satisfying $\pi(u_{st}) = \delta_{s,t} \, s$ for all $s \in S$.

Assume that for every $(s_1,\ldots,s_n) \in \cR$, there exists a $\bG$-invariant vector $\xi$ in the $n$-fold tensor power of $U$ with $\xi(s_1,\ldots,s_n) \neq 0$. Then $\widehat{\bG}$ is a quantization of $\Gamma$ in the sense of Definition \ref{def.quantization}.
\end{lemma}
\begin{proof}
Assume that $\Lambda$ is a countable group, $\psi : \cO(\bG) \to \C[\Lambda]$ is a surjective Hopf $*$-algebra homomorphism and $\rho : \Lambda \to \Gamma$ is a surjective group homomorphism such that $\rho \circ \psi = \pi$. We have to prove that $\rho$ is an isomorphism.

The elements $\psi(u_{st})$ are the entries of a unitary representation of $\widehat{\Lambda}$, which must be the direct sum of $|S|$ one-dimensional representations. We thus find a unitary matrix $X \in \cU(\C^S)$ such that $X \psi(u) X^*$ is a diagonal matrix. Denote by $D \in M_S(\C) \ot \C[\Gamma]$ the diagonal matrix with diagonal entries $s \in S$. Applying $\rho$ to $X \psi(u) X^*$, it follows that $X D X^*$ remains diagonal. Since the elements $s \in S$ are all distinct, this forces $X$ to be the product of a diagonal matrix and a permutation matrix. In particular, $\psi(u)$ is already a diagonal matrix. The diagonal entries of $\psi(u)_{ss}$ then are one-dimensional representations of $\widehat{\Lambda}$ and thus given by group elements $\gamma(s) \in \Lambda$.

Since $(s,s^{-1}) \in \cR$ for all $s \in S$ and since $\Gamma$ is generated by $S$ with relations $\cR$, it suffices to prove that $\gamma(s_1) \cdots \gamma(s_n) = e$ for all $(s_1,\ldots,s_n) \in \cR$. Fix $(s_1,\ldots,s_n) \in \cR$. Take a $\bG$-invariant vector $\xi \in \ell^2(S^n)$ for the $n$-fold tensor power of $U$ with $\xi(s_1,\ldots,s_n) \neq 0$. The fact that $\xi$ is an invariant vector is expressed by the relation
$$\sum_{t_1,\ldots,t_n \in S} u_{k_1 t_1} \cdots u_{k_n t_n} \, \xi(t_1,\ldots,t_n) = \xi(k_1,\ldots,k_n) \, 1 \quad\text{for all $k_1,\ldots,k_n \in S$.}$$
Applying $\psi$, we get that $\xi(k_1,\ldots,k_n) \, \gamma(k_1) \cdots \gamma(k_n) = \xi(k_1,\ldots,k_n) \, 1$ in $\C[\Lambda]$ for all $k_1,\ldots,k_n \in S$. Since $\xi(s_1,\ldots,s_n) \neq 0$, this means that $\gamma(s_1) \cdots \gamma(s_n) = e$.
\end{proof}

\begin{example}\label{ex.quantizations}
Let $n \geq 1$ be an integer. Define the $2n \times 2n$ matrix $J_n = \bigl(\begin{smallmatrix} 0 & I_n \\ I_n & 0\end{smallmatrix}\bigr)$ and consider the Kac type compact quantum group $\bG = A_o(J_n)$, which is isomorphic with $A_o(2n) = A_o(I_{2n})$. Denote by $U = (u_{ij})$ the fundamental representation of $A_o(J_n)$. Also consider $A_u(n)$ with its fundamental representation $V$. Denote by $a_1,\ldots,a_n$ the free generators of $\F_n$. We have the surjective Hopf $*$-algebra homomorphisms
\begin{align*}
& \pi_1 : \cO(A_o(J_n)) \to \cO(A_u(n)) : \pi_1(U) = \begin{pmatrix} V & 0 \\ 0 & \overline{V}\end{pmatrix} \quad\text{and} \\
& \pi_2 : \cO(A_u(n)) \to \C[\F_n] : \pi_2(v_{ij}) = \delta_{i,j} a_i \; .
\end{align*}
By Lemma \ref{lem.quantization}, the composition $\pi_2 \circ \pi_1$ turns $\widehat{A_o(J_n)}$ into a quantization of $\F_n$. Indeed, it suffices to observe that the vector $\sum_{i=1}^{2n} (e_i \ot J_n e_i)$ is, by definition, a $\bG$-invariant vector for $U^{\ot 2}$ whose ``evaluation on $(a_i,a_i^{-1})$'' equals $1$ for all $i$. A fortiori, the intermediate $\widehat{A_u(n)}$ is a quantization of $\F_n$.

With a similar reasoning, the surjective Hopf $*$-algebra homomorphism $\pi : \cO(A_o(n)) \to (\Z/2\Z)^{\ast n} : \pi(u_{ij}) = \delta_{i,j} a_i$ turns $\widehat{A_o(n)}$ into a quantization of $(\Z/2\Z)^{\ast n}$.
\end{example}

We conclude this section by computing a few examples of planar quantizations of discrete groups.

Recall from \cite{BBC07,Bic01} that the hyperoctahedral quantum group $H^+(d)$ is defined as the closed subgroup of $A_o(d)$ by adding the relation $u_{ij} u_{ik} = 0 = u_{ji} u_{ki}$ whenever $j \neq k$. The space of morphisms between tensor powers of $U$ is given by the linear span of all noncrossing partitions in which each block contains an even number of points. Using the matrix $J_d = \bigl(\begin{smallmatrix} 0 & I_d \\ I_d & 0\end{smallmatrix}\bigr)$, we can define, as in \cite[Section 3]{BaS10} and \cite[Example 3.5]{BNY15}, a twisted version of $H^+(2d)$ generated by the entries of a $2d$-dimensional unitary representation with relations expressing that
$$\sum_{i=1}^{2d} (e_i \ot J_d(e_i)) \quad\text{and}\quad \sum_{i=1}^{2d} (e_i \ot J_d(e_i) \ot e_i \ot J_d(e_i))$$
are invariant vectors in the $2$-fold, resp.\ $4$-fold, tensor power of $U$. Then, $H^+(J_d)$ is monoidally equivalent with $H^+(2d)$, but not isomorphic with $H^+(2d)$.

The following result can be deduced from Proposition \ref{prop.identify-with-quantum-isometry-group} and \cite[Theorem 5.1]{BhS10}. Since we can give a short proof in our context, we include it for completeness.

\begin{proposition}
The planar quantization of the free group $\F_d$ with its canonical symmetric generating set $S \subset \F_d$ given by the free generators and their inverses, is isomorphic with the dual of $H^+(J_{2d})$.

The planar quantization of the free product $\Gamma = (\Z/2\Z)^{\ast d}$ with its canonical symmetric generating set $S$ given by the free generators of order $2$ is isomorphic with the dual of $H^+(d)$.

In particular, if $\Pi$ denotes the $d$-regular tree, then the unitary tensor category $\cC(\cP,\Pi)$ is equivalent with the representation category of $H^+(d)$ and thus described by noncrossing partitions in which each block contains an even number of points.
\end{proposition}
\begin{proof}
Denote by $\bG_1$ the planar quantization of $\F_d$ and denote $\bG_2 = H^+(J_{2d})$. We realize both as generated by the entries of an $S \times S$ unitary matrix, where the defining relations for $\bG_1$ are given by the $\bG_1$-invariant vectors
$$\xi_n = \sum_{s \in S^n : s_1 \cdots s_n = e} (s_1,\ldots, s_n) \quad\text{for all $n \geq 2$,}$$
and where the defining relations for $\bG_2$ are given by the $\bG_2$-invariant vectors $\xi_2$ and $\eta = \sum_{s \in S} (s,s^{-1},s,s^{-1})$. In all these formulas, we denote for $s \in S^n$ by $(s) \in \ell^2(S^n)$ the canonical basis vector.

To prove the first statement, it suffices to prove that all $\xi_n$ are $\bG_2$-invariant and that $\eta$ is $\bG_1$-invariant.

Using $\xi_4$, we find that the operator $T : \ell^2(S^2) \to \ell^2(S^2)$ given by
$$\langle T (s_1 s_2),(t_1 t_2)\rangle = \begin{cases} 1 &\quad\text{if $s_1 s_2 = t_1 t_2$,}\\ 0 &\quad\text{if $s_1 s_2 \neq t_1 t_2$,}\end{cases}$$
is a $\bG_1$-morphism. When $s_1,s_2,t_1,t_2$ are generators in $\F_d$, the relation $s_1 s_2 = t_1 t_2$ precisely holds in two cases: first when $s_2 = s_1^{-1}$ and $t_2 = t_1^{-1}$, and secondly if $s_1 = t_1$ and $s_2 = t_2$. These two cases overlap in the case where $t_2 = s_2 = s_1^{-1} = t_1^{-1}$. It follows that $T = 1 + \xi_2 \xi_2^* - P$ where $P$ is the orthogonal projection onto $\lspan \{(ss^{-1}) \mid s \in S\}$. So, $P$ is a $\bG_1$-morphism. Since $\eta = (1 \ot P \ot 1)(\xi_2 \ot \xi_2)$, it follows that $\eta$ is $\bG_1$-invariant.

Conversely, using $\eta$, we find that $P$ is a $\bG_2$-morphism. Note that $\xi_n = 0$ when $n$ is odd. We prove by induction on $n$ that $\xi_{2n}$ is $\bG_2$-invariant. When $n = 1$, this holds by definition. Assume that $n \geq 2$ and that $\xi_{2n-2}$ is $\bG_2$-invariant. Define the $\bG_2$-morphism $Q_2 := 1 - P$ and, for all $n \geq 3$, the $\bG_2$-morphism
$$Q_n = (Q_2 \ot 1^{\ot (n-2)}) (1 \ot Q_2 \ot 1^{\ot (n-3)}) \cdots (1^{\ot (n-2)} \ot Q_2) \; .$$
Note that $Q_n$ is the orthogonal projection onto
$$\lspan \bigl\{(s_1,\ldots,s_n) \in S^n \bigm|  s_1 \cdots s_n \;\;\text{is a reduced word} \bigr\} \; .$$
When $s \in S^{2n}$ and $s_1 \cdots s_{2n} = e$, the word $s_1 \cdots s_{2n}$ is not a reduced word. So, the set $\Xi_n = \{s \in S^{2n} \mid s_1 \cdots s_{2n} = e\}$ can be partitioned into the subsets
$$\Xi_{n,k} = \bigl\{s \in S^{2n} \bigm| s_1 \cdots s_{2n} = e \; , \;\; s_1 \cdots s_k \;\;\text{is a reduced word, and}\;\; s_{k+1} = s_k^{-1}\bigr\}$$
for all $1 \leq k \leq 2n-1$. It then follows that
$$\xi_{2n} = (\xi_2 \ot \xi_{2n-2}) + \sum_{k=2}^{2n-1} (Q_k \ot 1^{\ot (2n-k)})(1^{\ot (k-1)} \ot \xi_2 \ot 1^{\ot (2n-k-1)}) \xi_{2n-2} \; .$$
So, $\xi_{2n}$ is $\bG_2$-invariant.

The proof of the second statement is entirely similar. As for the last statement, it suffices to note that the $d$-regular tree is the Cayley graph of $(\Z/2\Z)^{\ast d}$ with its canonical symmetric generating set.
\end{proof}

Motivated by the property~(T) discrete quantum groups defined in \cite{VV18}, we also introduce the following variant of planar quantizations of discrete groups, adapted to non-symmetric generating sets. Given a countable group $\Gamma$ with a, not necessarily symmetric, finite generating set $F \subset \Gamma$, we can define the planar quantization $\bG_1$ of $(\Gamma,F \cup F^{-1})$ as in Definition \ref{def.planar-quantization-discrete-group}. We now define a canonical intermediate
$$\cO(\bG_1) \to \cO(\bG) \to \C[\Gamma] \; .$$

\begin{definition}\label{def.G-with-signs}
Let $\Gamma$ be a countable group generated by a finite subset $F \subset \Gamma$. We define $\bG$ as generated by the entries $u_{st}$ of an $F \times F$ unitary representation $U$ such that the $F^{-1} \times F^{-1}$ matrix $V$ with entries $v_{s^{-1}t^{-1}} = u_{st}^*$ is unitary and such that for every $n \geq 1$ and every $\eps \in \{\pm 1\}^n$, the vector
\begin{equation}\label{eq.relations-with-signs}
\xi_{n,\eps} \in \ell^2(F^{\eps_1} \times \cdots \times F^{\eps_n}) : \xi_{n,\eps}(s_1,\ldots,s_n) = \begin{cases} 1 &\quad\text{if $s_1 \cdots s_n = e$ in $\Gamma$,}\\ 0 &\quad\text{otherwise,}\end{cases}
\end{equation}
is invariant under the tensor product of $U^{\eps_i}$, $i=1,\ldots,n$, where $U^1 := U$ and $U^{-1} = V$.
\end{definition}

Note that by definition, $\bG$ is of Kac type and $\overline{U} \cong V$. Note that we have the canonical surjective Hopf $*$-algebra homomorphism $\rho : \cO(\bG) \to \C[\Gamma] : \rho(u_{st}) = \delta_{s,t} \, s$.

When $\widehat{\bG_1}$ is the planar quantization of $(\Gamma,F \cup F^{-1})$ with fundamental representation $W$ for $\bG_1$, then $\bG$ is the closed quantum subgroup of $\bG_1$ that arises by adding the relations $w_{st} = 0 = w_{ts}$ whenever $s \in F$, $t \not\in F$ and whenever $s \in F^{-1}$, $t\not\in F^{-1}$. The corresponding surjective Hopf $*$-algebra homomorphism $\pi : \cO(\bG_1) \to \cO(\bG)$ is given by
$$\pi(w_{st}) = \begin{cases} u_{st} &\quad\text{if $s,t \in F$,}\\ u_{s^{-1}t^{-1}}^* &\quad\text{if $s,t \in F^{-1}$,}\\ 0 &\quad\text{otherwise.}\end{cases}$$
In particular, if $F$ happens to be symmetric, then $\pi$ is an isomorphism. The following definition is therefore unambiguous.

\begin{definition}\label{def.planar-quantization-not-symmetric}
Given a countable group $\Gamma$ with a, not necessarily symmetric, finite generating set $F \subset \Gamma$, we consider the compact quantum group $\bG$ defined in \ref{def.G-with-signs} and we call its dual $\widehat{\bG}$ the planar quantization of $(\Gamma,F)$.
\end{definition}

Note that when $F$ is ``partially symmetric'', in the sense that $F \neq F \cap F^{-1} \neq \emptyset$, the fundamental representation $U$ is reducible because $u_{st} = 0 = u_{ts}$ whenever $s \in F \cap F^{-1}$ and $t \in F \setminus F^{-1}$.

In \cite{CMSZ91}, the concept of a triangle presentation was introduced, and this is a subset $T \subset F \times F \times F$ of the triple power of a finite set $F$, satisfying a number of combinatorial conditions (see also \cite[Definition 2.3]{VV18}). The groups $\Gamma_T$ with generating set $F$ and relations $str=e$ for all $(s,t,r) \in T$ are, by \cite{CMSZ91}, precisely the groups that act simply transitively on the vertices of an $\widetilde{A_2}$ building. In \cite[Definition 5.1]{VV18} the compact quantum group $\bG_T$ was defined, generated by the entries $u_{ij}$ of an $F \times F$ unitary representation $U$, such that the vector $\xi_T \in \ell^2(F \times F \times F)$ defined by
\begin{equation}\label{eq.xiT}
\xi_T(s,t,r) = \begin{cases} 1 &\quad\text{if $(s,t,r) \in T$,}\\ 0 &\quad\text{otherwise,}\end{cases}
\end{equation}
is invariant under the $3$-fold tensor power of $U$.

It was then proven in \cite{VV18} that in numerous cases, the dual $\widehat{\bG_T}$ is a discrete quantum group with property~(T). We now prove that $\widehat{\bG_T}$ is the planar quantization of $(\Gamma_T,F)$, which demonstrates that the planar quantization procedure does lead to very interesting discrete/compact quantum groups.

\begin{proposition}
Let $T \subset F \times F \times F$ be a triangle presentation. Then $\widehat{\bG_T}$ is isomorphic with the planar quantization of $(\Gamma_T,F)$.
\end{proposition}
\begin{proof}
Let $\widehat{\bG_p}$ be the planar quantization of $(\Gamma_T,F)$, so that $\bG_p$ is generated by the entries $u_{st}$ of an $F \times F$ unitary representation $U$ subject to the relations in Definition \ref{def.G-with-signs}.

As in \cite{VV18}, given $s,t \in F$, we write $s \to t$ if there exists a (necessarily unique) $r \in F$ such that $(s,t,r) \in T$. Throughout the proof, we will repeatedly use \cite[Proposition 3.2]{CMSZ91} saying that for all $m,k \geq 0$ with $m+k \geq 1$,
\begin{align*}
s_1 \cdots s_m t_k^{-1} \cdots t_1^{-1} \neq e \;\;\text{in $\Gamma_T$ when }&\text{$s_i,t_j \in F$ for all $i,j$, $s_i \not\to s_{i+1}$ for all $1 \leq i \leq m-1$,}\\ &\text{$t_i \not\to t_{i+1}$ for all $1 \leq i \leq k-1$ and $s_m \neq t_k$.}
\end{align*}
In particular, given $s,t,r \in F$, we have that $str = e$ in $\Gamma_T$ iff $(s,t,r) \in T$. Therefore, $\xi_T = \xi_{3,+++}$. With some abuse of notation, we also denote by $U$ the fundamental representation of $\bG_T$ and denote by $V$ the $F^{-1} \times F^{-1}$ matrix with entries $v_{s^{-1}t^{-1}} = u_{st}^*$. It thus suffices to prove that all vectors $\xi_{n,\eps}$ are $\bG_T$-invariant for all $n \geq 1$ and $\eps \in \{\pm 1\}^n$.

By definition, $\xi_{2,+-}$ and $\xi_{2,-+}$ are $\bG_T$-invariant. By \cite[Lemma 5.8]{VV18}, the orthogonal projections onto
\begin{alignat*}{2}
& \lspan\{(s s^{-1}) \mid s \in F\} \subset \ell^2(F \times F^{-1}) \;,\quad && \lspan\{(s^{-1} s) \mid s \in F\} \subset \ell^2(F^{-1} \times F) \; ,\\
& \lspan\{(s t) \mid s,t \in F, s \to t\} \subset \ell^2(F \times F) \; ,\quad && \lspan\{(t^{-1} s^{-1}) \mid s,t \in F, s \to t\} \subset \ell^2(F^{-1} \times F^{-1}) \; ,
\end{alignat*}
are all $\bG_T$-morphisms. Denoting $W = \{(s^{-1},t) \in F^{-1} \times F \mid s \neq t\}$ and $W' = \{(s,t^{-1}) \in F \times F^{-1} \mid s \neq t\}$, we also get that the projections $Q$, $Q'$ on the linear span of $W$, $W'$ are $\bG_T$-morphisms.

By the axioms of triangle presentations, for every $(s,t^{-1}) \in W'$, there is a unique $(r^{-1},l) \in W$ such that $s t^{-1} = r^{-1} l$ holds in $\Gamma_T$. We write $(r^{-1},l) = \si(s,t^{-1})$. We also denote by $\si$ the corresponding partial isometry from $\ell^2(F \times F^{-1})$ to $\ell^2(F^{-1} \times F)$, with $\si \si^* = Q$ and $\si^* \si = Q'$. One checks that
$$\si = Q (1 \ot \xi_T \ot 1 \ot \xi_{2,+-})^* (\xi_{2,-+} \ot 1 \ot \xi_T \ot 1) Q' \; .$$
So, $\si$ is a $\bG_T$-morphism.

Combining the diagonal projections above, the orthogonal projections $R_{k}$ and $R'_k$ on the linear spans of
\begin{align*}
& \bigl\{(s) \bigm| s \in F^k \;\text{and for all $i \in \{1,\ldots,k-1\}$, we have $s_i \not\to s_{i+1}$}\;\bigr\} \quad\text{and}\\
& \bigl\{(s) \bigm| s \in (F^{-1})^k \;\text{and for all $i \in \{1,\ldots,k-1\}$, we have $s_{i+1}^{-1} \not\to s_{i}^{-1}$}\;\bigr\} \; ,
\end{align*}
are $\bG_T$-morphisms.

For every $\eps \in \{\pm\}^n$, we denote by
$$\inv(\eps) = \sum_{i : \eps_i = -1} \#\{j \mid j > i \;\text{and}\; \eps_j = +1\}$$
the number of inversions of signs, so that $\inv(\eps) = 0$ iff $\eps$ is of the form $(++ \cdots + - \cdots -)$.

Since $e \not\in F$, we have that $\xi_{1,+} = 0$ and $\xi_{1,-} = 0$. Also, $st \neq e$ for all $s,t \in F$, so that $\xi_{2,++} = 0$ and $\xi_{2,--} = 0$. We already observed that $\xi_{2,+-}$ and $\xi_{2,-+}$ are $\bG_T$-invariant. Altogether, we have proven that $\xi_{n,\eps}$ is $\bG_T$-invariant whenever $n \in \{1,2\}$.

Define the order relation $(n',\eps') < (n,\eps)$ iff either $n' < n$ or ($n'=n$ and $\inv(\eps') < \inv(\eps)$). By induction, it suffices to write every $\xi_{n,\eps}$ with $n \geq 3$ in terms of $\bG_T$-morphisms and vectors $\xi_{n',\eps'}$ with $(n',\eps') < (n,\eps)$.

First consider the case where $\inv(\eps) > 0$. So, there exists a $k \in \{1,\ldots,n-1\}$ with $\eps_k = -1$ and $\eps_{k+1}=+1$. Write $\eps' = (\eps_1,\ldots,\eps_{k-1},+1,-1,\eps_{k+2},\ldots,\eps_n)$ and note that $\inv(\eps') < \inv(\eps)$. Also write $\eps\dpr = (\eps_1,\ldots,\eps_{k-1},\eps_{k+2},\ldots,\eps_n) \in \{\pm\}^{n-2}$. Then,
$$\xi_{n,\eps} = (1^{\ot (k-1)} \ot \xi_{2,-+} \ot 1^{\ot (n-k-1)}) \xi_{n-2,\eps\dpr} + (1^{\ot (k-1)} \ot \sigma \ot 1^{\ot (n-k-1)}) \xi_{n,\eps'} \; ,$$
as required.

Next consider the case where $\inv(\eps) = 0$, which we divide in three subcases: $\eps = (+ \cdots +)$, $\eps = (- \cdots -)$ and the case where $\eps$ consists of $m$ plus signs followed by $n-m$ minus signs with $1 \leq m \leq n-1$.

We make use of the following $\bG_T$-morphisms associated with $\xi_T$~:
\begin{align*}
& X : \ell^2(F^{-1}) \to \ell^2(F \times F) : X(s^{-1}) = \sum_{(t,r) \in F^2 : (s,t,r) \in T} (tr) \quad\text{and}\\
& Y : \ell^2(F) \to \ell^2(F^{-1} \times F^{-1}) : Y(s) = \sum_{(t,r) \in F^2 : (s,t,r) \in T} (r^{-1} t^{-1}) \; .
\end{align*}

First consider the case where $\eps = (+ \cdots +)$ consists of $n$ plus signs. For every $k \in \{1,\ldots,n-1\}$, denote by $\eps^{(k)} \in \{\pm\}^{n-1}$ a string of $n-2$ plus signs and a single minus sign in position $k$. Then,
$$\xi_{n,\eps} = (X \ot 1^{\ot (n-2)})\xi_{n-1,\eps^{(1)}} + \sum_{k=1}^{n-2} (R_{k+1} \ot 1^{\ot (n-k-1)})(1^{\ot k} \ot X \ot 1^{\ot (n-k-2)})\xi_{n-1,\eps^{(k+1)}} + R_n \xi_{n,\eps} \; .$$
By \cite[Proposition 3.2]{CMSZ91}, $R_n \xi_{n,\eps} = 0$, and we are done with the first case.

The case where $\eps = (- \cdots -)$ consists of $n$ minus signs, is handled analogously.

Finally consider the case where $\eps$ consists of $m$ plus signs followed by $n-m$ minus signs with $1 \leq m \leq n-1$. Define $\eps' = \{\pm\}^{n-2}$ consisting of $m-1$ plus signs followed by $n-m-1$ minus signs.
Note that
$$
\xi_{n,\eps} = (1^{\ot (m-1)} \ot \xi_{2,+-} \ot 1^{\ot (n-m-1)})\xi_{n-2,\eps'} + \eta \quad\text{with}\;\; \eta := (1^{\ot (m-1)} \ot Q' \ot 1^{\ot (n-m-1)}) \xi_{n,\eps} \; .
$$
When $1 \leq k \leq m-1$ or $m+1 \leq k \leq n-1$, denote by $\eps^{(k)} \in \{\pm\}^{n-1}$ the string of signs that one obtains by replacing in $\eps$ the two consecutive (same) signs $\eps_k \eps_{k+1}$ by the single sign $-\eps_k$.

If $m=1$, write $\eta' = \eta$. If $m \geq 2$, define $\eta' = (R_m \ot 1^{\ot (n-m)})(1^{\ot (m-1)} \ot Q' \ot 1^{\ot (n-m-1)})\xi_{n,\eps}$ and note that
\begin{align*}
\eta =  \; & (1^{\ot (m-1)} \ot Q' \ot 1^{\ot (n-m-1)})(X \ot 1^{\ot (n-2)})\xi_{n-1,\eps^{(1)}} \\
&\quad + \sum_{k=1}^{m-2} (R_{k+1} \ot 1^{\ot (m-k-2)} \ot Q' \ot 1^{\ot (n-m-1)})(1^{\ot k} \ot X \ot 1^{\ot (n-k-2)})\xi_{n-1,\eps^{(k+1)}} + \eta' \; .
\end{align*}
If $m = n-1$, by \cite[Proposition 3.2]{CMSZ91}, we have $\eta' = 0$ and we are done.

If $m \leq n-2$, we denote $T = (R_m \ot 1^{\ot (n-m)})(1^{\ot (m-1)} \ot Q' \ot 1^{\ot (n-m-1)})$ and write
\begin{align*}
\eta' =  \; & T (1^{\ot m} \ot Y \ot 1^{\ot (n-m-2)}) \xi_{n-1,\eps^{(m+1)}} \\
&\quad + \sum_{k=1}^{n-m-2} T (1^{\ot m} \ot R'_{k+1} \ot 1^{\ot (n-m-k-1)})(1^{\ot (m+k)} \ot Y \ot 1^{\ot (n-m-k-2)})\xi_{n-1,\eps^{(m+k+1)}} \\
&\quad\quad + T(1^{\ot m} \ot R'_{n-m}) \xi_{n,\eps} \; .
\end{align*}
By \cite[Proposition 3.2]{CMSZ91}, the last term is zero. This concludes the proof.
\end{proof}

\section{Quantum isomorphism of graphs and examples}\label{sec.quantum-isomorphism}

Recall that a matrix $(v_{ij})_{(i,j) \in I \times J}$ of bounded operators on a Hilbert space $H$ is called a \emph{magic unitary} if
\begin{equation}\label{eq.magic}
v_{ij} = v_{ij}^* = v_{ij}^2 \quad\text{and}\quad \sum_{k \in I} v_{kj} = 1 = \sum_{k \in J} v_{ik} \quad\text{strongly, for all $i \in I$, $j \in J$.}
\end{equation}
So, a magic unitary is nothing else than a unitary element $V \in B(H \ot \ell^2(J), H \ot \ell^2(I))$ for which all matrix entries are self-adjoint projections.

The notion of quantum isomorphism of two finite graphs is usually defined in the context of quantum strategies for the graph isomorphism game. In \cite[Theorem 4.4]{LMR17}, it was proven that such a \emph{quantum commuting isomorphism} of finite graphs is equivalent to finding a magic unitary that intertwines the adjacency matrices. That definition, \cite[Definition 4.1]{LMR17}, can be immediately generalized to infinite graphs as follows.

\begin{definition}\label{def.quantum-iso}
We call two graphs $\Pi$ and $\Pi'$ \emph{quantum isomorphic} if there exists a nonzero Hilbert space $H$ and a magic unitary $(v_{ij})_{(i,j) \in V(\Pi) \times V(\Pi')}$ on $H$ such that
\begin{equation}\label{eq.commute-adjacency}
\sum_{k \in V(\Pi') : k \sim j} v_{ik} = \sum_{k \in V(\Pi) : k \sim i} v_{kj} \quad\text{for all $i \in V(\Pi)$ and $j \in V(\Pi')$.}
\end{equation}
\end{definition}

In \cite[Definition 4.4]{BCEHPSW18}, an algebraic variant of quantum isomorphism was defined and also this definition has an immediate generalization to infinite graphs.

\begin{definition}\label{def.algebraic-iso}
We call two locally finite graphs $\Pi$ and $\Pi'$ \emph{algebraically isomorphic} if there exists a nonzero $*$-algebra $\cAtil$ generated by self-adjoint idempotents $(v_{ij})_{(i,j) \in V(\Pi) \times V(\Pi')}$ that are orthogonal, meaning that $v_{ik'} v_{jk'} = 0 = v_{ki'} v_{kj'}$ if $i \neq j$, $i' \neq j'$, and that satisfy
    $$\sum_{k \in I'} v_{ik} = 1 = \sum_{k \in I} v_{kj} \quad\text{strictly, and}\quad \sum_{k \in I' : k \sim j} v_{ik} = \sum_{k \in I : k \sim i} v_{kj} \quad\text{for all $i \in V(\Pi)$, $j \in V(\Pi')$.}$$
\end{definition}

Note that contrary to the Hilbertian setting in Definition \ref{def.quantum-iso}, self-adjoint idempotents that sum up to $1$ strictly in an abstract $*$-algebra are not necessarily orthogonal. Therefore we have to add this assumption in the definition of algebraic isomorphism.

In \cite[Theorem 7.16]{MR19} (see also \cite{MR20}), it is proven that two finite graphs $\Pi,\Pi'$ are quantum isomorphic if and only if for every finite planar graph $K$, the number of graph homomorphisms from $K$ to $\Pi$, resp.\ from $K$ to $\Pi'$, are equal. One cannot expect the same statement to hold for infinite graphs, because typically this homomorphism count only gives $0$ or $+\infty$. If $\Pi$ is connected and locally finite, it is thus more natural to count pointed homomorphisms, in the following way.

\begin{notation}
Let $K$ be a connected finite graph and $x \in V(K)$. Let $\Pi$ be a locally finite graph and $i \in V(\Pi)$. We write
$$\homom(K,x,\Pi,i) = \# \bigl\{ \vphi : V(K) \to V(\Pi) \bigm| \vphi \;\text{is a graph homomorphism and}\;\; \vphi(x) = i \; \bigr\} \; .$$
\end{notation}

Note that when $\Pi$ is a finite, vertex transitive graph, $\homom(K,x,\Pi,i) = |V(\Pi)|^{-1} \, \homom(K,\Pi)$.

\begin{definition}\label{def.planar-isomorphism}
Let $\Pi,\Pi'$ be connected locally finite graphs with vertex sets $I, I'$ and quantum orbits $(I_a)_{a \in \cE}$ and $(I'_a)_{a \in \cE'}$. We say that $\Pi$ and $\Pi'$ are \emph{planar isomorphic} if there exists a bijection $\rho : \cE \to \cE'$ such that for every connected finite planar graph $K$ and every $x \in V(K)$, we have
$$\homom(K,x,\Pi,i) = \homom(K,x,\Pi',i') \quad\text{whenever $i \in I_a$, $i' \in I'_{\rho(a)}$, $a \in \cE$.}$$
\end{definition}

For a planar graph $K$ and an arbitrary vertex $x \in V(K)$, the bi-labeled graph $\cK = (K,x,x)$ is planar in the sense of \cite[Definition 5.4]{MR19}. So, $\cK \in \cP(1,1)$ and
$$\homom(K,x,\Pi,i) = T^\cK_{ii} \; ,$$
where the right hand side is defined w.r.t.\ $\Pi$. Considering the equivalence relation $\approx$ on $V(\Pi)$ defined in \eqref{eq.equiv-rel-on-I} w.r.t.\ the graph category of planar bi-labeled graphs, we thus have by definition that for all planar graphs $K$, $\homom(K,x,\Pi,i) = \homom(K,x,\Pi,j)$ whenever $i \approx j$, i.e.\ whenever $i$ and $j$ lie in the same quantum orbit.

By \cite[Theorem 4.9]{BCEHPSW18} and \cite[Theorem 7.16]{MR19}, for finite graphs, quantum isomorphism, algebraic isomorphism and equality of homomorphism counts from planar graphs are all equivalent. Moreover, these properties are proven there, for finite graphs, to be equivalent to the existence of a specific kind of monoidal equivalence between the representation categories of the (compact) quantum automorphism groups. We now prove that these equivalences still hold for infinite graphs that are connected and locally finite. In particular, quantum isomorphism of connected locally finite graphs still implies monoidal equivalence of the locally compact quantum automorphism groups, see Remark \ref{rem.bigalois}.

In the following theorem, all notations, in particular all spaces $\Mor^\Pi(n,m)$ of matrices and all quantum orbit partitions $(I_a)_{a \in \cE}$, are taken w.r.t.\ the graph category of planar bi-labeled graphs, so that the associated locally compact quantum group is the quantum automorphism group $\QAut \Pi$ of the graph $\Pi$. Recall the notation
$$\cL(n,m) = \bigl\{ \cK = (K,x,y) \in \cP(n+1,m+1) \bigm| \text{$K$ is connected, $x_0 = y_0$ and $x_n = y_m$}\, \bigr\} \; .$$

\begin{theorem}\label{thm.quantum-iso-criteria}
Let $\Pi,\Pi'$ be connected locally finite graphs with vertex sets $I, I'$ and quantum orbits $(I_a)_{a \in \cE}$ and $(I'_a)_{a \in \cE'}$. Then the following statements are equivalent.
\begin{enumlist}
\item The graphs $\Pi$ and $\Pi'$ are \emph{quantum isomorphic} in the sense of Definition \ref{def.quantum-iso}.
\item The graphs $\Pi$ and $\Pi'$ are \emph{algebraically isomorphic} in the sense of Definition \ref{def.algebraic-iso}.
\item The graphs $\Pi$ and $\Pi'$ are \emph{planar isomorphic} in the sense of Definition \ref{def.planar-isomorphism}.
\item The unitary $2$-categories $\cC(\cP,\Pi)$ and $\cC(\cP,\Pi')$ are equivalent in the following way: there exists a bijection $\rho : \cE \to \cE'$ and bijective linear maps
\begin{multline*}
\Psi_{a-b}^{n,m} : \Mor_{a-b}^{\Pi}(n,m) \to \Mor_{\rho(a)-\rho(b)}^{\Pi'}(n,m) \quad\text{satisfying}\\ \Psi_{a-b}^{n,m}(T^\cK_{\Pi}(1_a \ot 1^{\ot (m-1)} \ot 1_b)) = T^\cK_{\Pi'}(1_{\rho(a)} \ot 1^{\ot (m-1)} \ot 1_{\rho(b)})
\end{multline*}
for all $a,b \in \cE$ and all $\cK \in \cL(n,m)$. This equivalence preserves the left and the right dimension of Definition \ref{def.dimensions}.
\end{enumlist}
\end{theorem}

When one of the conditions in Theorem \ref{thm.quantum-iso-criteria} holds, it also follows that the algebraic quantum groups $\QAut \Pi$ and $\QAut \Pi'$ are monoidally equivalent in the sense of \cite[Definition 2.17]{DC07} and hence, the corresponding locally compact quantum groups are monoidally equivalent as well, by \cite{DC07,DC08}. We discuss this in Remark \ref{rem.bigalois} below.

\begin{proof}
1 $\Rightarrow$ 2. Let $V = (v_{ij})_{(i,j) \in I \times I'}$ be a quantum isomorphism between $\Pi$ and $\Pi'$ on the Hilbert space $H$. Denote by $\cAtil \subset B(H)$ the $*$-algebra generated by the projections $v_{ij}$. Multiplying \eqref{eq.commute-adjacency} on the left by $v_{it}$ and on the right by $v_{sj}$, we find that
$$\rel(t,j) \, v_{it} v_{sj} = \rel(i,s) \, v_{it} v_{sj} \quad\text{for all $i,s \in I$, $t,j \in I'$,}$$
where $\rel(k,l) = 1$ if $k \sim l$ and $\rel(k,l) = 0$ if $k \not\sim l$. It follows that
$$\rel(i_0,i_1) \, \cdots \, \rel(i_{n-1},i_n) \, v_{i_0 j_0} \cdots v_{i_n,j_n} = \rel(j_0,j_1) \, \cdots \, \rel(j_{n-1},j_n) \, v_{i_0 j_0} \cdots v_{i_nj_n}$$
for all $i \in I^{n+1}$ and $j \in {I'}^{n+1}$. Summing over $i_1,\ldots,i_{n-1} \in I$ and $j_1,\ldots,j_{n-1} \in I'$, it follows that
$$p_n(i,s) \, v_{ij} v_{st} = p_n(j,t) \, v_{ij} v_{st} \quad\text{for all $i,s \in I$, $j,t \in I'$,}$$
where $p_n(i,s)$ denotes the number of paths of length $n$ from $i$  to $s$ in $\Pi$, resp.\ $\Pi'$. We conclude that $v_{ij} v_{st} = 0$ when the distance $d(i,s)$ in $\Pi$ differs from the distance $d(j,t)$ in $\Pi'$. So, since $\Pi$ and $\Pi'$ are locally finite, if we multiply the infinite sums in \eqref{eq.magic} on the left or on the right with an element $v_{st}$, only finitely many nonzero terms remain. It follows that these infinite sums also converge to $1$ strictly, in the algebraic sense. So, $\cAtil$ defines an algebraic isomorphism between $\Pi$ and $\Pi'$.

2 $\Rightarrow$ 4. Fix an algebraic isomorphism $\cAtil$ between $\Pi$ and $\Pi'$. Note that since $\cAtil$ contains a family of elements that sum up strictly to $1$, the $*$-algebra $\cAtil$ is nondegenerate. For every $n \geq 1$, consider the $I^{n} \times {I'}^{n}$ matrix $V_n$ with entries in $\cAtil$ defined by
$$V_n(i,j) = v_{i_1 j_1} \cdots v_{i_n j_n} \; .$$
Recall that we say that an infinite matrix $T$ is of finite type if for every $i$, there are only finitely many $k$ for which $T_{ik} \neq 0$, and for every $j$, there are only finitely many $k$ such that $T_{kj} \neq 0$. For every $\cK \in \cP_1(n,m)$, we have the finite type matrices $T^\cK_\Pi$ and $T^\cK_{\Pi'}$. Define
$$\cP_0(n,m) = \bigl\{ \cK \in \cP_1(n,m) \bigm| V_n T^\cK_{\Pi'} = T^\cK_{\Pi} V_m \bigr\} \; .$$
Repeating the argument given in the proof of Theorem \ref{thm.lc-qaut}, we find that $\cP_0 = \cP_1$.

We now prove the following claim: if $T$ is a finite type $I^{n} \times I^{m}$ matrix with entries in $\C$ and if $T V_m = 0$, then $T = 0$. To prove this claim, assume that $T V_m = 0$. Then for all $i \in I^{n}$, we can choose a finite subset $\cF_i \subset I^{m}$ such that $T_{ik} = 0$ if $k \not\in \cF_i$ and such that
$$\sum_{k \in \cF_i} T_{ik} \, v_{k_1 j_1} \cdots v_{k_m j_m} = 0 \quad\text{for all $j \in {I'}^{m}$.}$$
Let $a \in \cAtil$ be any nonzero element. Multiplying the above expression on the left by $a$ and on the right by $v_{s_m j_m} \cdots v_{s_1 j_1}$ for an arbitrary $s \in \cF_i$, we can sum first over $j_m$, then over $j_{m-1}$, up to summing over $j_1$ and conclude that $a T_{is} = 0$. So, $T_{is} = 0$ for all $i,s$, proving the claim. We similarly get that if $S$ is a finite type scalar matrix and $V_n S = 0$, then $S = 0$.

We next claim that when $T$ is a finite type scalar $I^{n} \times I^{m}$ matrix and $S$ is a finite type scalar ${I'}^{n} \times {I'}^{m}$ matrix such that $T V_m = V_n S$, then also $T^* V_n = V_m S^*$. So, we assume that for all $i \in I^{n}$ and $j \in {I'}^{m}$,
$$\sum_{k \in I^{m}} T_{ik} V_m(k,j) = \sum_{k \in {I'}^{n}} S_{kj} V_n(i,k) \; .$$
Fix arbitrary $a \in \cAtil$, $s \in I^m$ and $t \in {I'}^n$. Write $\ibar = (i_n,\ldots,i_1)$. We multiply the above sum on the left by $a V_m(\overline{s},\jbar)$ and on the right by $V_n(\ibar,\overline{t})$. Note that there are only finitely many $i,j,k$ such that
$$T_{ik} \, a \, V_m(\overline{s},\jbar) \, V_m(k,j) \, V_n(\ibar,\overline{t}) \neq 0 \quad\text{or}\quad S_{kj} \, a \, V_m(\overline{s},\jbar) \, V_n(i,k) \, V_n(\ibar,\overline{t}) \neq 0 \; .$$
So we may sum over $i,k,j$ in any order and find that
$$a \sum_{i \in I^n} T_{is} \, V_n(\ibar,\overline{t}) = a \sum_{j \in {I'}^m} S_{tj} \, V_m(\overline{s},\jbar) \; .$$
Taking the adjoint, it follows that $(T^* V_n)_{st} a^* = (V_m S^*)_{st} a^*$ for all $a \in \cAtil$ and $s \in I^m$, $t \in {I'}^n$. Since $\cAtil$ is nondegenerate, this means that $T^* V_n = V_m S^*$.

Define $\cM(n,m)$ as the set of all finite type scalar matrices $T$ for which there exists a finite type scalar matrix $S$ such that $T V_m = V_n S$. By the claim above, $S$ is unique and we denote $S = \Phi(T)$. We similarly define $\cM'(n,m)$ as the set of finite type scalar matrices $S$ for which there exists a finite type scalar matrix $T$ with $T V_m = V_n S$. Obviously, $\cM'(n,m) = \Phi(\cM(n,m))$ and $\Phi$ is a bijective linear map. By the second claim, $\Phi(T^*) = \Phi(T)^*$. By construction, $\Phi$ is multiplicative whenever the sizes of the matrices match.

We have proven above that for every $\cK \in \cP_1(n,m)$, we have $T^\cK_{\Pi} \in \cM(n,m)$ and $\Phi(T^\cK_\Pi) = T^\cK_{\Pi'}$. In particular, $\Phi$ is a $*$-isomorphism from $\Mor^\Pi(0,0)$ onto $\Mor^{\Pi'}(0,0)$. We now deduce from this that for every $a \in \cE$, we have that $1_a \in \cM(1,1)$ and $\Phi(1_a) = 1_{\rho(a)}$, where $\rho : \cE \to \cE'$ is a bijection.

We want to define $\rho$ by saying that $\rho(a) = b$ if there exist $i \in I_a$ and $j \in I'_b$ such that $v_{ij} \neq 0$. To make this well-defined, we first note that for every $i \in I$, there exists at least one $j \in I'$ with $v_{ij} \neq 0$, because these elements sum up strictly to $1$. Next, we prove the following: if $i \approx s$ and if $v_{ij} \neq 0$ and $v_{st} \neq 0$, then $j \approx t$. Let $S \in \Mor^{\Pi'}(0,0)$ be arbitrary and put $T = \Phi^{-1}(S)$. Since $T V = V S$, we find that $S_j v_{ij} = T_i v_{ij}$ and $S_t v_{st} = T_s v_{st}$. Thus, $S_j = T_i$ and $S_t = T_s$. Since $i \approx s$, we have $T_i = T_s$. So, $S_j = S_t$ for all $S \in \Mor^{\Pi'}(0,0)$, so that $j \approx t$. We conclude that the map $\rho : \cE \to \cE'$ is well-defined. Similarly, its inverse is well-defined and $\rho$ is thus a bijection. By construction, $1_a V = V 1_{\rho(a)}$, so that $1_a \in \cM(1,1)$ and $\Phi(1_a) = 1_{\rho(a)}$.

We have now proven that for all $\cK \in \cP_1(n,m)$ and all $a,b \in \cE$,
$$\Phi\bigl(T^\cK_\Pi (1_a \ot 1^{\ot (m-2)} \ot 1_b)\bigr) = T^\cK_{\Pi'} (1_{\rho(a)} \ot 1^{\ot (m-2)} \ot 1_{\rho(b)}) \; .$$
We can thus define $\Psi_{a-b}^{n,m}$ by restricting $\Phi : \cM(n+1,m+1) \to \cM'(n+1,m+1)$ to the subspace $\Mor^\Pi_{a-b}(n,m) \subset \cM(n+1,m+1)$.

4 $\Rightarrow$ 1.\ Repeating the construction of Section \ref{sec.construction-B}, we obtain a $*$-algebra $\cBtil$ that is linearly spanned by elements $F_n(\xi,\xi')$ for $n \geq 0$, $\xi \in \cF(I^{n+1})$ and $\xi' \in \cF({I'}^{n+1})$, such that
\begin{align*}
& F_n(\xi,\xi') \, F_m(\eta,\eta') = F_{n+m}(\xi \ot_I \eta,\xi' \ot_{I'} \eta') \; , \quad F_n(\xi,\xi')^* = F_n(J_n \xi,J_n \xi') \quad\text{and}\\
& F_n(T \xi,\xi') = F_m(\xi,\Psi^{n,m}_{a-b}(T)^* \xi') \; ,
\end{align*}
for all $T \in \Mor_{a-b}^{\Pi}(n,m)$, $a, b \in \cE$, $\xi \in \cF(I^{m+1})$, $\xi' \in \cF({I'}^{n+1})$. Moreover, $\cBtil$ admits the faithful positive functional $\vphitil : \cBtil \to \C$ given by
$$\vphitil(F_n(\xi,\xi')) = \sum_{a \in \cE} \; \sum_{i \in I_a,j \in I'_{\rho(a)}} \; \sum_{V \in \ibf^{\Pi}_{a-a}(n,0)} \; \langle (i), V^*(\xi) \rangle \; \langle \Psi^{n,0}_{a-a}(V)^*(\xi'), (j) \rangle \; .$$
Repeating the constructions of Section \ref{sec.quantum-aut-group} and Lemma \ref{lem.iso-A-B}, we next find a nondegenerate $*$-algebra $\cAtil$ generated by elements $(w_{ij})_{(i,j) \in I \times I'}$ giving an algebraic isomorphism of $\Pi$ and $\Pi'$, and such that $\cAtil$ admits a positive faithful functional $\theta : \cAtil \to \C$.

Denote by $H$ the Hilbert space completion of $\cAtil$ w.r.t.\ $\langle x,y \rangle = \theta(y^* x)$. For every $x \in \cAtil$ and for every self-adjoint idempotent $p \in \cA$, we have that $(px)^* (p x) + (x-px)^* (x-px) = x^* x$. Applying $\theta$ and using positivity, it follows that $\theta((px)^* (p x)) \leq \theta(x^* x)$. So, left multiplication by $p$ defines a bounded operator on $H$. Since $\cAtil$ is generated by the self-adjoint idempotents $w_{ij}$, it follows that left multiplication by any $x \in \cAtil$ defines a bounded operator $\pi(x) \in B(H)$. Now the self-adjoint projections $v_{ij} := \pi(w_{ij})$ define a quantum isomorphism between $\Pi$ and $\Pi'$.

3 $\Rightarrow$ 4. Using the notation of Lemma \ref{lem.crucial-technical-Mor-lemma}, we note that for all $a \in \cE$, the vector space $\Mor_{a-}^{\Pi}(n,m)$ carries a positive definite scalar product given by $\langle T, S \rangle = \Tr_\ell(S^* T)$. The same holds for $\Mor_{a-}^{\Pi'}(n,m)$. The proof of Lemma \ref{lem.crucial-technical-Mor-lemma} also provides $\cR_m \in \cP(2m+1,1)$ such that for all $\cK_1,\cK_2 \in \cL(n,m)$, $a \in \cE$, $i \in I_a$ and $k \in \{1,2\}$, we have $(\cK_k \ot 1^{\ot m}) \circ \cR_m \in \cL(n+m,0)$ and
$$\langle T^{\cK_1}_\Pi (1_a \ot 1^{\ot m}) , T^{\cK_2}_\Pi (1_a \ot 1^{\ot m}) \rangle = (T^\cS_\Pi)_i \quad\text{where $\cS = \cR_m^* \circ (\cK_2^* \circ \cK_1 \ot 1^{\ot m}) \circ \cR_m$.}$$
An analogous formula holds w.r.t.\ $\Pi'$. Also note that $\cS \in \cL(0,0)$ and note that statement 3 is precisely saying that
\begin{equation}\label{eq.equality-here}
(T^\cS_\Pi)_i = (T^\cS_{\Pi'})_j \quad\text{for all $\cS \in \cL(0,0)$, $i \in I_a$, $j \in I'_{\rho(a)}$, $a \in \cE$.}
\end{equation}
We thus conclude that
$$\langle T^{\cK_1}_\Pi (1_a \ot 1^{\ot m}) , T^{\cK_2}_\Pi (1_a \ot 1^{\ot m}) \rangle = \langle T^{\cK_1}_{\Pi'} (1_{\rho(a)} \ot 1^{\ot m}) , T^{\cK_2}_{\Pi'} (1_{\rho(a)} \ot 1^{\ot m}) \rangle$$
for all $\cK_1,\cK_2 \in \cL(n,m)$. So, there are unique unitary operators $\Psi^{n,m}_{a-}$ from $\Mor_{a-}^{\Pi}(n,m)$ onto $\Mor_{a-}^{\Pi'}(n,m)$ mapping $T^\cK_\Pi(1_a \ot 1^{\ot m})$ to $T^\cK_{\Pi'}(1_{\rho(a)} \ot 1^{\ot m})$.

Recall from Lemma \ref{lem.crucial-technical-Mor-lemma} that $\Mor_{a-b}^\Pi(n,m) \subset \Mor_{a-}^\Pi(n,m)$. To conclude that statement~4 holds, we fix $a,b \in \cE$ and $\cK \in \cL(n,m)$. We have to prove that
\begin{equation}\label{eq.needed-conclusion}
\Psi^{n-m}_{a-}\bigl(T^\cK_\Pi(1_a \ot 1^{\ot (m-1)} \ot 1_b) \bigr) = T^\cK_{\Pi'}(1_{\rho(a)} \ot 1^{\ot (m-1)} \ot 1_{\rho(b)}) \; .
\end{equation}
As we have seen in the proof of Lemma \ref{lem.crucial-technical-Mor-lemma}, we can pick a finite subset $\cE_0 \subset \cE$ such that
$$T^\cK_\Pi(1_a \ot 1^{\ot (m-1)} \ot 1_c)= 0 \quad\text{and}\quad T^\cK_{\Pi'}(1_{\rho(a)} \ot 1^{\ot (m-1)} \ot 1_{\rho(c)}) = 0$$
for all $c \in \cE \setminus \cE_0$. Since $\Mor^\Pi(0,0)$ is a $*$-algebra, we can then choose $T \in \Mor^\Pi(0,0)$ such that $T_i = 0$ whenever $i \in I_c$, $c \in \cE_0 \setminus \{b\}$ and $T_i = 1$ whenever $i \in I_b$. Take $\cS_r \in \cL(0,0)$ and $\al_r \in \C$ such that $T = \sum_{r=1}^k \al_r T^{\cS_r}_\Pi$ and write $S = \sum_{r=1}^k \al_r T^{\cS_r}_{\Pi'}$. By \eqref{eq.equality-here}, we get that $S_j = 0$ for all $j \in I'_c$, $c \in \rho(\cE_0) \setminus \{\rho(b)\}$, and $S_j = 1$ for all $j \in I'_{\rho(b)}$.

For every $r$, we have $T^\cK_\Pi(1^{\ot m} \ot T^{\cS_r}_\Pi) = T^{\cK \circ (1^{\ot m} \ot \cS_r)}_\Pi$, and similarly w.r.t.\ $\Pi'$. It follows that
$$\Psi^{n,m}_{a-}\bigl(T^\cK_\Pi(1_a \ot 1^{\ot (m-1)} \ot T) \bigr) = T^\cK_{\Pi'}(1_{\rho(a)} \ot 1^{\ot (m-1)} \ot S) \; .$$
By the properties of $T$ and $S$ and by the definition of $\cE_0$, we have that
\begin{align*}
& T^\cK_\Pi(1_a \ot 1^{\ot (m-1)} \ot T) = T^\cK_\Pi(1_a \ot 1^{\ot (m-1)} \ot 1_b) \quad\text{and} \\
& T^\cK_{\Pi'}(1_{\rho(a)} \ot 1^{\ot (m-1)} \ot S) = T^\cK_{\Pi'}(1_{\rho(a)} \ot 1^{\ot (m-1)} \ot 1_{\rho(b)}) \; .
\end{align*}
So, \eqref{eq.needed-conclusion} is proven.

4 $\Rightarrow$ 3. We apply statement~4 for $n=m=0$. Taking for $\cK \in \cL(0,0)$ the identity bi-labeled graph $\cM^{1,1}$, we get that $\Psi^{0-0}_{a-a}(1_a) = 1_{\rho(a)}$. Then taking an arbitrary $\cS \in \cL(0,0)$ and taking arbitrary $a \in \cE$, $i \in I_a$ and $j \in I_{\rho(a)}$, we get that $T^\cS_\Pi 1_a = (T^\cS_\Pi)_i 1_a$ and $T^\cS_{\Pi'} 1_{\rho(a)} = (T^\cS_{\Pi'})_j 1_{\rho(a)}$. Since $\Psi^{0-0}_{a-a}$ is linear, it follows that $(T^\cS_\Pi)_i = (T^\cS_{\Pi'})_j$, i.e.\ \eqref{eq.equality-here} holds. This precisely says that $\Pi$ and $\Pi'$ are planar isomorphic.
\end{proof}

\begin{remark}\label{rem.bigalois}
The $*$-algebra $\cAtil$ constructed in the proof of Theorem \ref{thm.quantum-iso-criteria}, which in particular realizes the algebraic quantum isomorphism between $\Pi$ and $\Pi'$, also provides in the following way a monoidal equivalence (in the sense of \cite[Definition 2.17]{DC07}) between the algebraic quantum groups $\QAut \Pi$ and $\QAut \Pi'$. The $*$-algebra $\cAtil$ is generated by $(w_{ij})_{i \in I,j \in I'}$. We realize $\QAut \Pi$ and $\QAut \Pi'$ as $(\cA,\Delta)$, resp.\ $(\cA',\Delta)$, with their canonical generators $(u_{ij})_{i,j \in I}$, resp.\ $(u'_{ij})_{i,j \in I'}$. It is now straightforward to check that the commuting coactions
\begin{align*}
& \al : \cAtil \to M(\cA \ot \cAtil) : \al(w_{ij}) = \sum_{k \in I} (u_{ik} \ot w_{kj}) \quad\text{and}\\
& \be : \cAtil \to M(\cAtil \ot \cA') : \be(w_{ij}) = \sum_{k \in I'} (w_{ik} \ot u'_{kj})
\end{align*}
turn $\cAtil$ into a bi-Galois object that defines the monoidal equivalence between $(\cA,\Delta)$ and $(\cA',\Delta)$. By the results of \cite{DC07}, this algebraic bi-Galois object can then be completed into a von Neumann algebraic bi-Galois object linking the locally compact quantum groups $\QAut \Pi$ and $\QAut \Pi'$, in the sense of \cite[Definition 4.1]{DC08}.
\end{remark}

In \cite[Theorem 6.4]{AMRSSV16}, examples of isomorphic but not quantum isomorphic finite graphs are given. This includes a concrete example of two graphs $\cG$ and $\cG'$ sharing the following properties: they have $24$ vertices, no loops, every vertex has degree $9$ and they are vertex transitive. We can use this to give examples of connected locally finite graphs $\Pi$ illustrating the following phenomena.

\begin{enumlist}
\item In Example \ref{ex.compact-noncompact}, we provide examples where $\Aut \Pi$ is compact, but $\QAut \Pi$ is noncompact and quantum vertex transitive.

\item In Example \ref{ex.compact-nonunimodular}, we provide examples where $\Aut \Pi$ is compact, but $\QAut \Pi$ is nonunimodular and quantum vertex transitive. In particular, the transitivity assumption in Proposition \ref{prop.unimodular} is essential.

\item In Example \ref{ex.nonunimodular-unimodular}, we provide examples where $\QAut \Pi$ is quantum vertex transitive and unimodular, but $\Aut \Pi$ is nonunimodular. So the converse of Proposition \ref{prop.unimodular} does not hold.
\end{enumlist}

All these examples arise from the following elementary construction. Let $B$ be a connected base graph with no loops and constant finite degree $d \geq 2$. Assume that for every vertex $v \in V(B)$, we are given a finite graph $\cG_v$. Assume that all these graphs $\cG_v$ have the same number of vertices, say $\kappa$, have no loops and are of constant degree $f < \kappa$. We then consider a product graph $\Pi$ whose vertex set is the disjoint union of $V(\cG_v)$, $v \in V(B)$, and where two vertices in the same $V(\cG_v)$ are connected in $\Pi$ iff they were connected in $\cG_v$, while for $i \in V(\cG_v)$ and $j \in V(\cG_w)$ with $v \neq w$, we connect $i$ to $j$ iff $v \sim w$ in $B$. Note that $\Pi$ is connected, has no loops and is of constant degree $f+d \kappa$.

We denote by $I = V(\Pi)$ the vertex set and define $\pi : I \to V(B)$ by $\pi(i) = v$ iff $i \in \cG_v$. As before, we denote by $\Mor(n,m)$ the space of $I^{n+1} \times I^{m+1}$ matrices associated with $\Pi$ and the graph category $\cP$ of planar bi-labeled graphs. Recall that $\Mor(1,1)$ consists of diagonal matrices and can be viewed as a $*$-algebra of functions on $I \times I$.

To make sure that $\Aut \Pi$ is easy to compute, we make the following assumption: for all $v \neq w$ in $V(B)$, there exists a $z \in V(B) \setminus \{v,w\}$ such that $z \sim v$ and $z \not\sim w$.

\begin{lemma}\label{lem.product-graph}
The function $S : I \times I \to \{0,1\}$ where $S(i,j) = 1$ iff $\pi(i) = \pi(j)$ belongs to $\Mor(1,1)$.

In particular, a permutation $\si$ of $V(\Pi)$ belongs to $\Aut \Pi$ if and only if there exists a $\rho \in \Aut B$ such that $\si(\cG_v) = \cG_{\rho(v)}$ for all $v \in V(B)$ and $\si$ restricts to an isomorphism between $\cG_v$ and $\cG_{\rho(v)}$ for all $v \in V(B)$.
\end{lemma}
\begin{proof}
Define $\cK = (K,x,y) \in \cL(1,1)$ with $V(K) = \{0,1,2\}$, $0 \sim 1 \sim 2$ and $x_0=y_0=0$, $x_1=y_1=2$, as illustrated in Figure \ref{fig.very-simple-K}.
\begin{figure}[h]
    \centering
    \begin{tikzpicture}
        \filldraw (0,0) circle (1.5pt) node[anchor=south]{$x_0$} node[anchor=north]{$y_0$};
        \filldraw (1.4,0) circle (1.5pt);
        \filldraw (2.8,0) circle (1.5pt) node[anchor=south]{$x_1$} node[anchor=north]{$y_1$};
        \draw (0,0) -- (1.4,0);
        \draw (1.4,0) -- (2.8,0);
    \end{tikzpicture}
    \caption{The bi-labeled graph $\cK \in \cL(1,1)$}\label{fig.very-simple-K}
\end{figure}

Then $\cK \in \cL(1,1)$ and the corresponding element $T \in \Mor(1,1)$ is given by $T(i,j) = \#\{ k \in I \mid k \sim i, k \sim j\}$. When $\pi(i) = \pi(j)$, we have that $\kappa d \leq T(i,j) \leq \kappa(d+1)$.
When $\pi(i) \sim \pi(j)$, by our assumption on $B$, there are at most $d-2$ vertices $z \in V(B) \setminus \{\pi(i),\pi(j)\}$ that are connected to both $\pi(i)$ and $\pi(j)$. Thus, $T(i,j) \leq 2f + (d-2) \kappa < \kappa d$. When $\pi(i) \neq \pi(j)$ and $\pi(i) \not\sim \pi(j)$, there are at most $d-1$ vertices $z \in V(B)$ that are connected to both $\pi(i)$ and $\pi(j)$. Thus, $T(i,j) \leq \kappa (d-1) < \kappa d$.

Choosing a polynomial $P$ that takes the value $1$ on the integers in $[\kappa d, \kappa(d+1)]$ and that takes the value $0$ on the integers in $[0,\kappa d - 1]$, it follows that $S = P(T)$ belongs to $\Mor(1,1)$. By construction, $S(i,j) = 1$ if $\pi(i) = \pi(j)$ and $S(i,j) = 0$ otherwise.

When $\si \in \Aut \Pi$, we have that $S(\si(i),\si(j)) = S(i,j)$ for all $(i,j) \in I \times I$. So, whenever $\pi(i) = \pi(j)$, we must have that $\pi(\si(i)) = \pi(\si(j))$. This precisely says that $\sigma$ globally respects the partition of $\Pi(V)$ into the subsets $(V(\cG_v))_{v \in V(B)}$.
\end{proof}

In the following examples, we make use of the concrete quantum isomorphic, but nonisomorphic finite graphs $\cG$, $\cG'$ of \cite{AMRSSV16} mentioned above. We use the following observation: if for all $v \in V(B)$, we are given a quantum isomorphism (in the sense of Definition \ref{def.quantum-iso}) between finite graphs $\cG_v$ and $\cG'_v$, then the direct sum over $v \in V(B)$ of these magic unitaries, defines a quantum isomorphism between the product graphs $\Pi$ and $\Pi'$ constructed from $(\cG_v)_{v \in V(B)}$, resp.\ $(\cG'_v)_{v \in V(B)}$, over the base graph $B$.

\begin{example}\label{ex.compact-noncompact}
We construct a connected locally finite graph $\Pi$ such that $\Aut \Pi$ is compact while $\QAut \Pi$ is noncompact and quantum vertex transitive.

Take the base graph $B$ to be $V(B) = \Z$ and $n \sim m$ iff $|n-m| = 1$. Define $\Pi$ as the product graph given by $\cG_n = \cG$ for all $n \neq 0$ and $\cG_0 = \cG'$. By Lemma \ref{lem.product-graph}, the classical automorphism group $\Aut \Pi$ is compact. In particular, all classical orbits are finite. For the following reason, $\QAut \Pi$ is quantum vertex transitive and, in particular, noncompact by Proposition \ref{prop.criterion-compact}.

Define $\Pi'$ as the product graph given by $\cG_n = \cG$ for all $n \in \Z$. Since $\cG$ and $\cG'$ are quantum isomorphic, we get that $\Pi$ and $\Pi'$ are quantum isomorphic. By Theorem \ref{thm.quantum-iso-criteria}, it suffices to prove that $\QAut \Pi'$ is quantum vertex transitive. But this is immediate, because $\Pi'$ is vertex transitive.
\end{example}

To give more elaborate unimodular vs.\ nonunimodular examples, we make use of (variants of) the grandparent graph of \cite{Tro85} as a base graph $B$ (see e.g.\ \cite[Example 7.1]{LP16}).

\begin{example}\label{ex.compact-nonunimodular}
We construct a connected locally finite graph $\Pi$ such that $\Aut \Pi$ is compact while $\QAut \Pi$ is nonunimodular and quantum vertex transitive.

Realize the grandparent graph in the following way. Fix an integer $d \geq 3$ and let $B_0$ be the $d$-regular tree. Choose an infinite geodesic path $b_\infty = b_0 b_1 b_2 \cdots$ in $B_0$, which just means that we fix any infinite path without backtracking. This defines an orientation $\om$ on $E(B_0)$, where $\om(v,w) = 1$ if the edge from $v$ to $w$ points in the direction of $b_\infty$ and $\om(v,w) = -1$ otherwise. We define the graph $B$ in the following way: we keep the vertex set $V(B) = V(B_0)$ and we define the set of edges $E(B)$ by adding to $E(B_0)$ an edge between $v,w \in V(B_0)$ if there exists an $x \in V(B_0)$ such that $v \sim x$ and $x \sim w$ in $B_0$ and $\om(v,x) = \om(x,w)$. Note that every $v \in V(B_0)$ has exactly one neighbor $w \in V(B_0)$ with $\om(v,w) = 1$ and $d-1$ neighbors $w \in V(B_0)$ with $\om(v,w) = -1$. It then follows that $B$ is a $(d^2-d+2)$-regular graph.

Fix a base vertex $e \in V(B)$. Then define $\Pi$ as the product graph given by $\cG_v = \cG$ for all $v \in V(B) \setminus \{e\}$ and $\cG_e = \cG'$.

Since the group of automorphisms of $B$ fixing $e$ is compact, it follows from Lemma \ref{lem.product-graph} that $\Aut \Pi$ is compact. Defining $\Pi'$ as the product graph given by $\cG_v = \cG$ for all $v \in V(B)$, we find that $\Pi'$ is quantum isomorphic with $\Pi$. Since $\Pi'$ is vertex transitive, a fortiori $\Pi'$ is quantum vertex transitive. By Theorem \ref{thm.quantum-iso-criteria}, also $\Pi$ is quantum vertex transitive. In particular, by Proposition \ref{prop.criterion-compact}, $\QAut \Pi$ is noncompact.

Since $\Pi$ is quantum vertex transitive, to prove that $\QAut \Pi$ is nonunimodular, it suffices to prove that the unitary tensor category $\cC(\cP,\Pi)$ admits an irreducible object with different left and right dimension. Since $\Pi$ and $\Pi'$ are quantum isomorphic, by Theorem \ref{thm.quantum-iso-criteria}, it thus suffices to provide a single endomorphism $T$ of an object in $\cC(\cP,\Pi')$ such that the left and right trace of $T$ are different.

Denote $I = V(\Pi')$ and consider $\pi : I \to V(B)$ where $\pi(i) = v$ if $i \in \cG_v$. Denote by $\Mor(1,1)$ the linear span of all diagonal matrices $T^\cK_{\Pi'}$, where $\cK = (K,x_0x_1,x_0x_1) \in \cP_c(2,2)$. Recall that $\Mor(1,1)$ is a $*$-algebra and note that we can view $\Mor(1,1)$ as well as the linear span of all functions $(i,j) \mapsto (T^\cK_{\Pi'})_{ij}$, where $\cK \in \cP_c(1,1)$. The adjacency function $A(i,j) = 1$ for $i \sim_{\Pi'} j$ and $A(i,j) = 0$ for $i \not\sim_{\Pi'} j$ belongs to $\Mor(1,1)$. By Lemma \ref{lem.product-graph}, also the function $S(i,j) = 1$ if $\pi(i) = \pi(j)$ and $S(i,j) = 0$ if $\pi(i) \neq \pi(j)$ belongs to $\Mor(1,1)$. Thus $W := (1-S)A$ belongs to $\Mor(1,1)$. By construction, $W(i,j) = 1$ if $\pi(i) \sim_B \pi(j)$ and $W(i,j) = 0$ otherwise.

Next define $\cK = (K,x,y) \in \cP_c(1,1)$ by $V(K) = \{0,1,2,3\}$ and edges $0 \sim 1 \sim 3 \sim 2 \sim 0$ forming a square, together with the diagonal edge $1 \sim 2$, and $x_1 = 0$, $y_1 = 1$, as illustrated in Figure \ref{fig.magic-square}.
\begin{figure}[h]
    \centering
    \begin{tikzpicture}
        \filldraw (0,0) circle (1.5pt);
        \filldraw (0,1.4) circle (1.5pt) node[anchor=south]{$x_1$};
        \filldraw (1.4,1.4) circle (1.5pt) node[anchor=south]{$y_1$};
        \filldraw (1.4,0) circle (1.5pt);
        \draw (0,0) -- (0,1.4) -- (1.4,1.4) -- (1.4,0) -- (0,0);
        \draw (0,0) -- (1.4,1.4);
    \end{tikzpicture}
    \caption{The bi-labeled graph $\cK \in \cP_c(1,1)$}\label{fig.magic-square}
\end{figure}

Choosing any orientation on $K$ and assigning the function $W \in \Mor(1,1)$ to every positive edge of $K$, it follows from Lemma \ref{lem.more-functions} below that the function
$$I \times I \to \R : (i,j) \mapsto (T^\cK_B)(\pi(i),\pi(j)) \quad\text{belongs to $\Mor(1,1)$.}$$
Write $R = T^\cK_B$. Among the edges in $B$, we distinguish the short edges belonging to $E(B_0)$ and the long edges that were added when defining $B$. One then checks that
$$R(v,w) = \begin{cases} 0 &\quad\text{if $v \not\sim_B w$,}\\
2d-1 &\quad\text{if $(v,w)$ is a positive short edge in $B$,}\\
d(d-1)+1 &\quad\text{if $(v,w)$ is a negative short edge in $B$,}\\
d &\quad\text{if $(v,w)$ is a long edge in $B$.}\end{cases}$$
Since $d \geq 3$, all these values are distinct. Define $P : V(B) \to \{0,1\}$ by $P(v,w) = 1$ if $(v,w)$ is a positive short edge in $B$ and $P(v,w) = 0$ otherwise. Since $\Mor(1,1)$ is a $*$-algebra, it follows that $P' : (i,j) \mapsto P(\pi(i),\pi(j))$ belongs to $\Mor(1,1)$. For every $i$, there are precisely $\kappa$ elements $j$ such that $P'(i,j) = 1$. For every $j$, there are precisely $(d-1)\kappa$ elements $i$ such that $P'(i,j) = 1$. Therefore, the left trace of $P'$ equals $\kappa$ and its right trace equals $(d-1)\kappa$. It thus follows that $\QAut(\Pi')$ and $\QAut(\Pi)$ are nonunimodular.
\end{example}

\begin{example}\label{ex.nonunimodular-unimodular}
We include a final example of a connected locally finite graph $\Pi$ such that $\QAut \Pi$ is quantum vertex transitive and unimodular, while $\Aut \Pi$ is nonunimodular.

Let $d \geq 3$ be an integer and let $B$ be the $d$-regular tree. Fix an infinite geodesic path $b_\infty = b_0 b_1 \cdots$ in $B$. As in Example \ref{ex.compact-nonunimodular}, consider the associated orientation $\om : E(B) \to \{\pm 1\}$. We denote by $\Aut_\om(B)$ the group of orientation preserving automorphisms of $B$. Note that $\Aut_\om(B)$ is nonunimodular (see e.g.\ \cite[Section 8.2]{LP16}).

We associate to $b_\infty$ in a unique way a coloring $\theta : V(B) \to \Z/3\Z$ of the vertices such that $\theta(b_k) = k + 3 \Z$ for all $k \geq 0$, and such that for any infinite geodesic path $c_0 c_1 \cdots c_s b_n b_{n+1} \cdots$ that joins the path $b_\infty$, we have $\theta(c_k) = k-s-1+n + 3 \Z$. We denote by $\Aut_\theta(B)$ the group of automorphisms of $B$ that leave $\theta$ invariant. Note that $\Aut_\theta(B) \subset \Aut_\om(B)$. Also note that we have a well-defined continuous homomorphism $\zeta : \Aut_\om(B) \to \Z/3\Z : \theta(\si(v)) = \zeta(\si) + \theta(v)$ for all $\si \in \Aut_\om(B)$ and $v \in V(B)$. By construction, $\Aut_\theta(B)$ is the kernel of $\zeta$, so that also $\Aut_\theta(B)$ is nonunimodular.

Defining $\cG_0 = \cG \sqcup \cG$, $\cG_1 = \cG \sqcup \cG'$ and $\cG_2 = \cG' \sqcup \cG'$, we have three nonisomorphic, but quantum isomorphic finite graphs. We define $\Pi$ as the product graph of $\cG_{\theta(v)}$, $v \in V(B)$. Lemma \ref{lem.product-graph} provides a continuous group homomorphism $\psi: \Aut \Pi \to \Aut_\theta(B)$ with $\Ker \psi$ being compact. Every $\sigma \in \Aut_\theta(B)$ gives rise to a canonical $\si' \in \Aut \Pi$ with $\psi(\si') = \si$. So, $\Aut \Pi$ is the semidirect product of $\Aut_\theta(B)$ and the compact group $\Ker \psi$. We conclude that $\Aut \Pi$ is nonunimodular.

We prove that $\QAut \Pi$ is quantum vertex transitive and unimodular. Defining $\Pi'$ as the product graph of copies of $\cG \sqcup \cG$ for all $v \in V(B)$, we find that $\Pi$ and $\Pi'$ are quantum isomorphic. Since $\Pi'$ is ignoring the coloring $\theta$ and since the graph $\cG \sqcup \cG$ is vertex transitive, by Lemma \ref{lem.product-graph}, the graph $\Pi'$ is vertex transitive and its automorphism group is the semidirect product of $\Aut B$ and a compact group. So, $\Aut \Pi'$ is unimodular. By Proposition \ref{prop.unimodular}, $\QAut \Pi'$ is quantum vertex transitive and unimodular. Since $\Pi$ and $\Pi'$ are quantum isomorphic, it follows from Theorem \ref{thm.quantum-iso-criteria} that also $\QAut \Pi$ is quantum vertex transitive and unimodular.
\end{example}

For the following lemma, we fix any connected locally finite graph $\Pi$ with vertex set $I$ and we denote by $\Mor(1,1)$ the linear span of all the functions $T^\cK : I \times I \to \C$ defined by $\cK = (K,x_0x_1,x_0x_1) \in \cP_c(2,2)$, or equivalently by $\cK = (K,x_1,x_2) \in \cP_c(1,1)$.

\begin{lemma}\label{lem.more-functions}
Let $(K,x_1,x_2) \in \cP_c(1,1)$. Assume that $K$ has no loops and is oriented by $E(K) = E_+(K) \sqcup E_-(K)$. Assume that for every $e \in E_+(K)$, we are given a function $T_e \in \Mor(1,1)$. Then also the function
$$S : I \times I \to \C : S(i,j) = \sum_{\vphi : V(K) \to I : \vphi(x_1) = i, \vphi(x_2) = j} \; \prod_{(v,w) \in E_+(K)} \; T_{(v,w)}(\vphi(v),\vphi(w))$$
belongs to $\Mor(1,1)$.
\end{lemma}
\begin{proof}
By linearity, we may assume that we are given, for every $e \in E_+(K)$, a planar bi-labeled graph $\cK_e = (K_e,x_e,y_e) \in \cP_c(1,1)$ such that $T_e(i,j) = T^{\cK_e}_{ij}$. We define the planar graph $K'$ by replacing every edge $(v,w) \in E_+(K)$ by $K_e$, identifying $v = x_e$ and $w = y_e$. We put $\cK' = (K',x_1,x_2)$. Then $\cK' = \cP_c(1,1)$ and, by construction, the corresponding function $T^{\cK'}$ equals $S$.
\end{proof}

\section{\boldmath The unitary $2$-category $\cC(\cP,\Pi)$ as category of $\QAut \Pi$-equi\-variant $\ell^\infty(I)$-bimodules}

Consider a continuous action $G \actson I$ of a locally compact second countable group $G$ on a countable set $I$. As explained in the beginning of Section \ref{sec.2-category} and following \cite{AV16}, we can define the following unitary $2$-category. The $0$-cells are given by the set $\cE$ of orbits $(I_a)_{a \in \cE}$ for the action $G \actson I$. Then $\cC_{a-b}(G \actson I)$ is defined as the category of $G$-equivariant Hilbert $\ell^\infty(I_a)$-$\ell^\infty(I_b)$-bimodules of finite type, with the relative tensor product defined by \eqref{eq.relative-tensor-product-Hilbert}.

Given a connected locally finite graph $\Pi$ with vertex set $I$ and given a graph category $\cD$ of bi-labeled graphs containing the planar bi-labeled graphs, we defined in Definition \ref{def.my-unitary-2-category} a unitary $2$-category $\cC(\cD,\Pi)$. We prove below that when $\cD = \cG$, the graph category of all bi-labeled graphs, then $\cC(\cG,\Pi)$ is isomorphic with $\cC(\Aut \Pi \actson I)$. More generally, given a faithful, continuous, integrable action $\bG \actson I$ of a locally compact quantum group $\bG$ on a countable set $I$, we define the unitary $2$-category $\cC(\bG \actson I)$ of $\bG$-equivariant Hilbert $\ell^\infty(I)$-bimodules, by quantizing the definitions of \cite{AV16}. We then prove in Proposition \ref{prop.link-with-equivariant-bimodules} that when $\bG$ is the quantum automorphism group associated with $\Pi$ and $\cD$, as defined in Theorem \ref{thm.left-Haar}, then $\cC(\bG \actson I) \cong \cC(\cD,\Pi)$. In combination with Proposition \ref{prop.link-with-classical-case}, we then find in particular that $\cC(\Aut \Pi \actson I) \cong \cC(\cG,\Pi)$.

\subsection{\boldmath The unitary $2$-category of $\bG$-equivariant $\ell^\infty(I)$-bimodules}\label{sec.2-category-equivariant-bimodules}

Let $\bG$ be a locally compact quantum group, in the sense of \cite{KV00}, defined by a pair $(M,\Delta)$ where $M$ is a von Neumann algebra and $\Delta : M \to M \ovt M$. Assume that $M$ is generated by the entries of a magic unitary $U \in M \ovt B(\ell^2(I))$ such that $(\Delta \ot \id)(U) = U_{13} U_{23}$. Recall that $U = (u_{ij})_{i,j \in I}$ is called a magic unitary of the rows and columns of $U$ are self-adjoint projections that sum up to $1$ in $M$. Then,
\begin{align*}
& \al : \ell^\infty(I) \to M \ovt \ell^\infty(I) : \al(p_i) = \sum_{j \in I} (u_{ij} \ot p_j) \quad\text{and}\\
& \be : \ell^\infty(I) \to M \ovt \ell^\infty(I) : \be(p_i) = \sum_{j \in I} (u_{ji} \ot p_j)
\end{align*}
are coactions of $(M,\Delta)$, resp.\ $(M,\Delta\op)$, on $\ell^\infty(I)$.

Denote by $S$ the antipode of $\bG$. Since $S(u_{ij}) = u_{ji}$, we find that $S^2 = \id$ and that $S$ is a normal $*$-anti-automorphism of $M$. We further assume that the left Haar weight $\vphi$ and the right Haar weight $\psi$ of $\bG$ are tracial and that $\vphi(u_{ij}) < +\infty$ and $\psi(u_{ij}) < +\infty$ for all $i,j \in I$. Note that these conditions precisely mean that $\al$ and $\be$ are integrable in the sense of \cite[Definition 1.4]{Vae00}.

Note that all assumptions that we made are satisfied by the quantum automorphism groups that we constructed in Section \ref{sec.quantum-aut-group}.

\begin{definition}
We call a pair $(\cH,V)$ a $\bG$-equivariant Hilbert $\ell^\infty(I)$-bimodule if $\cH$ is a Hilbert $\ell^\infty(I)$-bimodule and $V \in M \ovt B(\cH)$ is a unitary corepresentation of $(M,\Delta)$ such that
$$V (1 \ot \lambda(F)) V^* = (\id \ot \lambda)\be(F) \quad\text{and}\quad V^*(1 \ot \rho(F))V = (\id \ot \rho)\al(F) \quad\text{for all $F \in \ell^\infty(I)$,}$$
where $\lambda : \ell^\infty(I) \to B(\cH)$, resp.\ $\rho : \ell^\infty(I) \to B(\cH)$, denote the left, resp.\ right, module action.

We say that $(\cH,V)$ is of finite type if $p_i \cdot \cH$ and $\cH \cdot p_i$ are finite dimensional for all $i \in I$.
\end{definition}

Thanks to the following lemma, we can define the tensor product of two $\bG$-equivariant Hilbert $\ell^\infty(I)$-bimodules.

\begin{lemma}
Let $(\cH,X)$ and $(\cK,Y)$ be $\bG$-equivariant Hilbert $\ell^\infty(I)$-bimodules. Denote by $P$ the orthogonal projection of $\cH \ot \cK$ onto the closed subspace $\cH \ot_I \cK$. Then, $X_{12} Y_{13} (1 \ot P) = (1 \ot P) X_{12} Y_{13}$. Writing $Z = X_{12} Y_{13} (1 \ot P)$, we get that $(\cH \ot_I \cK,Z)$ is a $\bG$-equivariant Hilbert $\ell^\infty(I)$-bimodule.
\end{lemma}
\begin{proof}
Note that $P = \sum_{i \in I} (\rho_\cH(p_i) \ot \lambda_\cK(p_i))$. Thus,
$$X_{12} Y_{13} (1 \ot P) = \sum_{i,j \in I} X_{12} (u_{ji} \ot \rho_\cH(p_i) \ot \lambda_\cK(p_j)) Y_{13} = (1 \ot P) X_{12} Y_{13} \; .$$
\end{proof}

The space of morphisms from  $(\cH,X)$ to $(\cK,Y)$ is defined as the set of all bounded $\ell^\infty(I)$-bimodular operators $T : \cH \to \cK$ satisfying $Y(1 \ot T) = (1 \ot T)X$. The identity object is defined as the Hilbert space $\ell^2(I)$ with the obvious $\ell^\infty(I)$-bimodule structure and unitary corepresentation $U \in M \ovt B(\ell^2(I))$. One checks that
$$\xi \mapsto \sum_{i \in I} \xi \cdot p_i \ot e_i$$
defines a unitary equivalence between $(\cH,X)$ and the tensor product $(\cH,X) \ot_I (\ell^2(I),U)$. A similar result holds for $(\ell^2(I),U) \ot_I (\cH,X)$, so that $(\ell^2(I),U)$ indeed is an identity object. We have thus defined the C$^*$-tensor category of $\bG$-equivariant $\ell^\infty(I)$-bimodules.

Note however that the identity object $(\ell^2(I),U)$ need not be irreducible. We thus introduce the following notation, which by Corollary \ref{cor.about-rel-E}, is compatible with the notation that we used for quantum automorphism groups. Note that for actions of compact quantum groups, this equivalence relation $\approx$ was already defined in \cite[Definition 4.1]{DCKSS16}.

\begin{lemma}\label{lem.approx-eq-rel}
The relation on $I$ defined by $i \approx j$ if $u_{ij} \neq 0$ is an equivalence relation. Given a subset $W \subset I$, we have $\al(1_W) = 1 \ot 1_W$ iff $\be(1_W) = 1 \ot 1_W$ iff $W$ is invariant under $\approx$.
\end{lemma}
\begin{proof}
Since $S(u_{ij}) = u_{ji}$, the relation $\approx$ is symmetric. We have $(u_{ik} \ot 1)\Delta(u_{ij}) = u_{ik} \ot u_{kj}$ for all $i,j,k \in I$. Assume that $i \in I$ and $u_{ii} = 0$. It then follows that $u_{ik} \ot u_{ki} = 0$ for all $k \in I$. Since $u_{ik} = S(u_{ki})$, this means that $u_{ki} = 0$ for all $k \in I$. Summing over $k$, we find the contradiction that $1 = 0$. So the relation $\approx$ is reflexive. Finally, if $i \approx k$ and $k \approx j$, then $(u_{ik} \ot 1)\Delta(u_{ij}) = u_{ik} \ot u_{kj}$ is nonzero. So $u_{ij} \neq 0$ and $i \approx j$.

Given $F \in \ell^\infty(I)$, we have that $\al(F) = 1 \ot F$ iff $\be(F) = 1 \ot F$ iff $F(i) = F(j)$ for all $i,j \in I$ with $u_{ij} \neq 0$.
\end{proof}

We denote by $(I_a)_{a \in \cE}$ the set of equivalence classes for $\approx$. We denote by $1_a \in \ell^\infty(I)$ the indicator function of $I_a$. Note that by Lemma \ref{lem.approx-eq-rel}, we have that $U(1 \ot 1_a) = (1 \ot 1_a)U$.

\begin{definition}\label{def.unitary-2-category-of-equivariant-bimodules}
We denote by $\cC(\bG \actson I)$ the unitary $2$-category with $0$-cells $\cE$ and with $1$-cells $\cC_{a-b}(\bG \actson I)$ given by the $\bG$-equivariant Hilbert $\ell^\infty(I_a)$-$\ell^\infty(I_b)$-bimodules of finite type.
\end{definition}

We denote the identity object in $\cC_{a-a}(\bG \actson I)$ by $\eps_a = (\ell^2(I_a),U (1 \ot 1_a))$. To check that Definition \ref{def.unitary-2-category-of-equivariant-bimodules} indeed provides a unitary $2$-category, the only nontrivial point is that every $(\cH,V) \in \cC_{a-b}(\bG \actson I)$ admits a conjugate. We consider the $\ell^\infty(I_b)$-$\ell^\infty(I_a)$-bimodule $\overline{\cH}$. Since $S^2 = \id$, we have the conjugate unitary corepresentation $\overline{V} \in M \ovt B(\overline{\cH})$ satisfying $(\id \ot \om_{\overline{\xi},\overline{\eta}})(\overline{V}) = (\id \ot \om_{\eta,\xi})(V^*)$ for all $\xi,\eta \in \cH$. Denoting by $\onb(\cK)$ any choice of orthonormal basis of a Hilbert space, one checks that the maps
\begin{align*}
& s : \ell^2(I_a) \to \cH \ot_{I_b} \overline{\cH} : s(e_i) = \sum_{j \in I_b, \xi \in \onb(p_i \cdot \cH \cdot p_j)} (\xi \ot \overline{\xi}) \quad\text{and}\\
& t : \ell^2(I_a) \to \overline{\cH} \ot_{I_b} \cH : s(e_i) = \sum_{j \in I_b, \xi \in \onb(p_j \cdot \cH \cdot p_i)} (\overline{\xi} \ot \xi)
\end{align*}
solve the conjugate equations for $(\cH,V)$ and its conjugate $(\overline{\cH},\overline{V})$.

In the theory of compact quantum groups, one obtains intertwiners between unitary representations by integration over the compact quantum group. In particular, for irreducible representations, this leads to the Peter-Weyl orthogonality relations. The following lemma provides a similar result in the context of $\bG$-equivariant $\ell^\infty(I)$-bimodules. It will be a key tool in Section \ref{sec.isomorphism-with-equiv-cat}.

\begin{lemma}\label{lem.integrate-intertwiner}
Let $(\cH,X)$ and $(\cK,Y)$ be $\bG$-equivariant $\ell^\infty(I)$-bimodules. Fix $a \in \cE$ and $e \in I_a$. Denote by $\cN \subset B(\cH,\cK)$ the space of bounded $\ell^\infty(I)$-bimodular operators from $\cH$ to $\cK$.
\begin{enumlist}
\item There is a $C > 0$ such that $\vphi(u_{ie}) = C$ for all $i \in I_a$.
\item For every $T_0 \in \cN$, there is a unique morphism $\Phi(T_0)$ from $(\cH,X)$ to $(\cK,Y)$ satisfying
\begin{equation}\label{eq.formula-intertwiner}
\lambda_\cK(p_i) \Phi(T_0) = C^{-1} \, (\vphi \ot \id)\bigl((u_{ie} \ot 1)Y (1 \ot \lambda_\cK(p_e) T_0) X^*\bigr) \quad\text{for all $i \in I$.}
\end{equation}
\item If $T_0$ is a morphism from $(\cH,X)$ to $(\cK,Y)$ such that $T_0 = \lambda_\cK(1_a) T_0$, then $\Phi(T_0) = T_0$.
\end{enumlist}
\end{lemma}
\begin{proof}
1.\ For all $i,k \in I$, we have $(u_{ik} \ot 1)\Delta(u_{ie}) = u_{ik} \ot u_{ke}$. By left invariance of $\vphi$, we conclude that $\vphi(u_{ie}) = \vphi(u_{ke})$ whenever $i \approx k$. So we find $C > 0$ such that $\vphi(u_{ie}) = C$ for all $i \in I_a$.

2.\ Fix $T_0 \in \cN$. Since $(u_{ie} \ot 1) (\id \ot \lambda_\cK)\be(p_e) = u_{ie} \ot \lambda_\cK(p_i)$, the right hand side of \eqref{eq.formula-intertwiner} defines an element $T_i \in B(\cH,p_i \cdot \cK)$. We also have that $(\id \ot \lambda_\cH)\be(p_e) (u_{ie} \ot 1) = u_{ie} \ot \lambda_\cH(p_i)$, that $u_{ie}^2 = u_{ie}$ and that $\vphi$ is a trace. Therefore,
\begin{align*}
C \, T_i & = (\vphi \ot \id)\bigl( (u_{ie} \ot 1)Y (1 \ot \lambda_{\cK}(p_e) T_0\lambda_{\cH}(p_e)) X^*\bigr)\\
& = (\vphi \ot \id)\bigl( (u_{ie} \ot 1)Y (1 \ot \lambda_{\cK}(p_e) T_0) X^* (\id \ot \lambda_\cH)\be(p_e) (u_{ie} \ot 1)\bigr) \\
& = (\vphi \ot \id)\bigl( (u_{ie} \ot 1)Y (1 \ot \lambda_{\cK}(p_e) T_0) X^* (u_{ie} \ot \lambda_{\cH}(p_i))\bigr) = C\, T_i \, \lambda_\cH(p_i) \; .
\end{align*}
Since moreover $\|T_i\| \leq \|T_0\|$ for all $i \in I$, the operator $T = \sum_{i \in I} T_i$ is a well-defined bounded operator satisfying $\lambda_\cK(F) T = T \lambda_\cH(F)$ for all $F \in \ell^\infty(I)$.

Also note that
$$C \, \rho_\cK(F) \, T_i = (\vphi \ot \id)\bigl((u_{ie} \ot 1)Y (\id \ot \rho_\cK)\al(F) (1 \ot \lambda_\cK(p_e) T_0) X^*\bigr) = C \, T_i \, \rho_\cH(F)$$
so that $\rho_\cK(F) \, T = T \, \rho_\cH(F)$ for all $F \in \ell^\infty(I)$.

We finally need to prove that $(1 \ot T) X = Y (1 \ot T)$. By the left invariance of $\vphi$, we get for all $i,k \in I$,
\begin{align*}
C \, u_{ik} \ot T_i & = (\id \ot \vphi \ot \id)\bigl((u_{ik} \ot 1 \ot 1) (\Delta \ot \id)\bigl((u_{ie} \ot 1)Y (1 \ot \lambda_{\cK}(p_e) T_0) X^*\bigr)\bigr) \\
& = (\id \ot \vphi \ot \id)\bigl( (u_{ik} \ot u_{ke} \ot 1) Y_{13} Y_{23} (1 \ot 1 \ot \lambda_{\cK}(p_e) T_0) X^*_{23} X^*_{13} \bigr) \\
& = (u_{ik} \ot 1) Y \bigl( 1 \ot (\vphi \ot \id)\bigl((u_{ke} \ot 1) Y (1 \ot \lambda_{\cK}(p_e) T_0) X^* \bigr)\bigr) X^* \\
& = C \, (u_{ik} \ot 1) Y ( 1 \ot \lambda_\cK(p_k) T ) X^* = C \, (u_{ik} \ot \lambda_\cK(p_i)) Y (1 \ot T) X^* \; .
\end{align*}
First summing over $k$ and then summing over $i$, we find that $1 \ot T = Y (1 \ot T) X^*$.

3.\ If $T_0$ is a morphism from $(\cH,X)$ to $(\cK,Y)$, we find that $\lambda_\cK(p_i) \Phi(T_0) = C^{-1} \, \vphi(u_{ie}) \, \lambda_\cK(p_i) T_0$ for all $i \in I$. Since $u_{ie} = 0$ if $i \not\in I_a$, we conclude that $\Phi(T_0) = \lambda_\cK(1_a) T_0$.
\end{proof}

\subsection{\boldmath Isomorphism between $\cC(\bG \actson I)$ and $\cC(\cD,\Pi)$}\label{sec.isomorphism-with-equiv-cat}

We fix a connected locally finite graph $\Pi$ with vertex set $I$ and a graph category $\cD$ containing all planar bi-labeled graphs. We denote by $\bG$ the associated quantum automorphism group defined in Theorem \ref{thm.left-Haar}, with von Neumann algebra realization $(M,\Delta)$, generated by the entries of a magic unitary $U = (u_{ij})_{i,j \in I}$.

\begin{proposition}\label{prop.link-with-equivariant-bimodules}
The unitary $2$-categories $\cC(\cD,\Pi)$ and $\cC(\bG \actson I)$ defined in \ref{def.my-unitary-2-category} and \ref{def.unitary-2-category-of-equivariant-bimodules} are equivalent.
\end{proposition}
\begin{proof}
We denote by $(\cA,\Delta)$ the algebraic quantum group underlying $\bG$.
For every $n \geq 0$, define the $\bG$-equivariant $\ell^\infty(I)$-bimodule $(\cH^{(n)},U^{(n)})$ by $\cH^{(n)} = \ell^2(I^{n+1})$ and $U^{(n)} = U_{12} U_{13} \cdots U_{1,n+2}$. By the defining relations for $\cA$, we have that $U^{(n)} (1 \ot T^\cK) = (1 \ot T^\cK) U^{(m)}$ for all $\cK \in \cD_c(n,m)$. We also have that $U(1 \ot 1_a) = (1 \ot 1_a) U$ for all $a \in \cE$. Thus, every $T \in \Mor_{a-b}(n,m)$ defines a morphism from $(\cH^{(m)},U^{(m)})$ to $(\cH^{(n)},U^{(n)})$. In particular, every orthogonal projection $P \in \Mor_{a-b}(n,n)$ defines an object in $\cC_{a-b}(\bG \actson I)$.

To conclude the proof of the proposition, we thus only have to prove the following two statements. We denote by $\cM(n,m)$ the space of all bounded, $\bG$-equivariant, $\ell^\infty(I)$-bimodular operators from $(\cH^{(m)},U^{(m)})$ to $(\cH^{(n)},U^{(n)})$.

\begin{enumlist}
\item For all orthogonal projections $P \in \Mor_{a-b}(n,n)$ and $Q \in \Mor_{a-b}(m,m)$, we have that
\begin{equation}\label{eq.other-inclusion}
P \, \cM(n,m) \, Q \subset \Mor_{a-b}(n,m) \; .
\end{equation}

\item For every $\bG$-equivariant $\ell^\infty(I)$-bimodule $(\cH,X)$, there exists an $n \geq 0$, $a,b \in \cE$, an orthogonal projection $P \in \Mor_{a-b}(n,n)$ and a morphism $T$ from $(\cH,X)$ to $(\cH^{(n)},U^{(n)})$ such that $P T \neq 0$.
\end{enumlist}

Proof of 1.\ For $i \in I^{n+1}$, we again denote by $(i)$ the canonical basis vector in $\ell^2(I^{n+1})$. For all $i \in I^{n+1}$ and $j \in I^{m+1}$, we denote by $E_{ij} \in B(\ell^2(I^{m+1}),\ell^2(I^{n+1}))$ the obvious rank $1$ operator. Fix a base vertex $e \in I_a$ and denote by $\vphi_e$ the corresponding Haar functional defined in \ref{def.left-Haar}. We use the notation of Lemma \ref{lem.integrate-intertwiner} and start by proving that
\begin{equation}\label{eq.nice-on-dense}
\Phi(E_{s i k,s j k}) \in \Mor_{a-b}(n,m) \quad\text{for all $s \in I_a$, $k \in I_b$, $i \in I^{n-1}$ and $j \in I^{m-1}$.}
\end{equation}
If $s \neq e$, then $\Phi(E_{s i k,s j k})=0$ by definition. So we may assume that $s = e$. Whenever $r \in I^{n+1}$ and $l \in I^{m+1}$ with $r_0 = l_0 \in I_a$ and $r_n = l_m \in I_b$, the matrix entry $(\Phi(E_{e i k,e j k}))_{rl}$ is given by
\begin{align*}
\vphi_e\bigl(u_{r_0 e} U_{n+1}(r,eik) U_{m+1}(l,ejk)^*\bigr) & = \vphi_e\bigl(u_{r_0 e} U_{n+1}(r,eik) U_{m+1}(l_m \cdots l_0,k j_{m-1} \cdots j_1 e)\bigr) \\
& = \vphi_e\bigl(U_{n+m+1}(r l_{m-1} \cdots l_1 r_0, e i k j_{m-1} \cdots j_1 e)\bigr) \; .
\end{align*}
By \eqref{eq.formula-left-Haar}, this last expression equals
\begin{equation}\label{eq.stapje}
\sum_{V \in \ibf_{a-a}(n+m,\eps_a)} \langle V(r_0) , (r l_{m-1} \cdots l_1 r_0) \rangle \, \langle (e i k j_{m-1} \cdots j_1 e), V(e) \rangle \; .
\end{equation}
Note that since $e,i,k,j$ are fixed, there are only finitely many nonzero terms in the above sum. Also, by Lemma \ref{lem.crucial-technical-Mor-lemma}, for every $V \in \Mor_{a-a}(n+m,\eps_a)$, there is a unique $Z_V \in \Mor_{a-b}(n,m)$ satisfying
$$\langle V(r_0) , (r l_{m-1} \cdots l_1 r_0)\rangle = \langle Z_V(r_0 l_1 \cdots l_{m-1} r_n), (r) \rangle$$
for all $r \in I^{n+1}$ with $r_n \in I_b$, $l_1,\ldots,l_{m-1} \in I$. But then the expression in \eqref{eq.stapje} equals the $rl$ matrix entry of the finite sum
$$\sum_{V \in \ibf_{a-a}(n+m,\eps_a)}  \langle (e i k j_{m-1} \cdots j_1 e), V(e) \rangle \; Z_V \; .$$
We have thus proven \eqref{eq.nice-on-dense}.

To prove \eqref{eq.other-inclusion}, since $P \, \Mor_{a-b}(n,m) \, Q$ is a finite dimensional vector space, it suffices to prove that any $T_0 \in P \, \cM(n,m) \, Q$ belongs to the $\sigma$-weak closure of $P \, \Mor_{a-b}(n,m) \, Q$. Denote by $\cN_0$ the linear span of the rank $1$ operators $E_{s i k,s j k}$ with $s \in I_a$, $k \in I_b$, $i \in I^{n-1}$ and $j \in I^{m-1}$. Choose a sequence $T_l \in \cN_0$ such that $T_l \to T_0$ $\sigma$-weakly. Then, $\Phi(T_l) \to \Phi(T_0)$ $\sigma$-weakly. By Lemma \ref{lem.integrate-intertwiner}, we have $\Phi(T_0) = T_0$. So, $P \, \Phi(T_l) \, Q \to T_0$ $\sigma$-weakly. By \eqref{eq.nice-on-dense}, all $P \, \Phi(T_l) \, Q$ belong to $P \, \Mor_{a-b}(n,m) \, Q$. This concludes the proof of 1.

Proof of 2.\ We apply Lemma \ref{lem.integrate-intertwiner} to $(\cK,Y) = (\cH^{(n)},U^{(n)})$ and $T_0 \in B(\cH,\cK)$ an arbitrary bounded $\ell^\infty(I)$-bimodular operator. If for all choices of $n \geq 1$, $T$ and $e \in I$, the expression in \eqref{eq.formula-intertwiner} equals zero for all $i \in I$, it follows that
\begin{align*}
\vphi\bigl(u_{ie} \, U_{n+1}(isr,klt) \, a \bigr) = 0 \quad &\text{when $a = (\id \ot \om_{\xi,\eta})(X^*)$, for all $i,e,k,r,t \in I$, $s,l \in I^{n-1}$,}\\
& \text{$n \geq 1$, $\xi \in \cH$, $\eta \in p_k \cdot \cH \cdot p_t$.}
\end{align*}
Taking $k = e$, so that $u_{ie} \, U_{n+1}(isr,klt) = U_{n+1}(isr,elt)$, summing over $i,r \in I$ and only considering $n \geq 2$, we find that
$$\vphi\bigl( U_{n-1}(s,l) \, (\id \ot \om_{\xi,\eta})(X^*) \bigr) = 0 \quad\text{for all $n \geq 1$, $s,l \in I^n$ and $\xi,\eta \in \cH$.}$$
Because the elements $U_{n-1}(s,l)$ are dense in $M$ and $\vphi$ is faithful, it follows that $X = 0$, which is absurd. So, we find $n \geq 0$ and a nonzero morphism $T$ from $(\cH,X)$ to $(\cH^{(n)},U^{(n)})$. Pick $a,b \in \cE$ such that $(1_a \ot 1^{\ot (n-1)} \ot 1_b) T$ is still nonzero. Since $\Mor_{a-b}(n,n)$ contains an increasing sequence of orthogonal projections converging strongly to $1$, we can take an orthogonal projection $P \in \Mor_{a-b}(n,n)$ such that $P T \neq 0$.
\end{proof}

Recall from Section \ref{sec.bi-labeled-graphs} that $\cG$ denotes the category of all bi-labeled graphs and $\cP$ the category of planar bi-labeled graphs. Given a connected locally finite graph $\Pi$, the associated unitary $2$-categories $\cC(\cP,\Pi)$ and $\cC(\cG,\Pi)$ were introduced in Definition \ref{def.my-unitary-2-category}.

\begin{corollary}\label{cor.characterize-no-quantum-symmetry}
Let $\Pi$ be a connected locally finite graph. Then $\Pi$ has no quantum symmetry, meaning that $\QAut \Pi = \Aut \Pi$, if and only if the natural functor $\cC(\cP,\Pi) \to \cC(\cG,\Pi)$ is an equivalence of $2$-categories.
\end{corollary}
\begin{proof}
If $\QAut \Pi = \Aut \Pi$, then it follows from Proposition \ref{prop.link-with-equivariant-bimodules} that the natural functor $\cC(\cP,\Pi) \to \cC(\cG,\Pi)$ is an isomorphism of $2$-categories.

Conversely, assume that this functor is an isomorphism of $2$-categories. This means that the equivalence relations $\approx$ defined by $\cP$, resp.\ $\cG$, on $I$ are the same and that $\Mor^\cP_{a-b}(n,m) = \Mor^\cG_{a-b}(n,m)$ for all $a,b \in \cE$ and $n,m \geq 0$. Define $\cK = (K,x,y) \in \cG(4,4)$ by $V(K) = \{0,1,2,3\}$, $E(K) = \emptyset$, $x_i = i$ for all $i \in \{0,1,2,3\}$ and $y_0 = 0$, $y_1 = 2$, $y_2 = 1$, $y_3 = 3$, as illustrated in Figure \ref{fig.twist-graph}.
\begin{figure}[h]
    \centering
    \begin{tikzpicture}
        \filldraw (0,0) circle (1.5pt) node[anchor=south]{$x_0$} node[anchor=north]{$y_0$};
        \filldraw (1.4,0) circle (1.5pt) node[anchor=south]{$x_1$} node[anchor=north]{$y_2$};
        \filldraw (2.8,0) circle (1.5pt) node[anchor=south]{$x_2$} node[anchor=north]{$y_1$};
        \filldraw (4.2,0) circle (1.5pt) node[anchor=south]{$x_3$} node[anchor=north]{$y_3$};
    \end{tikzpicture}
    \caption{The bi-labeled graph $\cK \in \cG(4,4)$}\label{fig.twist-graph}
\end{figure}

Fix $i,j \in I^4$. We claim that the equality $(F_3 T^\cK)_{ij} = (T^\cK F_3)_{ij}$ holds in the $*$-algebra $\cB$ associated with $\cP$. Take the unique $a, b \in \cE$ such that $j_0 \in I_a$ and $j_3 \in I_b$. As in the proof of Lemma \ref{lem.ibf}, we can take a connected $\cK' = (K,x,x) \in \cG(4,4)$ such that $T^{\cK'}$ is a diagonal matrix and $T^{\cK'}_{jj} > 0$. Then, $\cK \circ \cK'$ is connected, so that
$$T^{\cK \circ \cK'} \, (1_a \ot 1 \ot 1 \ot 1_b) \;\;\text{and}\;\;  T^{\cK'} \, (1_a \ot 1 \ot 1 \ot 1_b) \;\;\text{belong to}\;\; \Mor^\cG_{a-b}(3,3) = \Mor^\cP_{a-b}(3,3) \; .$$
By Lemma \ref{lem.ideal-UB}.4, the following equalities hold in the space of matrices over $\cB$.
\begin{align*}
F_3 \, T^{\cK} \, T^{\cK'} \, (1_a \ot 1 \ot 1 \ot 1_b) &= F_3 \, T^{\cK \circ \cK'} \, (1_a \ot 1 \ot 1 \ot 1_b) = T^{\cK \circ \cK'} \, (1_a \ot 1 \ot 1 \ot 1_b) \, F_3
\\ & = T^\cK \, T^{\cK'} \, (1_a \ot 1 \ot 1 \ot 1_b) \, F_3 = T^\cK \, F_3 \, T^{\cK'} \, (1_a \ot 1 \ot 1 \ot 1_b) \; .
\end{align*}
Taking the $ij$ entry and dividing by $T^{\cK'}_{jj}$, the claim is proven. It follows that
$$F_3(i_0 i_1 i_2 i_3, j_0 j_1 j_2 j_3) = F_3(i_0 i_2 i_1 i_3,j_0 j_2 j_1 j_3) \quad\text{in $\cB$, for all $i,j \in I^4$.}$$
By Lemma \ref{lem.iso-A-B}, the $*$-algebra $\cA$ associated with $\cP$ is abelian. Hence, $\QAut \Pi = \Aut \Pi$.
\end{proof}

\begin{corollary}\label{cor.characterize-trivial-planar-quantization}
Let $\Gamma$ be a countable group with a finite symmetric generating set $S = S^{-1} \subset \Gamma$. Denote by $\Pi$ the associated Cayley graph and let $\bGhat$ be the planar quantization defined in \ref{def.planar-quantization-discrete-group} with canonical surjective Hopf $*$-algebra homomorphism $\pi : \cO(\bG) \to \C[\Gamma]$.

Then $\pi$ is an isomorphism if and only if $\QAut \Pi = \Gamma$, i.e.\ $\Pi$ has no quantum symmetry and $\Aut \Pi = \Gamma$.
\end{corollary}
\begin{proof}
For every $n,m \geq 0$, denote by $\cM(n,m)$ the space of all bounded $\ell^\infty(\Gamma)$-bimodular operators $T$ from $\ell^2(\Gamma^{m+1})$ to $\ell^2(\Gamma^{n+1})$ with the following property: for every $i \in \Gamma$, there exist only finitely many $j \in \Gamma^n$ and $k \in \Gamma^m$ such that $T_{ij,ik} \neq 0$ or $T_{ji,ki} \neq 0$. Define $\Mor^\Gamma(n,m)$ as the set of all $T \in \cM(n,m)$ that are $\Gamma$-equivariant. We thus get
$$\Mor^\cP(n,m) \subset \Mor^\cG(n,m) \subset \Mor^\Gamma(n,m)$$
for all $n,m \geq 0$. The $*$-homomorphism $\pi$ is an isomorphism if and only if $\Mor^\cP(n,m) = \Mor^\Gamma(n,m)$ for all $n,m \geq 0$.

By Corollary \ref{cor.characterize-no-quantum-symmetry}, we have $\Mor^\cP(n,m) = \Mor^\cG(n,m)$ for all $n,m \geq 0$ if and only if $\QAut \Pi = \Aut \Pi$. We thus only have to prove that $\Mor^\cG(n,m) = \Mor^\Gamma(n,m)$ for all $n,m \geq 0$ if and only if $\Aut \Pi = \Gamma$. Write $G = \Aut \Pi$. By Proposition \ref{prop.link-with-equivariant-bimodules} and Proposition \ref{prop.link-with-classical-case}, we have that $\Mor^\cG(n,m)$ is equal to the space $\Mor^G(n,m)$ of $G$-equivariant elements in $\cM(n,m)$. So if $G = \Gamma$, we have $\Mor^\cG(n,m) = \Mor^\Gamma(n,m)$.

Conversely, assume that $\Mor^G(n,m) = \Mor^\Gamma(n,m)$ for all $n,m \geq 0$. Define $K = \{\si \in G \mid \si(e) = e\}$. To conclude that $G = \Gamma$, we need to prove that $K = \{\id\}$. For this it suffices to prove that every operator $T_0 \in \lambda(p_e) \rho(p_e) \cM(n,m)$ is $K$-equivariant. Denote by $\theta_n : \Gamma \to \cU(\ell^2(\Gamma^{n+1}))$ the diagonal translation action. Given $T_0 \in \lambda(p_e) \rho(p_e) \cM(n,m)$, we get that
$$T := \sum_{g \in \Gamma} \theta_n(g) T_0 \theta_m(g)^*$$
belongs to $\Mor^\Gamma(n,m)$. So, $T \in \Mor^G(n,m)$. In particular, $T_0 = \lambda(p_e) \rho(p_e) T$ is $K$-equivariant.
\end{proof}

\end{document}